\documentclass[11pt]{article}
\usepackage[utf8]{inputenc} % allow utf-8 input
\usepackage[T1]{fontenc}    % use 8-bit T1 fonts

\usepackage{geometry}
\usepackage{setspace}
\usepackage{amsmath, amssymb, amsfonts, bm, mathtools, mathrsfs}
\usepackage{amsthm}
\usepackage[dvipsnames]{xcolor}
\definecolor{darkblue}{rgb}{0,0,.5}
\usepackage{graphicx}
\usepackage{subfigure}
\usepackage[colorlinks=true,allcolors=darkblue]{hyperref}       % hyperlinks
\usepackage{url}            % simple URL typesetting
\usepackage{booktabs}       % professional-quality tables
\usepackage{amsfonts}       % blackboard math symbols
\usepackage{nicefrac}       % compact symbols for 1/2, etc.
\usepackage{microtype}      % microtypography

\allowdisplaybreaks

\usepackage{float}
\usepackage{multirow}
\usepackage{footnote}
\usepackage{dsfont}
\usepackage{mathabx}

\usepackage{algorithm}
\usepackage{algorithmic}
\usepackage{nicefrac}
\usepackage{tikz}
\usepackage{overpic}

\usepackage{dsfont}
\usepackage{hyperref}
\usepackage[capitalize]{cleveref}
\usepackage{crossreftools}

\makeatletter
\newcommand*{\rom}[1]{\expandafter\@slowromancap\romannumeral #1@}
\makeatother

\DeclareMathOperator*{\argmin}{argmin}

\makeatletter
\long\def\@makecaption#1#2{
  \vskip 0.8ex
  \setbox\@tempboxa\hbox{\small {\bf #1:} #2}
  \parindent 1.5em  %% How can we use the global value of this???
  \dimen0=\hsize
  \advance\dimen0 by -3em
  \ifdim \wd\@tempboxa >\dimen0
  \hbox to \hsize{
    \parindent 0em
    \hfil 
    \parbox{\dimen0}{\def\baselinestretch{0.96}\small
      {\bf #1.} #2
    } 
    \hfil}
  \else \hbox to \hsize{\hfil \box\@tempboxa \hfil}
  \fi
}
\makeatother

\newcommand{\tr}{\operatorname{Tr}}

\usepackage{enumerate}
\usepackage{natbib}

\newtheorem{claim}{Claim}[section]
\newtheorem{lemma}[claim]{Lemma}
\newtheorem{assumption}{Assumption}%[section]
\newtheorem{definition}{Definition} %[section]
\newtheorem{theorem}{Theorem}[section]

\newtheorem{proposition}{Proposition}[section]
\newtheorem{remark}{Remark}
\newtheorem{corollary}{Corollary}[section]
\DeclareMathOperator{\Var}{Var}

\DeclareMathOperator{\Tr}{Tr}
\makeatletter
\renewcommand{\tableofcontents}{%
  \section*{\centering\contentsname}%
  \@starttoc{toc}%
}
\makeatother

\usepackage{xr}

\title{Universality of General Spiked Tensor Models}

\author{
Yanjin Xiang \thanks{School of Mathematical Sciences, Peking University; email: \texttt{2401110086@stu.pku.edu.cn}.}\and
Zhihua Zhang\thanks{School of Mathematical Sciences, Peking University; email: \texttt{zhzhang@math.pku.edu.cn}.}
}

\begin{document}

\maketitle

\begin{abstract}
We study asymmetric rank-one spiked tensor models in the high-dimensional regime, where the noise entries are independent and identically distributed with zero mean, unit variance, and finite fourth moment. This extends the classical Gaussian framework to a substantially broader class of noise distributions. We analyze the maximum-likelihood estimator associated with the best rank-one approximation of an order-$d$ tensor, for $d\ge 3$.

Our approach is formulated along an informative, spectrally separated branch of stationary points of the non-convex maximum-likelihood landscape. In the core order-three asymmetric model, we verify locally in the high-signal regime that such an informative branch exists and remains separated from the bulk. Under this branch-selection framework, we show that the empirical spectral distribution of a suitable block-wise tensor contraction converges almost surely to the same deterministic limit as in the Gaussian case. As a consequence, the asymptotic singular value and the mode-wise alignments between the estimated and planted spike directions admit the same explicit characterizations as under Gaussian noise.

These results establish a universality principle for asymmetric spiked tensor models: the high-dimensional spectral behavior and statistical limits of the selected maximum-likelihood stationary point are robust beyond the Gaussian setting. Our proof combines resolvent methods from random matrix theory, cumulant expansions under finite fourth-moment assumptions, and Efron--Stein-type variance bounds. A main technical difficulty is to control the statistical dependence between the estimator and the noise, including the associated cross terms in the non-Gaussian setting.
\end{abstract}

\newpage
\tableofcontents
\newpage
\section{Introduction}

High-dimensional tensor models have become central in modern statistics, signal processing,
and machine learning, where latent low-rank structures are to be inferred from noisy
multi-way observations \cite{Kolda_Bader2009,AnandkumarJMLR2014,MontanariRichard2014tensorPCA,Jagannath_2020}. Recently, Goulart et al. \cite{goulart2022randommatrixperspectiverandom}  gave an asymptotic characterization of the maximum likelihood estimation performance for a symmetric order-$d$ rank-one model
with Gaussian noise. Seddik et al.
\cite{seddik2024whentrandomtensorsmeet} studied asymptotic properties of asymmetric order-$d$ spiked tensor models with Gaussian noise. The key ideal in the proof of these works leverages random matrix theory tools for tensor models and then uses Stein's lemma due to the Gaussian noise setting.

In non-Gaussian noise cases, however,  Stein's lemma is no longer applicable. 
A fundamental question in  context is whether the sharp asymptotic behaviors derived
under idealized Gaussian noise assumptions remain valid for more general noise distributions.
Understanding such robustness properties is essential, as real-world data is rarely Gaussian.
This work aims to address this issue. 

Given a positive integer $d \ge 3$, let $n_1,\ldots,n_d \in \mathbb{N}^+$ denote the dimensions of  $d$-way tensor and let
 $N = n_1 + \cdots + n_d$.
We consider a $d$-fold rank-$R$ spiked tensor model
\begin{equation}
    \textbf{T}
    =
    \sum_{r=1}^R \beta_r\, \bm{x}_{r}^{(1)} \otimes \cdots \otimes \bm{x}_{r}^{(d)}
    + \frac{1}{\sqrt{N}}\, \textbf{W},
    \label{eq1.1}
\end{equation}
where $\beta_1 \ge \cdots \ge \beta_R > 0$ are signal-to-noise ratios (SNRs),
$\{\bm{x}_{r}^{(l)}\}_{r=1}^R$ are mutually orthogonal unit vectors in $\mathbb{R}^{n_l}$ for each
mode $l \in [d]$, and $\bf{W}$ is a noise tensor with independent and identically distributed (i.i.d.)
entries satisfying
\begin{equation} \label{eq:moment-assumption}
\mathbb{E}[W_{i_1\cdots i_d}] = 0, \qquad
\mathbb{E}[W_{i_1\cdots i_d}^2] = 1, \qquad
\mathbb{E}|W_{i_1\cdots i_d}|^4 < \infty.
\end{equation}
%for some $\delta > 0$.
In the paper, we mainly focus on the rank-1 case
\begin{equation}
    \textbf{T}
    =
    \beta\, \bm{x}^{(1)} \otimes \cdots \otimes \bm{x}^{(d)}
    + \frac{1}{\sqrt{N}}\, \textbf{W},
    \label{eq1.2}
\end{equation}
which already captures the essential phenomena.

We would retrieve the spike $ \beta\, \bm{x}^{(1)} \otimes \cdots \otimes \bm{x}^{(d)}$  from the noise tensor $ \textbf{T}$.
A natural estimator of the latent rank-one structure is obtained by solving the
maximum likelihood (ML) problem
\begin{equation}
    (\lambda_*, \bm{u}_*^{(1)}, \ldots, \bm{u}_*^{(d)})
    =
    \argmin_{\lambda > 0,\ \bm{u}^{(i)} \in \mathbb{S}^{n_i-1}}
    \left\|
    \textbf{T}
    - \lambda\, \bm{u}^{(1)} \otimes \cdots \otimes \bm{u}^{(d)}
    \right\|_{\bm{F}}^2,
    \label{eq1.3}
\end{equation}
where $\mathbb{S}^{n_i-1}$ denotes the unit sphere in $\mathbb{R}^{n_i}$.
The vectors $\bm{u}_*^{(i)}$ play the role of dominant singular vectors generalized to the
tensor setting, while $\lambda_*$ represents the leading singular value.

Let $\textbf{T}(\bm{u}^{(1)}, \ldots, \bm{u}^{(d)}):=\sum_{i_1, \ldots, i_d} u^{(1)}_{i_1} \cdots u_{i_d}^{(d)} T_{i_1, \ldots, i_d}$ represent the contraction of tensor $ \bf{T}$ on the vectors $\bm{u}^{(i)}$ and $\textbf{T}(\bm{u}^{(1)}, \ldots, \bm{u}^{(k)}, \cdot, \ldots, \cdot)$ for $k<d$ is the  $(d{-}k)$-th order tensor obtained by contracting $\textbf{T}$
along the first $k$ modes.
With these notions and by standard variational arguments \cite{lim2005singularvaluestensors}, the ML problem~(\ref{eq1.3}) is equivalent to
\begin{equation}
    \max_{\substack{\bm{u}^{(i)} \in \mathbb{R}^{n_i} \\ \|\bm{u}^{(i)}\|=1,\ i\in[d]}}
    \left|
    \textbf{T}(\bm{u}^{(1)}, \ldots, \bm{u}^{(d)})
    \right|.
    \label{eq1.4}
\end{equation}
The corresponding Lagrangian function is defined as 
\begin{equation}
    \mathcal{L}(\bm{u}^{(1)}, \ldots, \bm{u}^{(d)}, \lambda)
    :=
    \textbf{T}(\bm{u}^{(1)}, \ldots, \bm{u}^{(d)})
    - \lambda\Big(\prod_{i=1}^d \|\bm{u}^{(i)}\| - 1\Big),
    \qquad \lambda > 0.
    \label{eq1.5}
\end{equation}
Therefore, any stationary point $(\lambda, \bm{u}^{(1)}, \ldots, \bm{u}^{(d)})$ with unit-norm vectors
must satisfy the Karush--Kuhn--Tucker (KKT) conditions
\begin{equation}
\begin{cases}
    \textbf{T}(\bm{u}^{(1)}, \ldots, \bm{u}^{(i-1)}, \cdot,
    \bm{u}^{(i+1)}, \ldots, \bm{u}^{(d)})
    = \lambda\, \bm{u}^{(i)}, & i \in [d], \\
    \lambda = \textbf{T}(\bm{u}^{(1)}, \ldots, \bm{u}^{(d)}),
\end{cases}
\label{eq1.6}
\end{equation}
where the dot denotes the contraction of $\textbf{T}$ along all modes except $i$.

\subsection{Main Contributions}

We investigate the asymptotic behavior of the ML estimator associated with
the stationary conditions in~\eqref{eq1.6} in the high-dimensional regime.
Following the setting in \cite{seddik2024whentrandomtensorsmeet}, we assume that
the tensor dimensions satisfy
\[
n_i \to \infty,
\qquad
\frac{n_i}{\sum_{j=1}^d n_j} \to c_i \in (0,1).
\]
Our analysis is carried out under the branch-selection framework formalized in
Assumptions~\ref{assumption3.1} and~\ref{assumption3.2}.
It is worth pointing out that \cite{seddik2024whentrandomtensorsmeet,liu2025alignmentmatchingtestshighdimensional} also works
under an informative stationary branch assumption.
In the order-three asymmetric model, we further show in
Proposition~\ref{prop:high_signal_order3} that the qualitative content of
Assumption~\ref{assumption3.1} can be verified locally in the high-signal regime.

Under this framework, we derive explicit asymptotic characterizations of the
leading singular value and of the mode-wise alignments
$\langle \bm{x}^{(i)}, \bm{u}^{(i)}_* \rangle$, $i=1,\ldots,d$.
Our results apply to selected stationary points of the ML problem satisfying the
regularity assumptions introduced in Section~\ref{sec:assumption}.
Rather than attempting a complete description of the non-convex optimization
landscape, our goal is to identify the asymptotic spectral and statistical
behavior along an informative stationary branch that remains separated from the
bulk.

From a methodological perspective, our analysis is based on a structured matrix
representation naturally induced by the first-order optimality condition
\cite{lim2005singularvaluestensors}.
This representation allows us to exploit tools from random matrix theory to
characterize both the limiting spectral behavior and the associated eigenvector
statistics.
More specifically, our approach combines resolvent methods, cumulant expansions
under finite fourth-moment assumptions, and Efron--Stein-type variance bounds.
A central technical difficulty is to control the statistical dependence between
the selected stationary point and the noise, including delicate cross terms that
arise in the non-Gaussian setting.

More precisely, for $d \ge 3$, we show that there exists a threshold
$\beta_s > 0$ such that, for all $\beta > \beta_s$, the stationary point
$(\lambda, \bm{u}^{(1)}_*, \ldots, \bm{u}^{(d)}_*)$ satisfies
\begin{equation}
\begin{cases}
\lambda \xrightarrow{\text{a.s.}} \lambda^\infty(\beta), \\[0.5ex]
\big|\langle \bm{x}^{(i)}, \bm{u}^{(i)}_* \rangle\big|
\xrightarrow{\text{a.s.}} q_i\big(\lambda^\infty(\beta)\big),
\qquad i=1,\ldots,d,
\end{cases}
\label{eq1.7}
\end{equation}
where $\lambda^\infty(\beta)$ is the unique solution of
$f(\lambda^\infty(\beta),\beta)=0$, with
\[
f(z,\beta) = z + g(z) - \beta \prod_{i=1}^{d} q_i(z),
\]
and
\[
q_i(z) = \sqrt{1 - \frac{g_i^2(z)}{c_i}}, \qquad
g_i(z) = \frac{g(z) + z}{2}
- \frac{\sqrt{4c_i + (g(z)+z)^2}}{2}.
\]
Here, the function $g(z)$ is defined as the solution to the fixed-point equation
\[
g(z) = \sum_{i=1}^{d} g_i(z),
\qquad \text{for $z$ sufficiently large}.
\]

In addition, in the balanced setting $c_i = 1/d$ for all $i \in [d]$, we show that
for $\beta \in [0,\beta_s]$,
\begin{equation}
\begin{cases}
\lambda \xrightarrow{\text{a.s.}} \lambda^\infty
\le 2\sqrt{\frac{d-1}{d}}, \\[0.5ex]
\big|\langle \bm{x}^{(i)}, \bm{u}^{(i)}_* \rangle\big|
\xrightarrow{\text{a.s.}} 0,
\end{cases}
\label{eq:cease_correlation}
\end{equation}
indicating that the selected stationary points of the ML problem become
asymptotically uninformative below the threshold $\beta_s$.

\begin{remark}
The quantities $q_i(z)$ admit an equivalent representation given by
\begin{equation}
    q_i(z)
    =
    \left(
    \frac{\alpha_i(z)^{\,d-3}}{\prod_{j\neq i}\alpha_j(z)}
    \right)^{\frac{1}{2d-4}},
\end{equation}
where
\begin{equation}
    \alpha_i(z)
    =
    \frac{\beta}{z + g(z) - g_i(z)}.
\end{equation}
This expression is well-defined for $c_i \in [0,1]$ and $d \ge 3$.
\end{remark}

\subsection{Related work}

In a seminal work, Seddik et al.~\cite{seddik2024whentrandomtensorsmeet}
characterized the limiting empirical eigenvalue distribution associated with the
tensor contraction operator as a function of the signal-to-noise ratio in the
case where the noise entries are i.i.d.\ standard Gaussian.
Very recently, Liu et al.~\cite{liu2025alignmentmatchingtestshighdimensional}
remarked that their technique can be employed to extend the same limiting result
to the case where the noise entries are i.i.d.\ sub-exponential.
We also work with the tensor contraction operator $\Phi_d$ introduced in
\cite{seddik2024whentrandomtensorsmeet}.

At the same time, both of the above works treat the informative branch and its
deterministic asymptotic location essentially as hypotheses, and do not study in
detail the theoretical role of the analogue of our Assumption~\ref{assumption3.2}.
By contrast, Section~\ref{sec:assumption} of the present paper makes this point
more explicit: Assumption~\ref{assumption3.1} serves only as a branch-selection
hypothesis, whereas Assumption~\ref{assumption3.2} is used to encode
deterministic asymptotic locations for the selected branch and already implies
useful finite-$N$ consequences, such as high-probability control of the outlier
position and of the mode-wise alignments.
Moreover, in the core order-three asymmetric model, we show that the qualitative
content of Assumption~\ref{assumption3.1} can be verified locally in the
high-signal regime.

However, the main challenge in the spiked tensor model lies in the fact that the
vectors selected by the ML landscape are statistically correlated with the
noise, which necessitates a careful analysis of whether the associated cross
terms vanish asymptotically.
Unfortunately, the argument in \cite{seddik2024whentrandomtensorsmeet} relies on
an order estimate of
\(
\Phi_3({\bf W}, \partial \bm{u}, \partial \bm{v}, \partial \bm{w})
\)
that does not hold even under standard Gaussian noise distributions; see
Remark~\ref{remark4} following Lemma~\ref{lemA.1}.
On the other hand, the technique in
\cite{liu2025alignmentmatchingtestshighdimensional} is designed to establish the
empirical eigenvalue distribution of the noise tensor
\(
\frac{1}{\sqrt{N}}{\bf W}
\)
itself, and does not address the effect of the cross terms arising from the
dependence between the selected stationary point and the noise.

In our work, we complete the analysis of the cross terms in
Lemma~\ref{lemA.1}, thereby correcting and strengthening the results of
\cite{seddik2024whentrandomtensorsmeet}, and further extending them to a more
general setting requiring only independence, centering, unit variance, and a
finite fourth-moment assumption on the noise variables.

The spiked tensor model extends the classical spiked random matrix framework to
higher-order data.
In the matrix case ($d=2$), the model reduces to the well-studied spiked matrix
model, for which it is known that, in the large-dimensional regime, there exists
an order-one critical signal-to-noise ratio $\beta_c$.
Below $\beta_c$, it is information-theoretically impossible to detect or recover
the spike, whereas above $\beta_c$ the spike can be detected and recovered in
polynomial time using singular value decomposition.
This phenomenon is commonly referred to as the BBP (Baik--Ben Arous--Péché)
phase transition~\cite{baik2005phasetransition, benaychgeorges2011eigenvalues, Capitaine2009largest, Peche2006largesteigenvalue}.

In the (symmetric) spiked tensor model for $d \ge 3$, an order-one critical value
$\beta_c(d)$ also exists in the high-dimensional asymptotic regime, below which
it is information-theoretically impossible to detect or recover the spike, while
above $\beta_c(d)$ recovery is theoretically possible via the 
ML estimator.
However, unlike the matrix case where ML can be computed in polynomial time,
computing the ML estimator for tensors of order $d \ge 3$ is NP-hard~\cite{MontanariRichard2014tensorPCA, Biroli_2020}.
This motivates the study of a more practical phase transition characterized by
an algorithmic critical value $\beta_a(d,n)$, above which spike recovery becomes
possible using polynomial-time algorithms.

Montanari and
Richard~\cite{MontanariRichard2014tensorPCA} first introduced the symmetric spiked tensor model
(of the form $\bm{Y} =\mu \bm{x}^{\otimes d}+ \bm{W}$ with symmetric noise $\bm{W}$) and investigated its
algorithmic aspects.
Using heuristic arguments, they suggested that spike recovery is achievable in
polynomial time via approximate message passing (AMP) or tensor power iteration
when $\mu \gtrsim n^{\frac{d-2}{2}}$.
This phase transition was later established rigorously for AMP
in~\cite{Lesieur_2017, Jagannath_2020} and more recently for tensor power iteration in~\cite{Huang2022poweriteration}.

Montanari and
Richard~\cite{MontanariRichard2014tensorPCA} further proposed a tensor unfolding approach, which
consists in reshaping the tensor into a matrix $\mathrm{Mat}(\bm{Y})=\mu \bm{x} \bm{y}^\top+ \bm{Z}$
of size $n^q \times n^{d-q}$ for $q\in[d-1]$, followed by a singular value
decomposition.
For $q=1$, they predicted that successful recovery occurs when
$\mu \gtrsim n^{\frac{d-2}{4}}$.
In a recent work, Ben Arous \emph{et al.}~\cite{BenArousHuangHuang2021long} studied spiked long rectangular
random matrices under fairly general noise assumptions (bounded fourth-order
moments) and proved the existence of a BBP-type phase transition for the extreme
singular value and singular vectors.
Applying their result to the asymmetric rank-one spiked tensor model in~(1) with
equal dimensions via tensor unfolding, they obtained the exact threshold predicted
in~\cite{MontanariRichard2014tensorPCA}, namely $\beta \gtrsim n^{\frac{d-2}{4}}$, for successful signal
recovery.

The structure of the optimization landscape associated with~(\ref{eq1.6})
raises natural questions about the number and nature of its stationary points,
including both local optima and saddle points.
For symmetric spiked tensor models, a detailed landscape analysis was carried out
in~\cite{arous2019landscapespikedtensormodel}, where it was shown that the behavior of
local maxima undergoes a sharp transition as the signal-to-noise ratio varies.
Specifically, when $\beta < \beta_c$, the objective values attained at all local maxima
(including the global one) are concentrated in a narrow interval, whereas for $\beta > \beta_c$
the global maximum separates from this band and grows with $\beta$.

Complementary insights were obtained in~\cite{goulart2022randommatrixperspectiverandom},
where an order-$3$ symmetric spiked tensor model was investigated from a random matrix
theory perspective.
In that work, it was argued that there exists an intermediate threshold
$0 < \beta_s < \beta_c$ such that, for $\beta \in [\beta_s,\beta_c]$, the ML
problem admits a local optimum that is already correlated with the planted spike,
while this spike-aligned solution becomes globally optimal once $\beta > \beta_c$.

Motivated by these results, we conjecture that analogous phenomena persist for asymmetric
spiked tensor models.
In particular, we expect the existence of an order-one critical value $\beta_c$ above which
the ML problem~(\ref{eq1.4}) admits a global maximum aligned with the true
signal.
As in~\cite{goulart2022randommatrixperspectiverandom}, our analysis does not yield an explicit
characterization of $\beta_c$, and determining its precise value is left for future work.
Nevertheless, for asymmetric spiked random tensors, we identify a threshold $\beta_s$ such
that for $\beta > \beta_s$ there exists at least one stationary point of the ML objective
exhibiting nontrivial correlation with the planted spike.

\subsection{Organization}

The remainder of this paper is organized as follows.
In Section~\ref{sec:preliminaries}, we give some notation and notions, and collect several probabilistic 
and random matrix theoretic tools that will be used in the paper. 
%including cumulant expansions under finite moment assumptions and basic spectral norm bounds.
In Section~\ref{sec:assumption}, we discuss the assumption of regular informative stationary branches 
on stationary points and show that it should be satisfied locally in the high-signal regime.
In Section~\ref{sec:order-3}, we present a detailed analysis of the asymmetric rank-one spiked
tensor model of order $d=3$.
%We characterize the limiting spectral distribution of the associated block-wise
%tensor contraction, establish concentration results for the singular value and
%the mode-wise alignments, and derive explicit asymptotic formulas in both the
%informative and uninformative regimes.
In Section~\ref{sec:order-d} we extend the analysis to general asymmetric spiked tensor models of
arbitrary order $d \ge 3$.
%We show that the limiting spectral behavior and the asymptotic characterizations
%of the singular value and alignments persist in the general setting, thereby
%establishing universality beyond the order-$3$ case.
In Section~\ref{sec:rank-r}, we discuss extensions of our results to rank-$r$ spiked tensor models
with orthogonal signal components and show that the asymptotic behavior decouples
across different spikes. In Section~\ref{sec:proof-sketch} we present a proof sketch for our main theorems. 
Finally, in Section~\ref{sec:conclusion} we conclude the work with a summary of the main findings and a
discussion of several open problems and directions for future research.
All the proof details are given in the appendices.  

% \section{Preliminaries}

% We need a tool in random matrix theory(RMT):
% \begin{lemma}\cite{Khorunzhy_1996}
%     If $\xi$ is a real-valued random variable such that
% $\mathbb{E}\{|\xi|^{p+2}\}<\infty$ and if $f(t)$ is a complex-valued function
% of a real variable such that its first $p+1$ derivatives are continuous and
% bounded, then
% \begin{equation}\label{II16}
% \mathbb{E}\{\xi f(\xi)\}
% =
% \sum_{a=0}^{p}\frac{\kappa_{a+1}}{a!}\,
% \mathbb{E}\!\left\{f^{(a)}(\xi)\right\}
% +\varepsilon,
% \end{equation}
% where $\kappa_a$ are the semi-invariants (cumulants) of $\xi$,
% \[
% |\varepsilon|
% \le
% C\,\sup_t |f^{(p+1)}(t)|\,\mathbb{E}\{|\xi|^{p+2}\},
% \]
% and the constant $C$ depends on $p$ only.
% \label{lem2.3}
% \end{lemma}

% \begin{lemma}[Operator norm bound]\label{fact:opnorm}
% Under Assumption~\ref{assumption2}, let $W=(W_{ij})_{1\le i,j\le N}$ be an $N\times N$ random matrix
% with i.i.d.\ entries distributed as $W$. Then
% \begin{equation}\label{eq:opnorm}
% \|W\| \lesssim \sqrt N \quad a.s.
% \end{equation}
% \end{lemma}

% Finally we denote 
% \begin{equation}\label{ccc}
%     \mathbb C^+_{\eta_0} = \{z\in \mathbb{C}|\Im(z) \ge\eta_0\}.
% \end{equation}
\section{Preliminaries}\label{sec:preliminaries}

This section collects several probabilistic and random matrix theoretic tools
that will be used throughout the paper.
All results stated below are standard, and are recalled here for completeness
and to fix notation.

\subsection{Notation and Notions}

For a positive integer $n$, we denote $[n] := \{1,\ldots,n\}$.
Typically, 
vectors are denoted by bold lowercase letters such as $\bm{a},\bm{b},\bm{c}$,
and matrices by bold uppercase letters such as $\bm{A},\bm{B},\bm{C}$.
Tensors are also denoted by bold uppercase letters, such as $\textbf{T},\textbf{W}$.
The set of real $m\times n$ matrices is denoted by $\mathbb{M}_{m,n}$, and
$\mathbb{M}_n := \mathbb{M}_{n,n}$ denotes the set of real square matrices of size $n$.
The set of real tensors of order $d$ and dimensions $n_1,\ldots,n_d$ is denoted by
$\textbf{T}_{n_1,\ldots,n_d}$, while $\textbf{T}_n^{(d)}$ denotes the set of
$d$-th order hypercubic tensors of size $n$.
For a tensor $\textbf{T}\in\textbf{T}_{n_1,\ldots,n_d}$, the entry indexed by
$(i_1, \ldots, i_d)\in[n_1]\times\cdots\times[n_d]$ is written as
$T_{i_1, \ldots, i_d}$.

%Given a tensor $\textbf{W} \in \textbf{T}_{n_1, \ldots, n_d}$, the notation
%\[
%\textbf{W} \sim \textbf{T}_{n_1,\ldots,n_d}
%\]
%means that $\bf{W}$ has i.i.d. entries with distribution $W_{_{i_1,\cdots ,i_d}}$ satisfying %Assumption~\ref{assumption2}.

%Scalars are denoted by (lowercase) letters $a,b,c$.

For a vector $\bm{u}\in\mathbb{R}^n$, we denote by $\|\bm{u}\|_2$ its Euclidean norm,
and by $\langle \bm{u},\bm{v}\rangle := \sum_i u_i v_i$ the standard inner product.
The canonical basis vector in $\mathbb{R}^d$ is denoted by $\bm{e}_i^{(d)}$,
whose $j$-th coordinate is $(\bm{e}_i^{(d)})_j = \delta_{ij}$ (Dirac delta).
For matrices and tensors, $\|\cdot\|$ denotes the operator (or spectral) norm and $\|\cdot\|_{F}$ denotes the Frobenius norm.

%Given a tensor $\textbf{T}\in\textbf{T}_{n_1,\ldots,n_d}$ and vectors
%$\bm{u}^{(1)}\in\mathbb{R}^{n_1},\ldots,\bm{u}^{(d)}\in\mathbb{R}^{n_d}$,
%we define the full contraction
%\[
%\textbf{T}(\bm{u}^{(1)},\ldots,\bm{u}^{(d)}):=\sum_{i_1,\ldots,i_d}T_{i_1,\ldots,i_d}
%\prod_{k=1}^d u^{(k)}_{i_k}.
%\]
%More generally, for vectors $\bm{u}^{(1)},\ldots,\bm{u}^{(k)}$ with $k<d$,
%the partial contraction
%\[
%\textbf{T}(\bm{u}^{(1)},\ldots,\bm{u}^{(k)},\cdot,\ldots,\cdot)
%\]
%denotes the resulting $(d-k)$-th order tensor obtained by contracting $\textbf{T}$
%along the first $k$ modes.

%Finally, we denote by $\mathbb{S}^{N-1}$ the unit sphere in $\mathbb{R}^N$.

Given a $z \in \mathbb{C}$,  let $\Re z$ and $\Im z$ be the real and imaginary parts of $z$, respectively.
For an $\eta_0>0$, we denote by
\begin{equation}\label{ccc}
\mathbb{C}^+_{\eta_0}
=
\{\,z\in\mathbb{C}:\ \Im z \ge \eta_0\,\}
\end{equation}
the closed upper half-plane at distance $\eta_0$ from the real axis.

For a symmetric matrix $\bm{S} \in \mathbb{M}_n$, its resolvent is defined as
\[
\bm{R}_{\bm{S}}(z)\equiv (\bm{S}- z \bm{I})^{-1}, \; z \in \mathbb{C}\setminus S(\bm{S}),
\]
where $S(\bm{S})=\{\lambda_1, \ldots, \lambda_n\}$ is the spectrum of $\bm{S}$. 
Throughout the paper, resolvents will be considered for spectral parameters
belonging to $\mathbb{C}^+_{\eta_0}$. The resolvent provides rich information on the behaviors of the eigenvalues of $\bm{S}$ through the novel Stieltjes transform. 

\begin{definition}[Stieltjes transform] Given a probability measure $\nu$, the Stieltjes transform of $\nu$ is defined by 
    \[
    g_{\nu}(z) \equiv \int{\frac{d \nu(\lambda)}{\lambda-z}}, \quad \mbox{for }  z \in \mathbb{C} \setminus S(\nu),
    \]
    where $S(\nu)$ represents the support of $\nu$.
\end{definition}

Consider that the empirical spectral measure of $\bm{S}$ is defined as
\[
\nu_{\bm{S}} \equiv \frac{1}{n} \sum_{i=1}^n \delta_{\lambda_i(\bm{S})}.
\]
Thus, we can compute the Stieltjes transform of $\nu_{\bm{S}}$ by
\[
g_{\nu_{\bm{S}}} (z) = \frac{1}{n} \sum_{i=1}^n \int{\frac{\delta_{\lambda_i(\bm{S})} (d \lambda)}{\lambda-z}} = \frac{1}{n} \sum_{i=1}^n \frac{1}{\lambda_i(\bm{S)} - z} = \frac{1}{n} \tr \bm{R}_{\bm{S}}(z),
\]
which associates the Stieltjes transform with the resolvent.  
%\subsection*{Cumulant expansion}

\subsection{Tensor contraction operator $\Phi_d$} 
%Compared with random matrix models, 
Unfortunately, the resolvent notion does not generalize to tensors. In this paper we leverage the tensor contraction operator $\Phi_d$ \cite{seddik2024whentrandomtensorsmeet}.
In particular, it maps a tensor $\textbf{W}$ and the unit vectors $\bm a^{(i)}$ to a matrix. That is, 
\begin{equation*}
\Phi_d\colon \ 
\textbf{T}_{n_1,\ldots,n_d}\times
\mathbb S^{n_1-1}\times\cdots\times\mathbb S^{n_d-1}
\longrightarrow
\mathbb M_{\sum_i n_i},
\end{equation*}
defined by
\[
(\textbf{W},\bm a^{(1)},\ldots,\bm a^{(d)})
\longmapsto
\begin{bmatrix}
\bm 0_{n_1\times n_1} & \textbf{W}^{12} & \textbf{W}^{13} & \cdots & \textbf{W}^{1d}\\
(\textbf{W}^{12})^\top & \bm 0_{n_2\times n_2} & \textbf{W}^{23} & \cdots & \textbf{W}^{2d}\\
(\textbf{W}^{13})^\top & (\textbf{W}^{23})^\top & \bm 0_{n_3\times n_3} & \cdots & \textbf{W}^{3d}\\
\vdots & \vdots & \vdots & \ddots & \vdots\\
(\textbf{W}^{1d})^\top & (\textbf{W}^{2d})^\top & (\textbf{W}^{3d})^\top & \cdots & \bm 0_{n_d\times n_d}
\end{bmatrix},
\] 
with
\[
\textbf{W}^{ij}
=
\textbf{W}\big(\bm a^{(1)},\ldots,\bm a^{(i-1)},\,\cdot,\,
\bm a^{(i+1)},\ldots,\bm a^{(j-1)},\,\cdot,\,
\bm a^{(j+1)},\ldots,\bm a^{(d)}\big)
\in\mathbb M_{n_i,n_j}.
\]
It is easily seen that the operator  $\Phi_d$ is linear in $\bf{W}$.  This operator provides an approach for dealing with tensors using random matrix tools.    
Specifically, we can resort to the resolvent method.

\subsection{Cumulant expansion}
A key ingredient in our analysis is a cumulant expansion formula that allows us
to replace Gaussian integration by parts when the noise distribution has only
finite moments.
We will repeatedly apply the following lemma, which goes back to
Khorunzhy et al.~\cite{Khorunzhy_1996}.

\begin{lemma}\cite{Khorunzhy_1996}\label{lem2.3}
Let $\xi$ be a real-valued random variable such that
$\mathbb{E}|\xi|^{p+2}<\infty$ for some integer $p\ge 0$, and let
$f\colon \mathbb{R}\to\mathbb{C}$ be a function whose first $p+1$ derivatives are
continuous and bounded.
Then
\begin{equation}\label{II16}
\mathbb{E}\{\xi f(\xi)\}
=
\sum_{a=0}^{p}\frac{\kappa_{a+1}}{a!}\,
\mathbb{E}\!\left\{f^{(a)}(\xi)\right\}
+\varepsilon,
\end{equation}
where $\kappa_a$ denotes the $a$-th cumulant of $\xi$, and the remaining term
satisfies
\[
|\varepsilon|
\le
C\,\sup_{t\in\mathbb{R}} |f^{(p+1)}(t)|\,\mathbb{E}\{|\xi|^{p+2}\},
\]
with a constant $C$ depending only on $p$.
\end{lemma}

%\subsection*{Spectral norm control}

\subsection{Spectral norm control}

We also require a basic control of the operator norm of random matrices with
independent entries.
Such bounds are classical in random matrix theory and ensure that resolvents are
well defined in the upper half-plane.

\begin{lemma}\cite{BaiYin1986}[Operator norm bound]\label{fact:opnorm}
%Under Assumption~\ref{assumption2}, 
Let $\bm{W}_N$ %$=(W_{ij})_{1\le i,j\le N}$ 
be an $N\times N$ random matrix with i.i.d.\ entries of zero mean, unit variance and finite fourth moment.
Then
\begin{equation}\label{eq:opnorm}
\limsup_{N\to \infty}\left\|\frac{\bm{W}_N}{\sqrt{N}} \right\| \le 2,
\end{equation}
which means
\begin{equation*}
    \|\bm{W}_N\|\lesssim \sqrt{N}\qquad \text{almost surely}.
\end{equation*}
\end{lemma}

\section{Regular informative stationary branches}
\label{sec:assumption}

In the Gaussian spiked tensor model studied in
\cite{seddik2024whentrandomtensorsmeet},
the asymptotic analysis is carried out along a sequence of stationary points
that remains informative and spectrally separated from the bulk.
Because our universality argument is likewise formulated around the tensor
contraction operator, we work with an analogous branch-selection hypothesis in
the present non-Gaussian setting.

Recall that a stationary point
\(
(\lambda_\ast,\bm u_\ast^{(1)},\ldots,\bm u_\ast^{(d)})
\)
of the constrained optimization problem \eqref{eq1.5} satisfies the KKT system
\begin{equation}
\label{eq:sec3-kkt}
T(\bm u_\ast^{(1)},\ldots,\bm u_\ast^{(i-1)},\cdot,
\bm u_\ast^{(i+1)},\ldots,\bm u_\ast^{(d)})
=
\lambda_\ast \bm u_\ast^{(i)},
\qquad i\in[d].
\end{equation}
Our analysis will be performed along a distinguished sequence of such
stationary points which carries nontrivial information on the planted spike
and whose associated contraction matrix stays uniformly away from the bulk
spectrum.

The assumption below should be viewed as a \emph{branch-selection hypothesis}
for the general model. Its role is to isolate the informative stationary branch
along which the universality analysis is carried out, without committing to a
full landscape analysis of the non-convex optimization problem \eqref{eq1.5}.
Importantly, in the order-three asymmetric model, this branch can be verified
locally in the high-signal regime; see
Proposition~\ref{prop:high_signal_order3} below.

\begin{assumption}[Regular informative stationary branch]
\label{assumption3.1}
Assume that there exists a sequence of stationary points
\(
(\lambda_\ast,\bm u_\ast^{(1)},\ldots,\bm u_\ast^{(d)})
\)
satisfying \eqref{eq:sec3-kkt} such that the following hold.

\begin{enumerate}
\item[(i)] \emph{Regularity.}
For all sufficiently large \(N\),
\[
\lambda_\ast
\notin
\sigma\!\bigl(
\Phi_d(T,\bm u_\ast^{(1)},\ldots,\bm u_\ast^{(d)})
\bigr).
\]

\item[(ii)] \emph{Outlier regime.}
There exists a deterministic compact set \(K_\nu\subset\mathbb R\) such that,
with probability tending to one, the spectrum of the relevant contraction
matrix is contained in \(K_\nu\), and
\[
\liminf_{N\to\infty}\mathrm{dist}(\lambda_\ast,K_\nu)>0.
\]

\item[(iii)] \emph{Nontrivial alignment.}
For each \(i\in[d]\),
\[
\liminf_{N\to\infty}
\bigl|
\langle \bm u_\ast^{(i)},\bm x^{(i)}\rangle
\bigr|>0.
\]
\end{enumerate}
\end{assumption}

When deterministic asymptotic locations are needed, we impose the following
supplementary condition on the selected branch.

\begin{assumption}[Deterministic asymptotic location]
\label{assumption3.2}
Assume in addition to Assumption~\ref{assumption3.1} that there exist
deterministic functions \(\lambda^\infty(\beta)\) and
\(a_{\bm x^{(i)}}^\infty(\beta)\) such that
\[
\lambda_\ast \xrightarrow{\mathbb P} \lambda^\infty(\beta),
\qquad
\bigl|
\langle \bm u_\ast^{(i)}, \bm x^{(i)}\rangle
\bigr|
\xrightarrow{\mathbb P} a_{\bm x^{(i)}}^\infty(\beta),
\qquad i\in[d].
\]
Moreover,
\[
\lambda^\infty(\beta)\notin K_\nu,
\qquad
a_{\bm x^{(i)}}^\infty(\beta)>0,
\qquad i\in[d].
\]
\end{assumption}

\begin{remark}[Role of the regularity condition]
\label{rem:regularity}
Assumption~\ref{assumption3.1}-(i) is precisely the non-resonance condition
needed for the resolvent
\[
\bigl(
\Phi_d(T,\bm u_\ast^{(1)}, \ldots, \bm u_\ast^{(d)})-\lambda_\ast \bm I
\bigr)^{-1}
\]
to be well defined.
This is the quantity that enters the derivative formulas used later in the
paper. Thus Assumption~\ref{assumption3.1}-(i) excludes collisions between
\(\lambda_\ast\) and the residual spectrum of the contraction matrix.
\end{remark}

\begin{remark}[Bulk region]
\label{rem:bulkset}
The compact set \(K_\nu\) in Assumption~\ref{assumption3.1}-(ii) should be
viewed as the deterministic bulk region generated by the corresponding
contraction ensemble. Once the limiting spectral law is identified in
Sections~4 and~5, one may take \(K_\nu=\mathcal S(\nu)\), the support of that
law. We keep the notation \(K_\nu\) here in order to avoid a circular
dependence between the present assumptions and the later identification of
\(\nu\).
\end{remark}

\begin{remark}[Scope of the assumptions]
\label{rem:scopeassumption}
Assumption~\ref{assumption3.1} serves only to select the informative branch in
the general model. It should not be interpreted as an attempt to exclude the
difficult part of the optimization landscape; rather, it isolates the branch
along which the universality analysis is performed. The stronger
Assumption~\ref{assumption3.2} is only used when deterministic limiting
singular values and alignments are discussed. In the core order-three
asymmetric model, Proposition~\ref{prop:high_signal_order3} below shows that
the qualitative content of Assumption~\ref{assumption3.1} can in fact be
verified locally in the high-signal regime.
\end{remark}

We first record the finite-\(N\) consequences of
Assumption~\ref{assumption3.2}. Although the latter is formulated in terms of
deterministic asymptotic locations, it immediately yields high-probability
control of the outlier position and of the mode-wise alignments.

\begin{proposition}[High-probability stability of the informative branch]
\label{prop:hp_risb}
Assume Assumptions~\ref{assumption3.1} and~\ref{assumption3.2}. Set
\[
\Delta_\beta:=\mathrm{dist}\bigl(\lambda^\infty(\beta),K_\nu\bigr)>0,
\qquad
a_\beta^{(i)}:=a_{x^{(i)}}^\infty(\beta)>0,
\qquad i\in[d].
\]
Then for every choice of constants
\[
0<\varepsilon<\Delta_\beta,
\qquad
0<\eta_i<a_\beta^{(i)}, \quad i\in[d],
\]
there exists \(N_0=N_0(\beta,\varepsilon,\eta_1,\ldots,\eta_d)\) such that for
all \(N\ge N_0\),
\begin{align}
\mathbb P\!\left(
\mathrm{dist}(\lambda_\ast,K_\nu)\ge \Delta_\beta-\varepsilon
\right) &\ge 1-\varepsilon,
\label{eq:hpoutlier}
\\
\mathbb P\!\left(
|\langle \bm u_\ast^{(i)},\bm x^{(i)}\rangle|
\ge a_\beta^{(i)}-\eta_i,\ \forall i\in[d]
\right) &\ge 1-\varepsilon.
\label{eq:hpalign}
\end{align}
In particular, for all sufficiently large \(N\), with probability at least
\(1-\varepsilon\),
\[
\mathrm{dist}(\lambda_\ast,K_\nu)\ge \frac{\Delta_\beta}{2},
\qquad
|\langle \bm u_\ast^{(i)},\bm x^{(i)}\rangle|
\ge \frac{a_\beta^{(i)}}{2},
\quad i\in[d].
\]
\end{proposition}

\begin{proof}
Since
\(
\lambda_\ast \xrightarrow{\mathbb P} \lambda^\infty(\beta)
\)
and
\(
|\langle \bm u_\ast^{(i)},\bm x^{(i)}\rangle|
\xrightarrow{\mathbb P} a_\beta^{(i)}
\),
for every \(\varepsilon>0\) and sufficiently large \(N\),
\begin{align*}
\mathbb P\!\left(
|\lambda_\ast-\lambda^\infty(\beta)|<\varepsilon
\right) &\ge 1-\varepsilon,
\\
\mathbb P\!\left(
\bigl||
\langle \bm u_\ast^{(i)},\bm x^{(i)}\rangle
|-a_\beta^{(i)}
\bigr|<\eta_i,\ \forall i\in[d]
\right) &\ge 1-\varepsilon.
\end{align*}
On the event
\(
|\lambda_\ast-\lambda^\infty(\beta)|<\varepsilon
\),
the triangle inequality yields
\[
\mathrm{dist}(\lambda_\ast,K_\nu)
\ge
\mathrm{dist}(\lambda^\infty(\beta),K_\nu)-|\lambda_\ast-\lambda^\infty(\beta)|
\ge
\Delta_\beta-\varepsilon,
\]
which proves \eqref{eq:hpoutlier}. Likewise, on the event
\[
\bigl||
\langle \bm u_\ast^{(i)},\bm x^{(i)}\rangle
|-a_\beta^{(i)}
\bigr|<\eta_i,
\]
we have
\[
|\langle \bm u_\ast^{(i)},\bm x^{(i)}\rangle|
\ge a_\beta^{(i)}-\eta_i,
\qquad i\in[d],
\]
which proves \eqref{eq:hpalign}. The final statement follows by taking
\(\eta_i=a_\beta^{(i)}/2\) and replacing \(\varepsilon\) by
\(\min\{\varepsilon,\Delta_\beta/2\}\).
\end{proof}

An immediate consequence is a uniform high-probability bound on the resolvent
evaluated along the informative branch.

\begin{corollary}[Resolvent bound along the informative branch]
\label{cor:resolvent_hp}
Assume Assumptions~\ref{assumption3.1} and~\ref{assumption3.2}. Then for every
\(0<\varepsilon<\Delta_\beta\), there exists \(N_0\) such that for all
\(N\ge N_0\),
\[
\mathbb P\!\left(
\left\|
\bigl(
\Phi_d(T,\bm u_\ast^{(1)},\ldots,\bm u_\ast^{(d)})-\lambda_\ast \bm I
\bigr)^{-1}
\right\|
\le
\frac{1}{\Delta_\beta-\varepsilon}
\right)\ge 1-\varepsilon.
\]
In particular, for every \(\varepsilon>0\), for all sufficiently large \(N\),
\[
\mathbb P\!\left(
\left\|
\bigl(
\Phi_d(T,\bm u_\ast^{(1)},\ldots,\bm u_\ast^{(d)})-\lambda_\ast \bm I
\bigr)^{-1}
\right\|
\le
\frac{2}{\Delta_\beta}
\right)\ge 1-\varepsilon.
\]
\end{corollary}

\begin{proof}
For any Hermitian matrix \(A\) and any \(\mu\notin\sigma(A)\),
\[
\|(A-\mu \bm I)^{-1}\|
=
\frac{1}{\mathrm{dist}(\mu, \sigma(A))}.
\]
By Assumption~\ref{assumption3.1}(ii), with probability tending to one, the
spectrum of the relevant contraction matrix is contained in \(K_\nu\).
On the event in \eqref{eq:hpoutlier}, we therefore have
\[
\mathrm{dist}\!\left(
\lambda_\ast,
\sigma\!\bigl(\Phi_d(T,\bm u_\ast^{(1)},\ldots,\bm u_\ast^{(d)})\bigr)
\right)
\ge
\mathrm{dist}(\lambda_\ast, K_\nu)
\ge
\Delta_\beta-\varepsilon,
\]
and hence
\[
\left\|
\bigl(
\Phi_d(T,\bm u_\ast^{(1)},\ldots,\bm u_\ast^{(d)})-\lambda_\ast \bm I
\bigr)^{-1}
\right\|
\le
\frac{1}{\Delta_\beta-\varepsilon}.
\]
This proves the first claim. The second follows by replacing \(\varepsilon\)
with \(\min\{\varepsilon,\Delta_\beta/2\}\).
\end{proof}

We next show that, in the order-three asymmetric model, the qualitative
content of Assumption~\ref{assumption3.1} can be verified directly in the
high-signal regime.

\begin{proposition}[High-signal local verification in the order-three model]
\label{prop:high_signal_order3}
Consider the order-three asymmetric spiked tensor model
\[
T=\beta\,\bm x\otimes \bm y\otimes \bm z+\frac1{\sqrt N}W,
\qquad
\bm x\in S^{m-1},\ \bm y\in S^{n-1},\ \bm z\in S^{p-1},
\qquad N=m+n+p.
\]
Assume that the entries of \(W\) are independent, centered, have variance one,
and finite fourth moment.

Then there exist constants \(L_0>0\) and \(C>0\), depending only on the aspect
ratios, such that, with probability tending to one, the event
\(\mathcal E_{L_0}\) holds; on this event, if
\[
\beta \ge C N^{1/4},
\]
then there exists a stationary point
\[
(\lambda_\ast,\bm u_\ast,\bm v_\ast,\bm w_\ast)
\in \mathbb R\times S^{m-1}\times S^{n-1}\times S^{p-1}
\]
satisfying
\[
T(\bm v_\ast)\bm w_\ast=\lambda_\ast \bm u_\ast,\qquad
T(\bm u_\ast)\bm w_\ast=\lambda_\ast \bm v_\ast,\qquad
T(\bm v_\ast)^\top \bm u_\ast=\lambda_\ast \bm w_\ast,
\]
and
\begin{align}
\|\bm u_\ast-\bm x\|+\|\bm v_\ast-\bm y\|+\|\bm w_\ast-\bm z\|
&\le \frac{C}{\beta},
\label{eq:order3_close}\\
|\lambda_\ast-\beta|
&\le C,
\label{eq:order3_lambda_close}\\
|\langle \bm u_\ast,\bm x\rangle|
+|\langle \bm v_\ast,\bm y\rangle|
+|\langle \bm w_\ast,\bm z\rangle|
&\ge 3-\frac{C}{\beta^2}.
\label{eq:order3_align}
\end{align}
Moreover, if \(K_\nu\subset\mathbb R\) is a deterministic compact bulk set
containing the spectrum of the contraction matrix
\(\Phi_3(T,\bm u_\ast,\bm v_\ast,\bm w_\ast)\) with probability tending to one, then for
all sufficiently large \(\beta\),
\begin{equation}
\label{eq:order3_outlier}
\mathrm{dist}(\lambda_\ast,K_\nu)\ge c\beta
\end{equation}
for some constant \(c>0\), and in particular
\begin{equation}
\label{eq:order3_regular}
\lambda_\ast\notin \sigma\!\bigl(\Phi_3(T,\bm u_\ast,\bm v_\ast,\bm w_\ast)\bigr).
\end{equation}
Consequently, in the order-three asymmetric model, the regularity, outlier,
and nontrivial alignment properties required in
Assumption~\ref{assumption3.1} hold with high probability in the high-signal
regime.
\end{proposition}

\begin{proof}
See Appendix \ref{appSec3}.
\end{proof}
\begin{remark}[On higher-order extensions]
The proof of Proposition~\ref{prop:high_signal_order3} is based on a local
fixed-point argument near the planted spike. The same strategy can in
principle be extended to general order \(d\), by parameterizing the stationary
equations in a neighborhood of
\(
(\bm x^{(1)},\ldots,\bm x^{(d)})
\)
and treating the noise terms as perturbative multilinear contractions.
In that case, one expects the corresponding informative stationary point to be
constructible in the high-signal regime
\(
\beta \gg N^{(d-2)/4}.
\)

The main additional ingredient in the general \(d\)-th order case is a local
high-probability control of the partially contracted noise tensor and of its
Lipschitz dependence on the local coordinates. Establishing these bounds
would require a substantial amount of additional notation and multilinear
estimates, we do not pursue the full higher-order extension here.
\end{remark}
\begin{remark}
It is worth emphasizing that the present paper is not merely a reformulation of
previous Gaussian results.
Although the statistical questions addressed here are analogous to those
considered in earlier works, the role of Gaussianity there is essential.

For the symmetric model studied by %Jagannath--Lopatto--Miolane~
Jagannath et al.
\textcolor{blue}{\cite{Jagannath_2020}}, the analysis of
weak recovery and best asymptotic alignment is carried out in an explicitly
Gaussian tensor model.
Their argument uses the Gaussian likelihood-ratio structure, the associated
free-energy representation, and Gaussian integration-by-parts identities to
connect testing and estimation thresholds with a variational characterization of
the constrained maximum likelihood.
Likewise, the asymmetric analysis of  Goulart et al~\textcolor{blue}{\cite{goulart2022randommatrixperspectiverandom}}\ is formulated in the
Gaussian setting and does not provide a universality statement under general
non-Gaussian perturbations.

In contrast, our framework requires only that the noise entries be
independent, centered, have unit variance, and possess a finite fourth moment.
Moreover, %Section~\ref{sec:assumption} 
we make explicit a point that is left
implicit in the Gaussian literature: Assumption~\ref{assumption3.1} plays the
role of selecting an informative branch, whereas
Assumption~\ref{assumption3.2} encodes its deterministic asymptotic location and
already yields useful finite-$N$ consequences, such as high-probability control
of the outlier position and of the mode-wise alignments.
The subsequent universality analysis then shows that the limiting spectral and
statistical behavior along this branch is identical to that of the Gaussian
case.

Technically, this requires a proof strategy that is different from the
Gaussian-based approaches.
Instead of relying on Gaussian integration by parts, we use resolvent methods,
cumulant expansions under finite fourth-moment assumptions, and
Efron--Stein-type variance bounds to control the statistical dependence between
the selected stationary point and the noise, including the cross terms that
arise in the non-Gaussian setting.
For this reason, the present work genuinely extends the scope of previous
Gaussian analyses rather than simply reproving them in a different notation.
\end{remark}

In Section~\ref{sec:order-3} we will see that in the order-three asymmetric model the empirical
spectral distribution of the contraction matrix is universal even though the
singular vectors entering the contraction are themselves random and depend on
the noise tensor. Combined with the high-probability stability established
above, this provides the analytic framework needed to transfer Gaussian
asymptotics to the finite-fourth-moment setting.

\section{The order-three asymmetric spiked tensor model with a rank-one perturbation}
\label{sec:order-3}

To illustrate the main idea of our approach, we begin with the asymmetric
rank-one spiked tensor model of order 3,
\begin{align}
    \textbf{T} = \beta \bm{x} \otimes \bm{y} \otimes \bm{z}
    + \frac{1}{\sqrt{N}}\, \bf{W},
    \label{eq3.1}
\end{align}
where $\textbf{T},\textbf{W}\in\mathbb{R}^{m\times n\times p}$ and
$N=m+n+p$.
The vectors
\(
\bm{x}\in\mathbb{S}^{m-1},
\bm{y}\in\mathbb{S}^{n-1},
\bm{z}\in\mathbb{S}^{p-1}
\)
represent the planted signal, while $\bm{W}$ is a noise tensor with i.i.d.\
entries satisfying Assumption~(\ref{eq:moment-assumption}) and independent of
$\bm{x},\bm{y},\bm{z}$.

\subsection{Tensor singular value and singular vectors}

According to the first-order optimality conditions~(\ref{eq1.6}), the
$\ell_2$-singular value and singular vectors
$(\lambda,\bm{u},\bm{v},\bm{w})$ associated with a best rank-one approximation
of $\textbf{T}$ satisfy the system
\begin{equation}
  \textbf{T}(\bm{v}) \bm{w} = \lambda \bm{u}, \qquad
  \textbf{T}(\bm{u}) \bm{w} = \lambda \bm{v}, \qquad
  \textbf{T}(\bm{v})^{\top} \bm{u} = \lambda \bm{w},
  \label{eq3.2}
\end{equation}
with
\(
(\bm{u},\bm{v},\bm{w})
\in
\mathbb{S}^{m-1}\times\mathbb{S}^{n-1}\times\mathbb{S}^{p-1}.
\)
Here we use the notation
\[
  \textbf{T}(\bm{u}) = \textbf{T}(\bm{u},\cdot,\cdot)\in\mathbb{M}_{n,p},\quad
  \textbf{T}(\bm{v}) = \textbf{T}(\cdot,\bm{v},\cdot)\in\mathbb{M}_{m,p},\quad
  \textbf{T}(\bm{w}) = \textbf{T}(\cdot,\cdot,\bm{w})\in\mathbb{M}_{m,n}.
\]
The corresponding singular value is given by the full contraction
\begin{equation}
  \lambda
  =
  \textbf{T}(\bm{u},\bm{v},\bm{w})
  =
  \sum_{i,j,k} u_i v_j w_k T_{ijk}.
  \label{eq3.3}
\end{equation}

The following regularity result ensures that the singular value and vectors can
be treated as differentiable functions of the noise tensor, thereby allowing
the use of cumulant expansions in the subsequent analysis.

\begin{lemma}
There exists an almost everywhere continuously differentiable mapping
\(
\mathcal{F}:\textbf{T}_{m,n,p}\to\mathbb{R}^{m+n+p+1}
\)
such that
\[
\mathcal{F}(\textbf{W})
=
\bigl(
\lambda(\textbf{W}),
\bm{u}(\textbf{W}),
\bm{v}(\textbf{W}),
\bm{w}(\textbf{W})
\bigr)
\]
coincides with the singular value and singular vectors of $\textbf{T}$ for
almost every realization of $\textbf{W}$.
\end{lemma}

The proof follows the same lines as
\cite[Lemma~8]{goulart2022randommatrixperspectiverandom}
and is therefore omitted.

Differentiating the identities in~(\ref{eq3.2}) w.r.t.\ a noise entry
$W_{ijk}$, $(i,j,k)\in[m]\times[n]\times[p]$, yields
\begin{align}
\begin{bmatrix}
\displaystyle \frac{\partial \bm{u}}{\partial W_{ijk}} \\[4pt]
\displaystyle \frac{\partial \bm{v}}{\partial W_{ijk}} \\[4pt]
\displaystyle \frac{\partial \bm{w}}{\partial W_{ijk}}
\end{bmatrix}
= 
-\frac{1}{\sqrt{N}}
\left(
\begin{bmatrix}
\textbf{0}_{m\times m} & \textbf{T}(\bm{w}) & \textbf{T}(\bm{v}) \\
\textbf{T}(\bm{w})^{\top} & \textbf{0}_{n\times n} & \textbf{T}(\bm{u}) \\
\textbf{T}(\bm{v})^{\top} & \textbf{T}(\bm{u})^{\top} & \textbf{0}_{p\times p}
\end{bmatrix}
-\lambda \bm{I}_N
\right)^{-1}
\begin{bmatrix}
v_j w_k (\bm{e}^{\,m}_i - u_i\bm{u}) \\[4pt]
u_i w_k (\bm{e}^{\,n}_j - v_j\bm{v}) \\[4pt]
u_i v_j (\bm{e}^{\,p}_k - w_k\bm{w})
\end{bmatrix},
\label{eq3.4}
\end{align}
where $\bm{e}^{\,d}_i$ denotes the $i$-th canonical basis vector in $\mathbb{R}^d$.
In addition, we have
\begin{equation}
\frac{\partial \lambda}{\partial W_{ijk}}
=
\frac{1}{\sqrt{N}}\, u_i v_j w_k .
\label{eq3.5}
\end{equation}

The second-order optimality conditions associated with the ML problem impose
additional constraints on the spectrum of an operator naturally induced by the
KKT system.
Specifically, using the linearity property of $\Phi_3$, we introduce the following expression
\begin{align}
\Phi_3(\textbf{T},\bm{u},\bm{v},\bm{w})
&=\Phi_3( \beta \bm{x} \otimes \bm{y} \otimes \bm{z}, \bm{u}, \bm{v},\bm{w})  + \frac{1}{\sqrt{N}}\, \Phi_3(\textbf{W},\bm{u},\bm{v},\bm{w}) \nonumber \\
&=
\beta\, \bm{V}\bm{B}\bm{V}^{\top}
+ \frac{1}{\sqrt{N}}\,
\Phi_3(\textbf{W},\bm{u},\bm{v},\bm{w})
\in\mathbb{M}_N,
\label{eq3.6}
\end{align}
where
\[
\bm{B} =
\begin{bmatrix}
0 & \langle \bm{z},\bm{w}\rangle & \langle \bm{y},\bm{v}\rangle \\
\langle \bm{z},\bm{w}\rangle & 0 & \langle \bm{x},\bm{u}\rangle \\
\langle \bm{y},\bm{v}\rangle & \langle \bm{x},\bm{u}\rangle & 0
\end{bmatrix},
\qquad
\bm{V} =
\begin{bmatrix}
\bm{x} & \bm{0}_m & \bm{0}_m \\
\bm{0}_n & \bm{y} & \bm{0}_n \\
\bm{0}_p & \bm{0}_p & \bm{z}
\end{bmatrix}.
\]

By classical results on the second-order necessary conditions for constrained
optimization \cite[Theorem~12.5]{NocedalWright2006}, a necessary condition for
$(\bm{u},\bm{v},\bm{w})$ to correspond to a local maximum of the ML objective is
\begin{equation}
\max_{\bm{k}\in\mathbb{S}^{N-1}\cap\bm{h}^{\perp}}
\bigl\langle
\Phi_3(\textbf{T},\bm{u},\bm{v},\bm{w})\bm{k},\bm{k}
\bigr\rangle
\le \lambda,
\label{eq3.7}
\end{equation}
where $\bm{h} = [\bm{u}^{\top},\bm{v}^{\top},\bm{w}^{\top}]^{\top}/\sqrt{3}$.

Consequently, when $\lambda>0$, the largest eigenvalue of
$\Phi_3(\textbf{T},\bm{u},\bm{v},\bm{w})$ equals $2\lambda$, while its remaining
spectrum must be below $\lambda$.
Our analysis therefore focuses on stationary points whose associated singular
value lies outside the bulk of the spectrum.
Equivalently, we assume that $\lambda$ is not an eigenvalue of
$\Phi_3(\textbf{T},\bm{u},\bm{v},\bm{w})$, an assumption that guarantees the
existence of the inverse in Eq.~(\ref{eq3.4}) and that is satisfied for sufficiently
large signal-to-noise ratios.

Finally, we note that for any stationary point verifying Eq.~(\ref{eq3.2}), the
eigenvalue $2\lambda$ remains isolated from the bulk of
$\Phi_3(\textbf{T},\bm{u},\bm{v},\bm{w})$ independent of $\beta$.
Moreover, there exists $\beta_s>0$ such that for $\beta\le\beta_s$ the associated
singular vectors become asymptotically uninformative, in the sense that their
alignments with the planted signal vanish, while an isolated eigenvalue persists
due to statistical dependence between
$(\bm{u},\bm{v},\bm{w})$ and $\textbf{W}$.
This phenomenon mirrors what has been observed previously for symmetric spiked
tensor models~\cite{goulart2022randommatrixperspectiverandom}.

\subsection{The limiting spectral measure associated with block-wise contractions of third-order tensors}

We begin by analyzing the spectral behavior of the matrix
$\frac{1}{\sqrt{N}}\Phi_3(\textbf{W},\bm{u},\bm{v},\bm{w})$
when the vectors $\bm{u},\bm{v},\bm{w}$ are deterministic.
This intermediate step plays a central role in our analysis, as it provides
the reference spectral distribution against which the effect of statistical
dependence between the signal and the noise will later be assessed.

\begin{theorem}\label{thm3.2}
Let $\textbf{W}\sim\textbf{T}_{m,n,p}(\bm W)$ be a sequence of random asymmetric tensors
whose entries satisfy Assumption~(\ref{eq:moment-assumption}), and let
\[
(\bm{u},\bm{v},\bm{w})
\in
\mathbb{S}^{m-1}\times\mathbb{S}^{n-1}\times\mathbb{S}^{p-1}
\]
be a sequence of deterministic unit vectors of increasing dimensions satisfying
Assumption~\ref{assumption3.2}.
Then the empirical spectral measure of
\[
\frac{1}{\sqrt{N}}\,
\Phi_3(\textbf{W},\bm{u},\bm{v},\bm{w})
\]
converges weakly almost surely to a deterministic probability measure $\nu$.
The associated Stieltjes transform $g(z)$ is characterized as the unique solution
to
\[
g(z)=\sum_{i=1}^{3} g_i(z),
\]
with $\Im g(z)>0$ for $\Im z>0$, where for each $i\in\{1,2,3\}$ the functions $g_i(z)$
satisfy
\[
g_i^2(z)-(g(z)+z)g_i(z)-c_i=0,
\qquad
z\in\mathbb{C}\setminus\mathcal{S}(\nu).
\]
\end{theorem}

\begin{proof}
See Appendix~\ref{appA.2}.
\end{proof}

In the balanced case, the limiting measure admits an explicit closed form.

\begin{corollary}\label{cor3.1}
Under the assumptions of Theorem~\ref{thm3.2} and with $c_1=c_2=c_3=\frac13$,
the empirical spectral measure of
\[
\frac{1}{\sqrt N}\,
\Phi_3(\textbf{W},\bm{a},\bm{b},\bm{c})
\]
converges weakly almost surely to the semicircle distribution supported on
\[
\mathcal{S}(\nu)
=
\Bigl[-2\sqrt{\tfrac23},\,2\sqrt{\tfrac23}\Bigr],
\]
whose density and Stieltjes transform are respectively given by
\[
\nu(dx)
=
\frac{3}{4\pi}\sqrt{\Bigl(\tfrac{8}{3}-x^2\Bigr)_+}\,dx,
\qquad
g(z)
=
\frac{-3z+3\sqrt{z^2-\tfrac83}}{4},
\quad
z\in\mathbb C\setminus\mathcal{S}(\nu).
\]
\end{corollary}

\begin{proof}
See Appendix~\ref{appA.3}.
\end{proof}

We now return to the original spiked tensor $\bf{T}$.
In contrast to the previous setting, the singular vectors
$\bm{u}_\ast,\bm{v}_\ast,\bm{w}_\ast$ entering the definition of
$\Phi_3(\textbf{T},\bm{u}_\ast,\bm{v}_\ast,\bm{w}_\ast)$
are no longer deterministic, but depend statistically on the noise tensor
$\textbf{W}$.
A priori, such dependence could alter the limiting spectral distribution.
The following result shows that this is not the case.

\begin{theorem}\label{thm3.3}
Let $\textbf{T}$ be a sequence of random tensors defined as in Eq.~(\ref{eq3.1}).
Under Assumption~\ref{assumption3.2}, the empirical spectral measure of
\[
\Phi_3(\textbf{T},\bm{u}_\ast,\bm{v}_\ast,\bm{w}_\ast)
\]
converges weakly almost surely to the same deterministic probability measure $\nu$
as in Theorem~\ref{thm3.2}.
Equivalently, its Stieltjes transform $g(z)$ is the unique solution to
\[
g(z)=\sum_{i=1}^3 g_i(z),
\]
with $\Im g(z)>0$ for $\Im z>0$, where for $i\in[3]$,
\[
g_i^2(z)-(g(z)+z)g_i(z)-c_i=0,
\qquad
z\in\mathbb C\setminus\mathcal{S}(\nu).
\]
\end{theorem}

%\begin{proof}
The proof is sketched in Section~\ref{sketch3.2}; a complete and detailed proof is given in Appendix~\ref{appA.4}.

\subsection{Concentration of the singular value and the alignments}

In the high-dimensional regime, both the singular value $\lambda$ and the
mode-wise alignments $\langle \bm{u},\bm{x}\rangle$,
$\langle \bm{v},\bm{y}\rangle$, and $\langle \bm{w},\bm{z}\rangle$
exhibit concentration around deterministic limits.
This property follows from a variance control based on Efron--Stein--type
arguments combined with the explicit derivative formulas established in Eq. (\ref{eq3.5}).

We first consider the singular value $\lambda$.
Recalling the expression~(\ref{eq3.3}), the variance of $\lambda$ can be bounded
as
\begin{equation}
\label{eq:var-lambda}
\mathrm{Var}(\lambda)
\le
\sum_{i,j,k}
\mathbb{E}\!\left|
\frac{\partial \lambda}{\partial W_{ijk}}
\right|^2
=
\frac{1}{N}
\sum_{i,j,k} u_i^2 v_j^2 w_k^2
=
\frac{1}{N}.
\end{equation}
In particular, $\mathrm{Var}(\lambda)=O(N^{-1})$, which implies that the
fluctuations of $\lambda$ vanish as the tensor dimensions grow.

Higher-order moment bounds can be obtained in a similar fashion, yielding
concentration estimates for $\lambda$.
For instance, by Chebyshev's inequality, for any $t>0$,
\begin{equation}
\label{eq:chebyshev-lambda}
\mathbb{P}\bigl(|\lambda-\mathbb{E}\lambda|\ge t\bigr)
\le
\frac{1}{N t^2}.
\end{equation}

An analogous argument applies to the alignment terms.
Using the derivative representation~(\ref{eq3.4}) together with the spectral
norm bounds established in Section~\ref{sec:preliminaries}, one can show that
there exists a constant $C>0$ such that, for all $t>0$,
\begin{align}
\label{eq:chebyshev-align}
\mathbb{P}\bigl(|\langle \bm{u},\bm{x}\rangle
- \mathbb{E}\langle \bm{u},\bm{x}\rangle|\ge t\bigr)
&\le \frac{C}{N t^2}, \notag\\
\mathbb{P}\bigl(|\langle \bm{v},\bm{y}\rangle
- \mathbb{E}\langle \bm{v},\bm{y}\rangle|\ge t\bigr)
&\le \frac{C}{N t^2}, \notag\\
\mathbb{P}\bigl(|\langle \bm{w},\bm{z}\rangle
- \mathbb{E}\langle \bm{w},\bm{z}\rangle|\ge t\bigr)
&\le \frac{C}{N t^2}.
\end{align}

As a consequence of these concentration bounds, both the singular value and the
alignments converge almost surely to deterministic limits.
For the remainder of the paper, we denote these limits by
\begin{equation}
\label{eq:as-limits}
\lambda^\infty(\beta)
\equiv \lim_{N\to\infty}\lambda, \quad
a_x^\infty(\beta)
\equiv \lim_{N\to\infty}\langle \bm{u},\bm{x}\rangle, \quad
a_y^\infty(\beta)
\equiv \lim_{N\to\infty}\langle \bm{v},\bm{y}\rangle, \quad
a_z^\infty(\beta)
\equiv \lim_{N\to\infty}\langle \bm{w},\bm{z}\rangle .
\end{equation}

\subsection{Asymptotic singular value and alignments}

With the concentration result from the previous subsection, it remains to estimate the expectations
\[
\mathbb{E}\lambda,\quad 
\mathbb{E}\langle \bm u,\bm x\rangle,\quad 
\mathbb{E}\langle \bm v,\bm y\rangle,\quad 
\mathbb{E}\langle \bm w,\bm z\rangle .
\]
We obtain the following result.
\begin{theorem}\label{thm3.4}
    Recall the notation in Theorem~\ref{thm3.3}. Under Assumptions~\ref{assumption3.2}, there exists
$\beta_s>0$ such that for $\beta>\beta_s$,
\[
\left\{
\begin{aligned}
\lambda &\xrightarrow{\text{a.s.}} \lambda^\infty(\beta),\\[4pt]
\bigl|\langle \bm u,\bm x\rangle\bigr| 
&\xrightarrow{\text{a.s.}} 
\frac{1}{\sqrt{\alpha_2(\lambda^\infty(\beta))\,\alpha_3(\lambda^\infty(\beta))}},\\[6pt]
\bigl|\langle \bm v,\bm y\rangle\bigr| 
&\xrightarrow{\text{a.s.}} 
\frac{1}{\sqrt{\alpha_1(\lambda^\infty(\beta))\,\alpha_3(\lambda^\infty(\beta))}},\\[6pt]
\bigl|\langle \bm w,\bm z\rangle\bigr| 
&\xrightarrow{\text{a.s.}} 
\frac{1}{\sqrt{\alpha_1(\lambda^\infty(\beta))\,\alpha_2(\lambda^\infty(\beta))}} .
\end{aligned}
\right.
\]
Here
\[
\alpha_i(z) \equiv \frac{\beta}{z+g(z)-g_i(z)},
\]
and $\lambda^\infty(\beta)$ satisfies
\[
f\bigl(\lambda^\infty(\beta),\beta\bigr)=0,
\qquad
f(z,\beta)= z+g(z)-\frac{\beta}{\alpha_1(z)\alpha_2(z)\alpha_3(z)}.
\]
In addition, for $\beta\in[0,\beta_s]$, $\lambda^\infty$ is bounded by an order-one constant and
\[
\bigl|\langle \bm u,\bm x\rangle\bigr|,\;
\bigl|\langle \bm v,\bm y\rangle\bigr|,\;
\bigl|\langle \bm w,\bm z\rangle\bigr|
\xrightarrow{\text{a.s.}} 0 .
\]
\end{theorem}
\begin{proof}
    See Appendix \ref{appA.5}
\end{proof}
\begin{remark}
Note that, given $g(z)$, the inverse formula expressing $\beta$ in terms of
$\lambda^\infty$ is explicit. Specifically, we have
\[
\beta(\lambda^\infty)
=
\sqrt{
\frac{\prod_{i=1}^{3}\bigl(\lambda^\infty+g(\lambda^\infty)-g_i(\lambda^\infty)\bigr)}
{\lambda^\infty+g(\lambda^\infty)}
}.
\]
In particular, this inverse formula provides an estimator of the signal-to-noise
ratio $\beta$ given the largest singular value $\lambda$ of $\textbf{T}$.
\end{remark}

% \subsection{Cubic third-order tensors: case $c_1=c_2=c_3=\tfrac13$.}

% In this section, we study the particular case where all the tensor dimensions
% are equal. As such, the three alignments
% $\langle u,x\rangle$, $\langle v,y\rangle$ and $\langle w,z\rangle$
% converge almost surely to the same quantity.
% In this case, the almost sure limits of $\lambda$ and
% $\langle u,x\rangle$, $\langle v,y\rangle$, $\langle w,z\rangle$
% can be obtained explicitly in terms of the signal strength $\beta$,
% as stated in the following corollary of Theorem~\ref{thm3.4}.

% \begin{corollary}\label{cor3.2}
% Under Assumption~\ref{assumption1} with $c_1=c_2=c_3=\tfrac13$, for
% $\beta>\beta_s=\tfrac{2\sqrt{3}}{3}$, we have
% \[
% \left\{
% \begin{aligned}
% \lambda &\xrightarrow{\mathrm{a.s.}}
% \lambda^\infty(\beta)
% \;\equiv\;
% \sqrt{
% \frac{\beta^2}{2}
% +2
% +\frac{\sqrt{3}\sqrt{(3\beta^2-4)^3}}{18\beta}
% },\\[10pt]
% \bigl|\langle u,x\rangle\bigr|,
% \ \bigl|\langle v,y\rangle\bigr|,
% \ \bigl|\langle w,z\rangle\bigr|
% &\xrightarrow{\mathrm{a.s.}}
% \frac{
% \sqrt{\,9\beta^2-12+\dfrac{\sqrt{3}\sqrt{(3\beta^2-4)^3}}{\beta}}
% +
% \sqrt{\,9\beta^2+36+\dfrac{\sqrt{3}\sqrt{(3\beta^2-4)^3}}{\beta}}
% }{
% 6\sqrt{2}\,\beta
% }.
% \end{aligned}
% \right.
% \]
% Moreover, for $\beta\in\bigl[0,\tfrac{2\sqrt{3}}{3}\bigr]$,
% \[
% \lambda \xrightarrow{\mathrm{a.s.}} \lambda^\infty \le 2\sqrt{\frac{2}{3}},
% \qquad
% \bigl|\langle u,x\rangle\bigr|,
% \ \bigl|\langle v,y\rangle\bigr|,
% \ \bigl|\langle w,z\rangle\bigr|
% \xrightarrow{\mathrm{a.s.}} 0 .
% \]
% \end{corollary}

% \begin{proof}
% See Appendix~\ref{appA.6}.
% \end{proof}
\subsection{Cubic third-order tensors: the balanced case $c_1=c_2=c_3=\tfrac13$}

We now specialize our analysis to the balanced setting in which all tensor
dimensions are equal.
In this case, symmetry implies that the three mode-wise alignments
$\langle \bm{u},\bm{x}\rangle$, $\langle \bm{v},\bm{y}\rangle$, and
$\langle \bm{w},\bm{z}\rangle$ converge almost surely to the same deterministic
limit.
As a consequence, the asymptotic behavior of both the singular value and the
alignments can be characterized explicitly as functions of the signal strength
$\beta$.

The following corollary, which is a direct consequence of
Theorem~\ref{thm3.4}, provides closed-form expressions for these limits and
reveals a sharp transition at a critical signal-to-noise ratio.

\begin{corollary}\label{cor3.2}
Under Assumption~\ref{assumption3.2} with $c_1=c_2=c_3=\tfrac13$, there exists a
critical threshold $\beta_s=\tfrac{2\sqrt{3}}{3}$ such that for $\beta>\beta_s$,
\[
\left\{
\begin{aligned}
\lambda &\xrightarrow{\mathrm{a.s.}}
\lambda^\infty(\beta)
\;\equiv\;
\sqrt{
\frac{\beta^2}{2}
+2
+\frac{\sqrt{3}\sqrt{(3\beta^2-4)^3}}{18\beta}
},\\[10pt]
\bigl|\langle \bm{u},\bm{x}\rangle\bigr|,
\ \bigl|\langle \bm{v},\bm{y}\rangle\bigr|,
\ \bigl|\langle \bm{w},\bm{z}\rangle\bigr|
&\xrightarrow{\mathrm{a.s.}}
\frac{
\sqrt{\,9\beta^2-12+\dfrac{\sqrt{3}\sqrt{(3\beta^2-4)^3}}{\beta}}
+
\sqrt{\,9\beta^2+36+\dfrac{\sqrt{3}\sqrt{(3\beta^2-4)^3}}{\beta}}
}{
6\sqrt{2}\,\beta
}.
\end{aligned}
\right.
\]
Moreover, for $\beta\in\bigl[0,\tfrac{2\sqrt{3}}{3}\bigr]$,
\[
\lambda \xrightarrow{\mathrm{a.s.}} \lambda^\infty \le 2\sqrt{\frac{2}{3}},
\qquad
\bigl|\langle \bm{u},\bm{x}\rangle\bigr|,
\ \bigl|\langle \bm{v},\bm{y}\rangle\bigr|,
\ \bigl|\langle \bm{w},\bm{z}\rangle\bigr|
\xrightarrow{\mathrm{a.s.}} 0 .
\]
\end{corollary}

\begin{proof}
See Appendix~\ref{appA.6}.
\end{proof}

\section{Generalization to arbitrary $d$-order tensors}
\label{sec:order-d}

% We now show that our approach can be generalized straightforwardly to the
% $d$-order spiked tensor model of Eq.~\ref{eq1.1}. Indeed, from Eq.~\ref{eq1.4}, the
% $\ell_2$-singular value $\lambda$ and vectors
% $u^{(1)},\ldots,u^{(d)}\in \mathbb S^{n_1-1}\times\cdots\times\mathbb S^{n_d-1}$,
% corresponding to the best rank-one approximation
% $\lambda u^{(1)}\otimes\cdots\otimes u^{(d)}$ of the $d$-order tensor $\textbf{T}$
% in Eq.~\ref{eq1.1}, satisfy the identities
% \begin{equation}
% \label{eq:26}
% \left\{
% \begin{aligned}
% &\textbf{T}\big(u^{(1)},\ldots,u^{(i-1)},\,\cdot,\,u^{(i+1)},\ldots,u^{(d)}\big)
% = \lambda u^{(i)},\\
% &\lambda
% = \textbf{T}\big(u^{(1)},\ldots,u^{(d)}\big)
% = \sum_{i_1,\ldots,i_d} u^{(1)}_{i_1}\cdots u^{(d)}_{i_d}\,T_{i_1,\ldots,i_d}.
% \end{aligned}
% \right.
% \end{equation}

We now extend the analysis to the general $d$-order spiked tensor model defined
in Eq.~(\ref{eq1.1}).
The extension relies on the observation that the variational formulation
in Eq.~(\ref{eq1.4}) and the associated KKT conditions preserve the same structural
form for any $d\ge 3$.
In particular, the $\ell_2$-singular value $\lambda$ and the singular vectors
$\bm u^{(1)},\ldots,\bm u^{(d)}\in
\mathbb S^{n_1-1}\times\cdots\times\mathbb S^{n_d-1}$,
associated with the best rank-one approximation
$\lambda \bm u^{(1)}\otimes\cdots\otimes \bm u^{(d)}$ of the tensor $\bf T$,
satisfy the system of identities
\begin{equation}
\label{eq:26}
\left\{
\begin{aligned}
&\textbf{T}\big(\bm u^{(1)},\ldots,\bm u^{(i-1)},\,\cdot,\,\bm u^{(i+1)},\ldots,\bm u^{(d)}\big)
= \lambda \bm u^{(i)},\\
&\lambda
= \textbf{T}\big(\bm u^{(1)},\ldots,\bm u^{(d)}\big)
= \sum_{i_1,\ldots,i_d} u^{(1)}_{i_1}\cdots u^{(d)}_{i_d}\,T_{i_1,\ldots,i_d}.
\end{aligned}
\right.
\end{equation}

\subsection{Associated random matrix ensemble}

To make the structure of the KKT system amenable to spectral
analysis, we introduce a family of matrices obtained by contracting the tensor
$\textbf{T}$ along all but two modes.
For $i\neq j$, let $\textbf{T}^{ij}\in\mathbb M_{n_i,n_j}$ denote the matrix obtained
by contracting $\textbf{T}$ with the singular vectors
$\{\bm u^{(1)},\ldots,\bm u^{(d)}\}\setminus\{\bm u^{(i)},\bm u^{(j)}\}$, namely
\begin{equation}
\label{eq:27}
\textbf{T}^{ij}
=
\textbf{T}\big(\bm u^{(1)},\ldots,\bm u^{(i-1)},\,\cdot,\,\bm u^{(i+1)},\ldots,
\bm u^{(j-1)},\,\cdot,\,\bm u^{(j+1)},\ldots,\bm u^{(d)}\big).
\end{equation}

Differentiating the KKT identities~(\ref{eq:26}) with respect to an entry
$W_{i_1,\ldots,i_d}$ of the noise tensor $\bf W$ yields explicit expressions
for the sensitivities of the singular vectors.
These expressions reveal that the derivatives are governed by the inverse of a
block-structured matrix whose off-diagonal blocks consist of pairwise tensor
contractions.
Concretely, the derivatives of
$u^{(1)},\ldots,u^{(d)}$ take the form
\begin{equation}
\label{eq4.3}
\begin{bmatrix}
\dfrac{\partial \bm u^{(1)}}{\partial W_{i_1,\ldots,i_d}}\\[4pt]
\vdots\\[2pt]
\dfrac{\partial \bm u^{(d)}}{\partial W_{i_1,\ldots,i_d}}
\end{bmatrix}
=
-\frac{1}{\sqrt N}
\bm{R}(\lambda)
\begin{bmatrix}
\Pi_{\ell\in\{2,\ldots,d\}} u^{(\ell)}_{i_\ell}
\big(e^{n_1}_{i_1}-u^{(1)}_{i_1}u^{(1)}\big)\\
\vdots\\
\Pi_{\ell\in\{1,\ldots,d-1\}} u^{(\ell)}_{i_\ell}
\big(e^{n_d}_{i_d}-u^{(d)}_{i_d}u^{(d)}\big)
\end{bmatrix},
\end{equation}
where
\[ \bm{R}(z) = \left(
\begin{bmatrix}
\textbf{0}_{n_1\times n_1} & \textbf{T}^{12} & \textbf{T}^{13} & \cdots & \textbf{T}^{1d}\\
(\textbf{T}^{12})^\top & \textbf{0}_{n_2\times n_2} & \textbf{T}^{23} & \cdots & \textbf{T}^{2d}\\
(\textbf{T}^{13})^\top & (\textbf{T}^{23})^\top & \textbf{0}_{n_3\times n_3} & \cdots & \textbf{T}^{3d}\\
\vdots & \vdots & \vdots & \ddots & \vdots\\
(\textbf{T}^{1d})^\top & (\textbf{T}^{2d})^\top & (\textbf{T}^{3d})^\top & \cdots & \textbf{0}_{n_d\times n_d}
\end{bmatrix}
-z \bm I_N
\right)^{-1}, \] 
$N=\sum_{i\in[d]} n_i$, and the derivative of $\lambda$ with respect to
$W_{i_1,\ldots,i_d}$ is written as
\begin{equation}
\label{eq:29}
\frac{\partial \lambda}{\partial W_{i_1,\ldots,i_d}}
=
\frac{1}{\sqrt N}\prod_{\ell\in[d]} u^{(\ell)}_{i_\ell}.
\end{equation}

As such, the associated random matrix model of $\textbf{T}$ is the matrix
appearing in the resolvent in Eq.~(\ref{eq4.3}). More generally, the $d$-order
block-wise tensor contraction ensemble $\mathcal B_d(\bf W)$ for
$\bm W$ satisfying Assumption~\ref{assumption3.2} is defined as
\begin{equation}
\label{eq:30}
\mathcal B_d(\textbf{W})
=
\Big\{
\Phi_d(\textbf{W},\bm a^{(1)},\ldots,\bm a^{(d)})
\ \big|\ 
(\bm a^{(1)},\ldots,\bm a^{(d)})\in
\mathbb S^{n_1-1}\times\cdots\times\mathbb S^{n_d-1}
\Big\}.
\end{equation}
%where $\Phi_d$ is the mapping
%\begin{equation}
%\label{eq:31}
%\Phi_d:\ 
%\textbf{T}_{n_1,\ldots,n_d}\times
%\mathbb S^{n_1-1}\times\cdots\times\mathbb S^{n_d-1}
%\longrightarrow
%\mathbb M_{\sum_i n_i},
%\end{equation}
%defined by
%\[
%(\textbf{W},\bm a^{(1)},\ldots,\bm a^{(d)})
%\longmapsto
%\begin{bmatrix}
%\textbf{0}_{n_1\times n_1} & \textbf{W}^{12} & \textbf{W}^{13} & \cdots & \textbf{W}^{1d}\\
%(\textbf{W}^{12})^\top & \textbf{0}_{n_2\times n_2} & \textbf{W}^{23} & \cdots & \textbf{W}^{2d}\\
%(\textbf{W}^{13})^\top & (\textbf{W}^{23})^\top & \textbf{0}_{n_3\times n_3} & \cdots & \textbf{W}^{3d}\\
%\vdots & \vdots & \vdots & \ddots & \vdots\\
%(\textbf{W}^{1d})^\top & (\textbf{W}^{2d})^\top & (\textbf{W}^{3d})^\top & \cdots & \textbf{0}_{n_d\times n_d}
%\end{bmatrix},
%\] 
%with
%\[
%\textbf{W}^{ij}= \textbf{W}\big(\bm a^{(1)},\ldots,\bm a^{(i-1)},\,\cdot,\,
%\bm a^{(i+1)},\ldots,\bm a^{(j-1)},\,\cdot,\, \bm a^{(j+1)},\ldots,\bm a^{(d)}\big) \in\mathbb M_{n_i,n_j}.
%\]

\begin{remark}
As in the order-$3$ case, the existence of the inverse in Eq.~(\ref{eq4.3})
requires that the spectral parameter $\lambda_\ast$ does not belong to the
spectrum of
$\Phi_d(\textbf{T},u_\ast^{(1)},\ldots,u_\ast^{(d)})$.
We therefore assume the existence of a tuple
$(\lambda_\ast,u_\ast^{(1)},\ldots,u_\ast^{(d)})$ satisfying the identities in
Eq.~(\ref{eq:26}) for which this non-degeneracy condition holds.
\end{remark}

\subsection{Limiting spectral measure of block-wise $d$-order tensor contractions}
\label{sec:5.2}

% In this section, we characterize the limiting spectral measure of the ensemble
% $\mathcal B_d(\bm W)$ for
% $\bm W\sim\mathbb T_{n_1,\ldots,n_d}(\bm W)$ with \(\bm W\) satisfies \ref{assumption2}
% in the limit when all tensor dimensions grow, as specified by the following
% assumption.
We next characterize the limiting spectral distribution of the ensemble
$\mathcal B_d(\textbf{W})$ when the contraction directions are deterministic.
This result provides the reference spectral law against which the spiked setting
will later be compared.
The high-dimensional asymptotic regime is specified by the following assumption, which was also used in \cite{seddik2024whentrandomtensorsmeet}.

\begin{assumption}\label{assumption4.1}
For all $i\in[d]$, assume that $n_i\to\infty$ with
\[
\frac{n_i}{\sum_j n_j}\;\longrightarrow\; c_i\in(0,1).
\]
\end{assumption}

We thus have the following result, which characterizes the spectrum of
\[
\frac{1}{\sqrt N}\,
\Phi_d\bigl(\textbf{W},\bm a^{(1)},\ldots,\bm a^{(d)}\bigr)
\]
for any deterministic unit-norm vectors
$\bm a^{(1)},\ldots,\bm a^{(d)}$.

\begin{theorem}\label{thm4.1}
Let
$\textbf{W}\sim\textbf{T}_{n_1,\ldots,n_d}(\bm W)$ with \(\bm W\) satisfying Assumption~(\ref{eq:moment-assumption})
be a sequence of random asymmetric tensors and
\[
(\bm a^{(1)},\ldots,\bm a^{(d)})
\in
\mathbb S^{n_1-1}\times\cdots\times\mathbb S^{n_d-1}
\]
be a sequence of deterministic vectors of increasing dimensions following
Assumption~\ref{assumption4.1}.
Then the empirical spectral measure of
\[
\frac{1}{\sqrt N}\,
\Phi_d\bigl(\textbf{W},\bm a^{(1)},\ldots,\bm a^{(d)}\bigr)
\]
converges weakly almost surely to a deterministic probability measure $\nu$
whose Stieltjes transform $g(z)$ is defined as the solution to
\[
g(z)=\sum_{i=1}^d g_i(z),
\]
such that $\Im g(z)>0$ for $\Im z>0$. Here, for each $i\in[d]$ the function
$g_i(z)$ satisfies
\[
g_i^2(z)-\bigl(g(z)+z\bigr)g_i(z)-c_i=0,
\qquad
z\in\mathbb C\setminus\mathcal S(\nu).
\]
\end{theorem}

\begin{proof}
See Appendix~\ref{appA.7}.
\end{proof}

\begin{corollary}\label{cor4.1}
With the setting of Theorem~\ref{thm4.1} and under Assumption~\ref{assumption4.1} with
$c_i = \frac{1}{d}$ for all $i \in [d]$, then the empirical spectral measure of
\[
\frac{1}{\sqrt{N}}\,\Phi_d\!\bigl(\textbf{W}, \bm a^{(1)}, \ldots, \bm a^{(d)}\bigr)
\]
converges weakly almost surely to the semi-circle distribution with support
\[
\mathcal S(\nu)
\equiv
\left[
-2\sqrt{\frac{d-1}{d}},\;
\;2\sqrt{\frac{d-1}{d}}
\right],
\]
whose density and Stieltjes transform are given respectively by
\[
\nu(\mathrm d x)
=
\frac{d}{2(d-1)\pi}
\sqrt{
\left(
\frac{4(d-1)}{d}-x^2
\right)^{+}
},
\]
and
\[
g(z)
\equiv
\frac{-zd + d\sqrt{z^2-\frac{4(d-1)}{d}}}{2(d-1)},
\qquad
z \in \mathbb C \setminus \mathcal S(\nu).
\]
\end{corollary}

\begin{proof}
See Appendix~\ref{appA.8}.
\end{proof}

\begin{theorem}\label{thm4.2}
Let \(\textbf{T}\) be a sequence of spiked random tensors defined as in Eq.~(\ref{eq1.2}).
Under Assumptions~\ref{assumption4.1} and~\ref{assumption3.2}, the empirical spectral measure of
\(\Phi_d(\textbf{T}, u_\ast^{(1)}, \ldots, u_\ast^{(d)})\)
converges weakly almost surely to a deterministic measure \(\nu\) whose
Stieltjes transform \(g(z)\) is defined as the solution to the equation
\[
g(z)=\sum_{i=1}^d g_i(z)
\]
such that \(\Im[g(z)]>0\) for \(\Im[z]>0\). Here, for \(i\in[d]\),
\(g_i(z)\) satisfies
\[
g_i^2(z)-(g(z)+z)g_i(z)-c_i=0
\quad \text{for } z\in\mathbb C\setminus \mathcal S(\nu).
\]
\end{theorem}
The proof is sketched in Section~\ref{sketch4.2}; a complete and detailed proof is given in Appendix \ref{appA.9}.
\subsection{Asymptotic singular value and alignments of hyper-rectangular tensors}

% Similarly to the 3-order case studied previously, when the dimensions of $\textbf{T}$
% grow large at a same rate, its singular value $\lambda$ and the corresponding
% alignments $\langle u^{(i)}, x^{(i)} \rangle$ for $i \in [d]$ concentrate almost surely
% around some deterministic quantities which we denote
Finally, concentration estimates imply that the singular value $\lambda$ and the
alignments $\langle \bm u^{(i)},\bm x^{(i)}\rangle$ admit deterministic almost sure limits
under Assumptions~\ref{assumption4.1} and~\ref{assumption3.2}.
We denote these limits by
\[
\lambda^\infty(\beta) \equiv \lim_{N \to \infty} \lambda
\quad \text{and} \quad
a_{x^{(i)}}^\infty(\beta) \equiv \lim_{N \to \infty} \bigl|\langle \bm u^{(i)}, \bm x^{(i)} \rangle\bigr|,
\]
respectively. Applying again Lemma \ref{lem2.3} to the identities in
Eq.~(\ref{eq:26}), we obtain the following theorem which characterizes the aforementioned
deterministic limits.

\begin{theorem}\label{thm4.3}
Recall the notation in Theorem~\ref{thm4.2}. Under Assumptions~\ref{assumption4.1} and~\ref{assumption3.2}, for $d \ge 3$,
there exists $\beta_s > 0$ such that for $\beta > \beta_s$,
\[
\begin{cases}
\lambda \xrightarrow{\text{a.s.}} \lambda^\infty(\beta), \\[1ex]
\bigl|\langle \bm x^{(i)}, \bm u^{(i)} \rangle\bigr|
\xrightarrow{\text{a.s.}} q_i\!\left(\lambda^\infty(\beta)\right),
\end{cases}
\]
where $q_i(z)$ is given by
\[
q_i(z) = \sqrt{1 - \frac{g_i^2(z)}{c_i}},
\]
and $\lambda^\infty(\beta)$ satisfies
\[
f\!\left(\lambda^\infty(\beta), \beta\right) = 0,
\quad \text{with} \quad
f(z,\beta) = z + g(z) - \beta \prod_{i=1}^d q_i(z).
\]

In addition, for $\beta \in [0,\beta_s]$, $\lambda^\infty$ is bounded (in particular,
when $c_i = \frac{1}{d}$ for all $i \in [d]$),
\[
\lambda \xrightarrow{\text{a.s.}} \lambda^\infty \le 2\sqrt{\frac{d-1}{d}},
\quad \text{and} \quad
\bigl|\langle \bm x^{(i)}, \bm u^{(i)} \rangle\bigr| \xrightarrow{\text{a.s.}} 0.
\]
\end{theorem}

\begin{proof}
See Appendix~\ref{appA.10}.
\end{proof}

\section{Extension to rank-$r$ spiked tensor models with orthogonal components}
\label{sec:rank-r}

In the previous sections, we focused on the rank-one asymmetric spiked tensor
model and established universality results for the spectral and statistical
behavior of the maximum likelihood estimator.
We now briefly discuss how these results extend to the rank-$r$ setting under
a natural orthogonality assumption on the signal components.

In particular, We consider the rank-$r$ spiked tensor model
\begin{equation}
\textbf{T}
=
\sum_{\ell=1}^r
\beta_\ell\,
\bm x^{(1)}_\ell \otimes \cdots \otimes \bm x^{(d)}_\ell
+
\frac{1}{\sqrt N}\,\textbf{W},
\label{eq:rankr_model}
\end{equation}
where $\beta_1 > \cdots > \beta_r > 0$ are signal-to-noise ratios,
$\bm x^{(i)}_\ell \in \mathbb S^{n_i-1}$ for each mode $i\in[d]$,
and the noise tensor $\textbf{W}$ has independent entries satisfying
$\mathbb E W_{i_1\ldots i_d}=0$,
$\mathbb E W_{i_1\ldots i_d}^2=1$, and
$\mathbb E|W_{i_1\ldots i_d}|^4<\infty$.
We assume that the signal components are mutually orthogonal across each mode,
namely,
\[
\langle \bm x^{(i)}_\ell,\bm x^{(i)}_{\ell'} \rangle = 0
\qquad
\text{for all } \ell\neq\ell',\; i\in[d].
\]

Under this assumption, the maximum likelihood estimator corresponds to the
best rank-$r$ approximation of $\textbf{T}$, which is obtained by solving
\begin{equation}
\min_{\lambda_\ell>0,\,
\bm u^{(i)}_\ell\in\mathbb S^{n_i-1}}
\left\|
\textbf{T}-
\sum_{\ell=1}^r
\lambda_\ell\,
\bm u^{(1)}_\ell\otimes\cdots\otimes\bm u^{(d)}_\ell
\right\|_F^2.
\label{eq:rankr_ML}
\end{equation}
By the uniqueness of orthogonal tensor decompositions, each estimated direction
$\bm u^{(i)}_\ell$ necessarily correlates with the corresponding true
signal direction $\bm x^{(i)}_\ell$.

To analyze the $\ell$-th component, we consider the block-wise tensor
contraction matrix
\[
\Phi_d\!\left(
\textbf{T},
\bm u^{(1)}_\ell,\ldots,\bm u^{(d)}_\ell
\right),
\]
which arises from the first-order optimality (KKT) conditions, in the same
manner as in the rank-one case.
Due to the orthogonality of the signal components and concentration in high
dimensions, for any $\ell'\neq\ell$ and any mode $i\in[d]$, it holds that
\[
\langle \bm u^{(i)}_\ell,\bm x^{(i)}_{\ell'} \rangle
\;\xrightarrow{\mathrm{a.s.}}\;0.
\]
As a consequence, we obtain the asymptotic equivalence
\begin{equation}
\bigl\|
\Phi_d(\textbf{T},\bm u^{(1)}_\ell,\ldots,\bm u^{(d)}_\ell)
-
\Phi_d(\textbf{T}_\ell,\bm u^{(1)}_\ell,\ldots,\bm u^{(d)}_\ell)
\bigr\|
\;\xrightarrow{\mathrm{a.s.}}\;0,
\label{eq:rankr_equiv}
\end{equation}
where
\[
\textbf{T}_\ell
=
\beta_\ell\,
\bm x^{(1)}_\ell\otimes\cdots\otimes\bm x^{(d)}_\ell
+
\frac{1}{\sqrt N}\,\textbf{W}.
\]

Therefore, the spectral behavior of
$\Phi_d(\textbf{T},\bm u^{(1)}_\ell,\ldots,\bm u^{(d)}_\ell)$
coincides asymptotically with that of the corresponding rank-one spiked tensor
model $\textbf{T}_\ell$.
In particular, the limiting spectral distribution, the emergence of an
outlying eigenvalue, and the asymptotic mode-wise alignments
$\alpha^\infty(\beta_\ell)$ are identical to those derived in the rank-one
setting, with the signal strength $\beta$ replaced by $\beta_\ell$.

In summary, under the orthogonality assumption, the rank-$r$ asymmetric spiked
tensor model exhibits a decoupled structure in the high-dimensional limit,
where each signal component behaves as an independent rank-one spike.
As a result, all the universality results established in the rank-one case extend
componentwise to the rank-$r$ setting.

\section{Proof Sketch for Main Theorems}
\label{sec:proof-sketch}

In this section, we present a proof outline for our main results: Theorem~\ref{thm3.3} and 
Theorem~\ref{thm4.2}. The detailed proofs are given in Appendices~\ref{appA.4} and~\ref{appA.9}, respectively.  

\subsection{Proof Sketch of Theorem~\ref{thm3.3}}\label{sketch3.2}

Fix $z\in\mathbb C^+$.  
Recall that
\[
\bm T=\beta\bm V\bm B\bm V^\top+\bm N,
\qquad
\bm N=\frac1{\sqrt N}\Phi_3(\textbf{W},\bm u,\bm v,\bm w),
\]
and denote
\(
\bm R(z)=(\bm T-z\bm I)^{-1}
\)
and
\(
\bm Q(z)=(\bm N-z\bm I)^{-1}.
\)

\subsubsection*{Step 1: Reduction}
By the Woodbury identity and the rank--three structure of $\bm V$,
\[
\frac1N\Tr\bm R(z)
=
\frac1N\Tr\bm Q(z)+\mathcal O(N^{-1}),
\]
uniformly for $z\in\mathbb C^+$. Hence $\bm T$ and $\bm N$ have the same limiting
spectral distribution.

\subsubsection*{Step 2: Resolvent identity and implicit terms}
Let $g_i(z)=\frac1N\Tr\bm Q^{ii}(z)$.  
By $\bm N\bm Q(z)-z\bm Q(z)=\bm I$ and block traces, we have
\begin{equation}\label{eq:sc-short}
\frac{1}{N\sqrt N}\sum_{i,j,k}
\mathbb E\!\left[W_{ijk}w_k Q^{12}_{ij}(z)\right]
+
\frac{1}{N\sqrt N}\sum_{i,j,k}
\mathbb E\!\left[W_{ijk}v_j Q^{13}_{ij}(z)\right]
-
z g_1(z)
=
c_1.
\end{equation}

A key difficulty here is that the singular vectors $(\bm u,\bm v,\bm w)$
depend on the noise tensor $\bm W$. Differentiating the resolvent therefore
produces \emph{implicit terms} of the form
\[
\bm O_{ijk}
:=
\Phi_3\!\left(
\textbf{W},
\frac{\partial \bm u}{\partial W_{ijk}},
\frac{\partial \bm v}{\partial W_{ijk}},
\frac{\partial \bm w}{\partial W_{ijk}}
\right),
\]
which have no \emph{a priori} control on the operator norm.

Our main technical input is a sharp operator-norm estimate for these
implicit matrices:
\begin{equation}\label{eq:O-bound-main}
\Biggl|
\frac{1}{N^2}
\sum_{i,j,k}
\mathbb E\!\left[
w_k\,
(\bm Q(z)\bm O_{ijk}\bm Q(z))^{12}_{ij}
\right]
\Biggr|
=
\mathcal O(N^{-1}),
\end{equation}
uniformly for $z$ in compact subsets of $\mathbb C^+$.
This bound crucially uses the structure of the eigenvector derivatives and
Frobenius-to-operator norm reductions; see Lemma~\ref{lemA.1} for the proof.

\subsubsection*{Step 3: Cumulant expansion and decomposition}
To evaluate the mixed expectations in \eqref{eq:sc-short}, we apply a
second--order cumulant expansion with respect to the tensor entries $W_{ijk}$.
This yields a decomposition into three qualitatively different contributions:
\begin{equation}\label{eq:cumulant-3parts}
\partial_{W_{ijk}}^2\!\bigl(w_k Q^{12}_{ij}(z)\bigr)
=
\mathrm{I}_{ijk}(z)
+
\mathrm{II}_{ijk}(z)
+
\mathrm{III}_{ijk}(z),
\end{equation}
corresponding respectively to
\begin{enumerate}
\item[(I)] the \emph{pure resolvent term}
$\,w_k\,\partial^2 Q^{12}_{ij}(z)$,
\item[(II)] the \emph{mixed derivative term}
$\,(\partial w_k)(\partial Q^{12}_{ij}(z))$,
\item[(III)] the \emph{pure eigenvector term}
$\,Q^{12}_{ij}(z)\,\partial^2 w_k$.
\end{enumerate}

Each part requires a different control mechanism:
\begin{itemize}
\item[(I)] is handled using resolvent identities and Ward--type bounds, yielding
\[
N^{-3/2}\sum_{i,j,k}\sup_{z\in K}
|\mathrm{I}_{ijk}(z)|
=\mathcal O(N^{-1/2});
\]
\item[(II)] exploits the orthogonal projections in the eigenvector derivatives
and ``pressed--back'' column bounds for the resolvent, giving
\[
N^{-3/2}\sum_{i,j,k}\sup_{z\in K}
|\mathrm{II}_{ijk}(z)|
=\mathcal O(N^{-3/2});
\]
\item[(III)] relies on sharp $\ell^2$--bounds for second--order eigenvector
derivatives together with Frobenius--norm resolvent estimates, and satisfies
\[
N^{-3/2}\sum_{i,j,k}\sup_{z\in K}
|\mathrm{III}_{ijk}(z)|
=\mathcal O(N^{-1/2}),
\]
\end{itemize}
uniformly for $z$ in compact subsets $K\subset\mathbb C^+$.

All three contributions are therefore negligible in the large--$N$ limit.
Full proofs are given in Appendices~\ref{secA.4.1}, \ref{secA.4.2} and \ref{secA.4.2}.

\subsubsection*{Step 4: Evaluation of mixed expectations}
With \eqref{eq:O-bound-main}, a second--order cumulant expansion yields
\begin{align*}
\frac{1}{N\sqrt N}\sum_{i,j,k}
\mathbb E\!\left[W_{ijk}w_k Q^{12}_{ij}(z)\right]
&\longrightarrow -\,g_1(z)g_2(z),\\
\frac{1}{N\sqrt N}\sum_{i,j,k}
\mathbb E\!\left[W_{ijk}v_j Q^{13}_{ij}(z)\right]
&\longrightarrow -\,g_1(z)g_3(z),
\end{align*}
uniformly in compact subsets of $\mathbb C^+$.

\subsubsection*{Conclusion}
Substituting into \eqref{eq:sc-short} gives
\[
\begin{cases}
-\,g_1(z)\bigl(g_2(z)+g_3(z)\bigr)-z g_1(z)=c_1,\\
-\,g_2(z)\bigl(g_1(z)+g_3(z)\bigr)-z g_2(z)=c_2,\\
-\,g_3(z)\bigl(g_1(z)+g_2(z)\bigr)-z g_3(z)=c_3,
\end{cases}
\]
whose unique solution in $(\mathbb C^+)^3$ is
\[
g_i(z)
=
\frac{g(z)+z-\sqrt{(g(z)+z)^2+4c_i}}{2},
\qquad
g(z)=\sum_{i=1}^3 g_i(z),
\]
with the branch chosen so that $\Im g_i(z)>0$ for $\Im z>0$.

\subsection{Proof Sketch of Theorem~\ref{thm4.2}}\label{sketch4.2}

By the Woodbury matrix identity, the resolvent of
\(
\bm T=\beta \bm V\bm B\bm V^\top+\bm N
\)
can be expressed in terms of the resolvent
\(
\bm Q(z)=(\bm N-z\bm I)^{-1}
\)
up to a finite-rank correction. Since
\(
\|\bm Q(z)\|\le(\Im z)^{-1}
\)
for \(z\in\mathbb C^+_{\eta_0}\), the perturbative term has bounded operator norm.
Hence, the spectral characterization of \(\bm T\) reduces to estimating
\(\frac1N\Tr\bm Q(z)\).

\subsubsection*{Step 1: Cumulant expansion and three-term decomposition}
To handle the statistical dependence between the tensor noise \(\bm W\) and the
singular vectors \(\{\bm u^{(k)}\}\), we apply a second-order cumulant expansion to
\[
\frac{1}{N\sqrt N}\sum
\bigl[\bm W^{1j}\bm Q^{1j}(z)^\top\bigr]_{i_1i_1}.
\]
This yields three contributions,
\[
A_{j1}+A_{j2}+\varepsilon^{(2)},
\]
corresponding respectively to derivatives acting on the resolvent,
derivatives acting on the singular vectors, and the cumulant remainder.

\subsubsection*{Step 2: Control of the implicit term}
The leading term \(A_{j1}\) coincides with the independent-noise case and was
computed in Appendix~\ref{appA.7}.
The implicit term \(A_{j2}\) involves derivatives of the singular vectors and can
be written in the form
\[
(\bm u^{(1)})^\top \bm Q^{1j}(z)\bm u^{(j)}
\sum_{\ell\neq1,j}\Tr\bm R^{\ell\ell}(\lambda).
\]
Using the resolvent bound and the normalization \(\|\bm u^{(k)}\|_2=1\), we obtain
\(A_{j2}=\mathcal O(N^{-1})\), hence this contribution vanishes asymptotically.

\subsubsection*{Step 3: Estimation of the cumulant remainder}
The remainder term is controlled by bounding
\[
\frac{1}{N\sqrt N}\sum_{i_1,\ldots,i_d}
\sup_{z\in\mathbb C^+_{\eta_0}}
\bigl|D^2F(z)\bigr|,
\qquad
F(z)=Q^{1j}_{i_1i_j}(z)\prod_{k\neq1,j}u^{(k)}_{i_k}.
\]
Where 
\[
D := \frac{\partial}{\partial W_{i_1\cdots i_d}},
\qquad
U := \prod_{k\neq 1,j} u^{(k)}_{i_k}.
\]
Using resolvent identities, we decompose \(D^2F\) into three parts,
\[
S_1+S_2+S_3,
\]
corresponding to \(D^2\bm Q\), \((D\bm Q)(DU)\), and \(\bm Q(D^2U)\), respectively.

A key difference from the proof of Theorem~\ref{thm3.2} is that, in the general
order-\(d\) setting, one can no longer bound all factors uniformly in the
$\ell^2$-norm. Instead, a careful bookkeeping of the powers of each component
\(u^{(k)}_{i_k}\) is required, and Cauchy--Schwarz must be applied selectively at
the level of individual indices.

The first two terms satisfy \(S_1,S_2=\mathcal O(N)\). The most delicate term \(S_3\) is controlled by a resolvent--vector absorption
argument: at least one factor \(u^{(1)}_{i_1}\) or \(u^{(j)}_{i_j}\) is always present,
allowing us to absorb the free index via Cauchy--Schwarz and avoid an \(N^2\)
blow-up. This yields the sharper bound \(S_3=\mathcal O(\sqrt N)\). Full proofs are given in Appendices~\ref{secS1}, \ref{secS2} and \ref{secS3}.

\subsubsection*{Conclusion}
Combining the above estimates shows that the cumulant remainder is negligible and
that the derivative of \(\bm Q(z)\) has the same leading behavior as in the
independent case. Consequently,
\[
A_j=-g_1(z)g_j(z)+\mathcal O(N^{-1}),
\]
and the limiting Stieltjes transform coincides with that obtained in
Appendix~\ref{appA.7}, completing the proof of Theorem~\ref{thm4.2}.

\section{Concluding Remarks}
\label{sec:conclusion}
In this work, we established a universality principle for asymmetric rank-one spiked tensor models in the high-dimensional regime under independent noise with zero mean, unit variance, and finite fourth moment. Along a selected informative stationary branch of the maximum-likelihood landscape, we showed that the empirical spectral distribution of the associated block-wise tensor contraction coincides asymptotically with that of the Gaussian model. As a consequence, the singular value and the mode-wise alignments admit the same deterministic asymptotic characterization as in the Gaussian setting.

Our analysis shows that the key mechanism behind this universality phenomenon is the resolvent structure induced by the Karush--Kuhn--Tucker equations for the maximum-likelihood problem. By combining resolvent methods from random matrix theory, cumulant expansions, and Efron--Stein-type variance bounds, we are able to control the dependence between the estimator and the noise, including the cross terms that arise in the non-Gaussian setting. This both corrects and extends previous Gaussian-based analyses.

Several important questions remain open. First, our results rely on the branch-selection framework formalized in Assumptions~3.1 and~3.2. In the order-three asymmetric model, Proposition~3.3 verifies locally in the high-signal regime the qualitative content of Assumption~3.1, but a complete characterization of the optimization landscape for general asymmetric spiked tensor models remains open. In particular, it is not yet clear under what conditions the informative branch exists uniquely, or when a spike-aligned stationary point is globally optimal.

Second, determining the precise threshold that separates the uninformative and informative regimes remains an interesting problem for future work. Finally, it would be interesting to extend the present analysis to weaker moment assumptions and to algorithmic estimators beyond maximum likelihood.

Overall, our results show that the sharp asymptotic phenomena first identified in Gaussian spiked tensor models persist for a much broader class of non-Gaussian noise distributions. This provides further support for the robustness of Gaussian predictions in high-dimensional tensor inference.

\bibliography{ref}
\bibliographystyle{plain}

\newpage

\appendix
\clearpage
\begin{center}
{\Large\bfseries APPENDIX}
\end{center}
%\section{Proofs of Section~3} \label{appSec3}

\section{Proof of Proposition\ref{prop:high_signal_order3}}\label{appSec3}
\begin{proof}
    We split the argument into deterministic and probabilistic parts.

\medskip
\noindent
\textbf{Step 1: Local coordinates near the planted spike.}
Write
\[
\bm u=\alpha_x \bm x+\bm a,\qquad
\bm v=\alpha_y \bm y+\bm b,\qquad
\bm w=\alpha_z \bm z+\bm c,
\]
where
\[
\bm a\perp \bm x,\qquad \bm b\perp \bm y,\qquad \bm c\perp \bm z,
\]
and
\[
\alpha_x=\sqrt{1-\|\bm a\|^2},\qquad
\alpha_y=\sqrt{1-\|\bm b\|^2},\qquad
\alpha_z=\sqrt{1-\|\bm c\|^2}.
\]
We solve the KKT system in the local ball
\[
\mathcal B_r:=
\bigl\{
(\bm a,\bm b,\bm c): \bm a\perp \bm x,\ \bm b\perp \bm y,\ \bm c\perp \bm z,\ 
\|\bm a\|,\|\bm b\|,\|\bm c\|\le r/\beta
\bigr\},
\]
for a fixed constant \(r>0\) to be chosen later.

\medskip
\noindent
\textbf{Step 2: Noise contractions and the event \(\mathcal E_{L_0}\).}
Define the fixed vector contractions
\[
\bm \xi_u:=\frac1{\sqrt N}W(\cdot,\bm y,\bm z)\in\mathbb R^m,\qquad
\bm \xi_v:=\frac1{\sqrt N}W(\bm x,\cdot,\bm z)\in\mathbb R^n,\qquad
\bm \xi_w:=\frac1{\sqrt N}W(\bm x,\bm y,\cdot)\in\mathbb R^p.
\]
Define the one-mode contraction matrices
\[
M_u^{(z)}:=\frac1{\sqrt N}W(\cdot,\cdot,\bm z)\in \mathbb R^{m\times n},\qquad
M_u^{(y)}:=\frac1{\sqrt N}W(\cdot,\bm y,\cdot)\in \mathbb R^{m\times p},
\]
\[
M_v^{(z)}:=\frac1{\sqrt N}W(\cdot,\cdot,\bm z)\in \mathbb R^{m\times n},\qquad
M_v^{(x)}:=\frac1{\sqrt N}W(\bm x,\cdot,\cdot)\in \mathbb R^{n\times p},
\]
\[
M_w^{(y)}:=\frac1{\sqrt N}W(\cdot,\bm y,\cdot)\in \mathbb R^{m\times p},\qquad
M_w^{(x)}:=\frac1{\sqrt N}W(\bm x,\cdot,\cdot)\in \mathbb R^{n\times p}.
\]
Finally, denote by
\[
\|W\|_{\mathrm{inj}}
:=
\sup_{\|\bm r\|=\|\bm s\|=\|\bm t\|=1}
|W(\bm r,\bm s,\bm t)|
\]
the injective norm of \(W\).

Let \(\mathcal E_{L_0}\) be the event on which
\begin{align}
\|\bm \xi_u\|+\|\bm \xi_v\|+\|\bm \xi_w\| &\le L_0,
\label{eq:EL0_vec}\\
\|M_u^{(z)}\|+\|M_u^{(y)}\|+\|M_v^{(x)}\| &\le L_0,
\label{eq:EL0_mat}\\
\frac{\|W\|_{\mathrm{inj}}}{\sqrt N} &\le L_0 N^{1/4}.
\label{eq:EL0_inj}
\end{align}
The vector bounds in \eqref{eq:EL0_vec} follow from standard concentration for
fixed contractions, the matrix bounds in \eqref{eq:EL0_mat} follow from
Lemma~2.3 applied to one-mode contractions, and the tensor bound
\eqref{eq:EL0_inj} is the standard injective-norm estimate for third-order
random tensors.

In the remainder of the proof we work on \(\mathcal E_{L_0}\).

\medskip
\noindent
\textbf{Step 3: Expansion of the KKT system.}
The first KKT equation reads
\[
T(\bm v)\bm w=\lambda \bm u.
\]
Using
\[
T=\beta \bm x\otimes \bm y\otimes \bm z+\frac1{\sqrt N}W,
\]
we obtain
\[
T(\bm v)\bm w
=
\beta \langle \bm y,\bm v\rangle \langle \bm z,\bm w\rangle \bm x
+\frac1{\sqrt N}W(\cdot,\bm v,\bm w)
=
\beta\alpha_y\alpha_z \bm x+\frac1{\sqrt N}W(\cdot,\bm v,\bm w).
\]
Expanding the noise term and projecting onto \(\bm x^\perp\), we get
\begin{equation}
\label{eq:a_eq}
\lambda \bm a
=
\Pi_{\bm x^\perp}\Bigl(
\alpha_y\alpha_z \bm \xi_u
+\alpha_y\,M_u^{(y)}\bm c
+\alpha_z\,M_u^{(z)}\bm b
+R_u(\bm b,\bm c)
\Bigr),
\end{equation}
where
\[
R_u(\bm b,\bm c):=\frac1{\sqrt N}W(\cdot,\bm b,\bm c).
\]

Similarly,
\begin{equation}
\label{eq:b_eq}
\lambda \bm b
=
\Pi_{\bm y^\perp}\Bigl(
\alpha_x\alpha_z \bm \xi_v
+\alpha_x\,M_v^{(x)}\bm c
+\alpha_z\,(M_u^{(z)})^\top \bm a
+R_v(\bm a,\bm c)
\Bigr),
\end{equation}
\begin{equation}
\label{eq:c_eq}
\lambda \bm c
=
\Pi_{\bm z^\perp}\Bigl(
\alpha_x\alpha_y \bm \xi_w
+\alpha_x\,(M_v^{(x)})^\top \bm b
+\alpha_y\,(M_u^{(y)})^\top \bm a
+R_w(\bm a,\bm b)
\Bigr),
\end{equation}
where
\[
R_v(\bm a,\bm c):=\frac1{\sqrt N}W(\bm a,\cdot,\bm c),\qquad
R_w(\bm a,\bm b):=\frac1{\sqrt N}W(\bm a,\bm b,\cdot).
\]

The scalar equation for \(\lambda\) is
\begin{equation}
\label{eq:lambda_scalar}
\lambda
=
T(\bm u,\bm v,\bm w)
=
\beta\alpha_x\alpha_y\alpha_z
+\frac1{\sqrt N}W(\bm u,\bm v,\bm w).
\end{equation}

\medskip
\noindent
\textbf{Step 4: Bounds on the remainder terms.}
By the injective norm bound \eqref{eq:EL0_inj}, for every \(\bm r\in S^{m-1}\),
\[
|\langle \bm r,R_u(\bm b,\bm c)\rangle|
=
\left|
\frac1{\sqrt N}W(\bm r,\bm b,\bm c)
\right|
\le
L_0 N^{1/4}\|\bm b\|\,\|\bm c\|.
\]
Taking the supremum over \(\bm r\in S^{m-1}\), we obtain
\begin{equation}
\label{eq:Ru_bd}
\|R_u(\bm b,\bm c)\|
\le
L_0 N^{1/4}\|\bm b\|\,\|\bm c\|.
\end{equation}
Exactly the same argument yields
\begin{equation}
\label{eq:RvRw_bd}
\|R_v(\bm a,\bm c)\|
\le
L_0 N^{1/4}\|\bm a\|\,\|\bm c\|,
\qquad
\|R_w(\bm a,\bm b)\|
\le
L_0 N^{1/4}\|\bm a\|\,\|\bm b\|.
\end{equation}
Hence, on \(\mathcal B_r\),
\begin{equation}
\label{eq:remainder_ball_bd}
\|R_u(\bm b,\bm c)\|+\|R_v(\bm a,\bm c)\|+\|R_w(\bm a,\bm b)\|
\le
3L_0\,\frac{r^2 N^{1/4}}{\beta^2}.
\end{equation}

\medskip
\noindent
\textbf{Step 5: Fixed-point map.}
For \((\bm a,\bm b,\bm c)\in\mathcal B_r\), define
\[
\Lambda(\bm a,\bm b,\bm c)
:=
\beta\alpha_x\alpha_y\alpha_z+\frac1{\sqrt N}W(\bm u,\bm v,\bm w),
\]
and
\[
\Psi(\bm a,\bm b,\bm c)
=
(\Psi_u(\bm a,\bm b,\bm c),\Psi_v(\bm a,\bm b,\bm c),\Psi_w(\bm a,\bm b,\bm c))
\]
by
\[
\Psi_u(\bm a,\bm b,\bm c)
:=
\frac1{\Lambda(\bm a,\bm b,\bm c)}
\Pi_{\bm x^\perp}
\Bigl(
\alpha_y\alpha_z \bm \xi_u
+\alpha_y M_u^{(y)}\bm c
+\alpha_z M_u^{(z)}\bm b
+R_u(\bm b,\bm c)
\Bigr),
\]
\[
\Psi_v(\bm a,\bm b,\bm c)
:=
\frac1{\Lambda(\bm a,\bm b,\bm c)}
\Pi_{\bm y^\perp}
\Bigl(
\alpha_x\alpha_z \bm \xi_v
+\alpha_x M_v^{(x)}\bm c
+\alpha_z (M_u^{(z)})^\top \bm a
+R_v(\bm a,\bm c)
\Bigr),
\]
\[
\Psi_w(\bm a,\bm b,\bm c)
:=
\frac1{\Lambda(\bm a,\bm b,\bm c)}
\Pi_{\bm z^\perp}
\Bigl(
\alpha_x\alpha_y \bm \xi_w
+\alpha_x (M_v^{(x)})^\top \bm b
+\alpha_y (M_u^{(y)})^\top \bm a
+R_w(\bm a,\bm b)
\Bigr).
\]

\medskip
\noindent
\textbf{Step 6: \(\Psi\) maps \(\mathcal B_r\) into itself.}
From \eqref{eq:lambda_scalar}, \eqref{eq:EL0_inj}, and the bounds on
\(\mathcal B_r\), we have
\[
\Lambda(\bm a,\bm b,\bm c)
=
\beta+O(1)+O\!\left(\frac{r}{\beta}\right)
+O\!\left(\frac{r^2N^{1/4}}{\beta^2}\right).
\]
Hence, if \(\beta\ge C N^{1/4}\) with \(C\) sufficiently large, then
\begin{equation}
\label{eq:lambda_half}
|\Lambda(\bm a,\bm b,\bm c)|\ge \frac{\beta}{2}
\qquad\text{for all }(\bm a,\bm b,\bm c)\in\mathcal B_r.
\end{equation}

Using \eqref{eq:EL0_vec}, \eqref{eq:EL0_mat}, \eqref{eq:remainder_ball_bd}, and
\eqref{eq:lambda_half}, we obtain
\[
\|\Psi_u(\bm a,\bm b,\bm c)\|
\le
\frac{2L_0}{\beta}
\Bigl(
1+\frac{2r}{\beta}+\frac{r^2N^{1/4}}{\beta^2}
\Bigr),
\]
and similarly for \(\Psi_v,\Psi_w\). Choosing \(r=4L_0\) and taking
\(\beta\ge C N^{1/4}\) with \(C\) sufficiently large, we get
\[
\|\Psi_u(\bm a,\bm b,\bm c)\|,
\ \|\Psi_v(\bm a,\bm b,\bm c)\|,
\ \|\Psi_w(\bm a,\bm b,\bm c)\|
\le \frac{r}{\beta},
\]
hence \(\Psi(\mathcal B_r)\subset \mathcal B_r\).

\medskip
\noindent
\textbf{Step 7: Contraction property.}
Let \((\bm a,\bm b,\bm c),(\tilde{\bm a},\tilde{\bm b},\tilde{\bm c})\in\mathcal B_r\). Set
\[
\|(\bm a,\bm b,\bm c)\|_\ast:=\max\{\|\bm a\|,\|\bm b\|,\|\bm c\|\}.
\]
A direct mean-value estimate gives
\[
|\alpha_x-\tilde\alpha_x|+|\alpha_y-\tilde\alpha_y|+|\alpha_z-\tilde\alpha_z|
\le
C\frac{r}{\beta}\|(\bm a,\bm b,\bm c)-(\tilde{\bm a},\tilde{\bm b},\tilde{\bm c})\|_\ast.
\]
Using the bilinear structure of \(R_u,R_v,R_w\), together with
\eqref{eq:EL0_mat} and \eqref{eq:EL0_inj}, we obtain
\[
\|R_u(\bm b,\bm c)-R_u(\tilde{\bm b},\tilde{\bm c})\|
\le
2L_0\frac{rN^{1/4}}{\beta}
\|(\bm a,\bm b,\bm c)-(\tilde{\bm a},\tilde{\bm b},\tilde{\bm c})\|_\ast,
\]
and similarly for \(R_v,R_w\). Combining these bounds with
\eqref{eq:lambda_half}, we get
\[
\|\Psi(\bm a,\bm b,\bm c)-\Psi(\tilde{\bm a},\tilde{\bm b},\tilde{\bm c})\|_\ast
\le
\left(
\frac{C_1L_0}{\beta}
+
\frac{C_2L_0rN^{1/4}}{\beta^2}
\right)
\|(\bm a,\bm b,\bm c)-(\tilde{\bm a},\tilde{\bm b},\tilde{\bm c})\|_\ast.
\]
If \(\beta\ge C N^{1/4}\) with \(C\) sufficiently large, the factor in
parentheses is \(<1/2\). Hence \(\Psi\) is a contraction on \(\mathcal B_r\).

By Banach's fixed-point theorem, \(\Psi\) has a unique fixed point
\[
(\bm a_\ast,\bm b_\ast,\bm c_\ast)\in\mathcal B_r.
\]

\medskip
\noindent
\textbf{Step 8: Reconstruction of the stationary point.}
Define
\[
\bm u_\ast=\alpha_{x,\ast}\bm x+\bm a_\ast,\qquad
\bm v_\ast=\alpha_{y,\ast}\bm y+\bm b_\ast,\qquad
\bm w_\ast=\alpha_{z,\ast}\bm z+\bm c_\ast,
\]
and
\[
\lambda_\ast:=\Lambda(\bm a_\ast,\bm b_\ast,\bm c_\ast).
\]
The fixed-point equations imply the projected KKT system in the three orthogonal
complements \(\bm x^\perp,\bm y^\perp,\bm z^\perp\). The components along
\(\bm x,\bm y,\bm z\) follow from the scalar identity
\eqref{eq:lambda_scalar}. Therefore
\[
T(\bm v_\ast)\bm w_\ast=\lambda_\ast \bm u_\ast,\qquad
T(\bm u_\ast)\bm w_\ast=\lambda_\ast \bm v_\ast,\qquad
T(\bm v_\ast)^\top \bm u_\ast=\lambda_\ast \bm w_\ast.
\]

\medskip
\noindent
\textbf{Step 9: Quantitative estimates.}
Since \((\bm a_\ast,\bm b_\ast,\bm c_\ast)\in\mathcal B_r\),
\[
\|\bm a_\ast\|+\|\bm b_\ast\|+\|\bm c_\ast\|\le \frac{3r}{\beta},
\]
which proves \eqref{eq:order3_close}. From \eqref{eq:lambda_scalar},
\[
|\lambda_\ast-\beta|
\le
\beta|1-\alpha_{x,\ast}\alpha_{y,\ast}\alpha_{z,\ast}|
+\left|\frac1{\sqrt N}W(\bm u_\ast,\bm v_\ast,\bm w_\ast)\right|
\le C,
\]
which proves \eqref{eq:order3_lambda_close}. Finally,
\[
\langle \bm u_\ast,\bm x\rangle=\alpha_{x,\ast}
=\sqrt{1-\|\bm a_\ast\|^2}
\ge 1-\|\bm a_\ast\|^2 \ge 1-\frac{C}{\beta^2},
\]
and similarly for \(\bm v_\ast,\bm w_\ast\), proving \eqref{eq:order3_align}.

\medskip
\noindent
\textbf{Step 10: Outlier and regularity.}
By \eqref{eq:order3_lambda_close},
\[
\lambda_\ast=\beta+O(1).
\]
If \(K_\nu\) is a deterministic compact bulk set containing the spectrum of
\(\Phi_3(T,\bm u_\ast,\bm v_\ast,\bm w_\ast)\) with high probability, then \(K_\nu\) is
bounded uniformly in \(N\). Therefore, for sufficiently large \(\beta\),
\[
\mathrm{dist}(\lambda_\ast,K_\nu)\ge c\beta,
\]
for some \(c>0\). This proves \eqref{eq:order3_outlier}, and
\eqref{eq:order3_regular} follows immediately.
\end{proof}
\section{Proofs of Section~4} \label{appA}
\subsection{Derivative of tensor singular value and vectors}\label{appA.1}
The following proof comes from \cite{seddik2024whentrandomtensorsmeet}.
Deriving the identities in Eqs.~(\ref{eq3.2}) and (\ref{eq3.3}) with respect to the entry
$W_{ijk}$ of the tensor noise $\textbf{W}$, we obtain the following system:
\begin{equation}
\left\{
\begin{aligned}
\textbf{T}(\bm{w}) \frac{\partial \bm{v}}{\partial W_{ijk}}
+ \textbf{T}(\bm{v}) \frac{\partial \bm{w}}{\partial W_{ijk}}
+ \frac{1}{\sqrt{N}} v_j w_k \bm{e}_i^{\,m}
&= \frac{\partial \lambda}{\partial W_{ijk}} \bm{u}
+ \lambda \frac{\partial \bm{u}}{\partial W_{ijk}}, \\[4pt]
\textbf{T}(\bm{w})^{\top} \frac{\partial \bm{u}}{\partial W_{ijk}}
+ \textbf{T}(\bm{u}) \frac{\partial \bm{w}}{\partial W_{ijk}}
+ \frac{1}{\sqrt{N}} u_i w_k \bm{e}_j^{\,n}
&= \frac{\partial \lambda}{\partial W_{ijk}} \bm{v}
+ \lambda \frac{\partial \bm{v}}{\partial W_{ijk}}, \\[4pt]
\textbf{T}(\bm{v})^{\top} \frac{\partial \bm{u}}{\partial W_{ijk}}
+ \textbf{T}(\bm{u})^{\top} \frac{\partial \bm{v}}{\partial W_{ijk}}
+ \frac{1}{\sqrt{N}} u_i v_j \bm{e}_k^{\,p}
&= \frac{\partial \lambda}{\partial W_{ijk}} \bm{w}
+ \lambda \frac{\partial \bm{w}}{\partial W_{ijk}}, \\[4pt]
\frac{\partial \lambda}{\partial W_{ijk}}
= \textbf{T}\!\left(\frac{\partial \bm{u}}{\partial W_{ijk}},
\bm{v}, \bm{w}\right)
+ \textbf{T}\!\left(\bm{u},\frac{\partial \bm{v}}{\partial W_{ijk}}, \bm{w}\right)
&+ \textbf{T}\!\left(\bm{u}, \bm{v},
\frac{\partial \bm{w}}{\partial W_{ijk}}\right)
+ \frac{1}{\sqrt{N}} u_i v_j w_k .
\end{aligned}
\right.
\end{equation}

Writing
\[
\textbf{T}\!\left(\frac{\partial \bm{u}}{\partial W_{ijk}},
\bm{v}, \bm{w}\right)
=
\left(\frac{\partial \bm{u}}{\partial W_{ijk}}\right)^{\!\top}
\textbf{T}(\bm{v}) \bm{w},
\]
and using again the identities in Eq.~(11), we obtain
\[
\textbf{T}\!\left(\frac{\partial \bm{u}}{\partial W_{ijk}},
\bm{v}, \bm{w}\right)
=
\lambda \left(\frac{\partial \bm{u}}{\partial W_{ijk}}\right)^{\!\top}
\bm{u}.
\]
Proceeding similarly for
$\textbf{T}(\bm{u}, \frac{\partial \bm{v}}{\partial W_{ijk}}, \bm{w})$
and
$\textbf{T}(\bm{u}, \bm{v}, \frac{\partial \bm{w}}{\partial W_{ijk}})$,
we obtain
\begin{equation}
\frac{\partial \lambda}{\partial W_{ijk}}
=
\lambda \left(
\left(\frac{\partial \bm{u}}{\partial W_{ijk}}\right)^{\!\top} \bm{u}
+
\left(\frac{\partial \bm{v}}{\partial W_{ijk}}\right)^{\!\top} \bm{v}
+
\left(\frac{\partial \bm{w}}{\partial W_{ijk}}\right)^{\!\top} \bm{w}
\right)
+ \frac{1}{\sqrt{N}} u_i v_j w_k .
\end{equation}

Furthermore, since
$\bm{u}^{\top} \bm{u}
= \bm{v}^{\top} \bm{v}
= \bm{w}^{\top} \bm{w} = 1$,
we have
\[
\left(\frac{\partial \bm{u}}{\partial W_{ijk}}\right)^{\!\top} \bm{u}
=
\left(\frac{\partial \bm{v}}{\partial W_{ijk}}\right)^{\!\top} \bm{v}
=
\left(\frac{\partial \bm{w}}{\partial W_{ijk}}\right)^{\!\top} \bm{w}
= 0 .
\]

Thus, the derivative of $\lambda$ simplifies to
\begin{equation}
\frac{\partial \lambda}{\partial W_{ijk}}
= \frac{1}{\sqrt{N}} u_i v_j w_k .
\end{equation}

Hence, we find that
\begin{equation}
\lambda
\begin{bmatrix}
\dfrac{\partial \bm{u}}{\partial W_{ijk}} \\[4pt]
\dfrac{\partial \bm{v}}{\partial W_{ijk}} \\[4pt]
\dfrac{\partial \bm{w}}{\partial W_{ijk}}
\end{bmatrix}
=
\frac{1}{\sqrt{N}}
\begin{bmatrix}
v_j w_k (\bm{e}_i^{\,m} - u_i \bm{u}) \\[4pt]
u_i w_k (\bm{e}_j^{\,n} - v_j \bm{v}) \\[4pt]
u_i v_j (\bm{e}_k^{\,p} - w_k \bm{w})
\end{bmatrix}
+
\Phi_3(\textbf{T}, \bm{u}, \bm{v}, \bm{w})
\begin{bmatrix}
\dfrac{\partial \bm{u}}{\partial W_{ijk}} \\[4pt]
\dfrac{\partial \bm{v}}{\partial W_{ijk}} \\[4pt]
\dfrac{\partial \bm{w}}{\partial W_{ijk}}
\end{bmatrix}.
\end{equation}

\subsection{Proof of Theorem \ref{thm3.2}}\label{appA.2}
Firstly we assume \(z \in \mathbb C^+_{\eta_0}\)(see (\ref{ccc})).
Denote the matrix model as
\begin{equation}
    \bm N \equiv \frac{1}{\sqrt{N}}\Phi_3(\textbf{W}, \bm{u}, \bm{v}, \bm{w}),
\end{equation}
where we recall that $\textbf{W} \sim \textbf{T}_{m,n,p}(\bm W)$ and
\(
(\bm{u}, \bm{v}, \bm{w})
\in \mathbb{S}^{m-1} \times \mathbb{S}^{n-1} \times \mathbb{S}^{p-1}
\)
are independent of $\bm{W}$.
We further denote the resolvent matrix of $\bm{N}$ as
\[
\bm{Q}(z)
\equiv (\bm{N} - z \bm{I}_N)^{-1}
=
\begin{bmatrix}
\bm{Q}^{11}(z) & \bm{Q}^{12}(z) & \bm{Q}^{13}(z) \\
\bm{Q}^{12}(z)^{\top} & \bm{Q}^{22}(z) & \bm{Q}^{23}(z) \\
\bm{Q}^{13}(z)^{\top} & \bm{Q}^{23}(z)^{\top} & \bm{Q}^{33}(z)
\end{bmatrix}.
\]

In order to characterize the limiting Stieltjes transform $g(z)$ of $\bm{N}$,
we need to estimate the quantity
\[
\frac{1}{N} \operatorname{tr} \bm{Q}(z)
\xrightarrow{\text{a.s.}} g(z).
\]
We further introduce the following limits:
\[
\frac{1}{N} \operatorname{tr} \bm{Q}^{11}(z)
\xrightarrow{\text{a.s.}} g_1(z),
\qquad
\frac{1}{N} \operatorname{tr} \bm{Q}^{22}(z)
\xrightarrow{\text{a.s.}} g_2(z),
\qquad
\frac{1}{N} \operatorname{tr} \bm{Q}^{33}(z)
\xrightarrow{\text{a.s.}} g_3(z).
\]

From the identity in Eq.~(7), we have
\[
\bm{N}\bm{Q}(z) - z \bm{Q}(z) = \bm{I}_N,
\]
from which we particularly obtain
\[
\frac{1}{\sqrt{N}}
\bigl[\textbf{W}(\bm{w}) \bm{Q}^{12}(z)^{\top}\bigr]_{ii}
+
\frac{1}{\sqrt{N}}
\bigl[\textbf{W}(\bm{v}) \bm{Q}^{13}(z)^{\top}\bigr]_{ii}
-
z\, \bm{Q}^{11}_{ii}(z)
= 1,
\]
or equivalently,
\begin{equation}
\frac{1}{N\sqrt{N}}
\sum_{i=1}^{m}
\bigl[\textbf{W}(\bm{w}) \bm{Q}^{12}(z)^{\top}\bigr]_{ii}
+
\frac{1}{N\sqrt{N}}
\sum_{i=1}^{m}
\bigl[\textbf{W}(\bm{v}) \bm{Q}^{13}(z)^{\top}\bigr]_{ii}
-
\frac{z}{N} \operatorname{tr} \bm{Q}^{11}(z)
=
\frac{m}{N}.
\label{eqA.6}
\end{equation}

We thus need to compute the expectations of
\[
\frac{1}{N\sqrt{N}}
\sum_{i=1}^{m}
\bigl[\textbf{W}(\bm{w}) \bm{Q}^{12}(z)^{\top}\bigr]_{ii},
\qquad
\frac{1}{N\sqrt{N}}
\sum_{i=1}^{m}
\bigl[\textbf{W}(\bm{v}) \bm{Q}^{13}(z)^{\top}\bigr]_{ii}.
\]

We first rewrite the quantity of interest as an explicit sum over the tensor
entries. Note that
\[
\bigl[\textbf{W}(\bm{w}) \bm{Q}^{12}(z)^{\top}\bigr]_{ii}
=
\sum_{j=1}^{n}\sum_{k=1}^{p}
W_{ijk}\, w_k \, \bm{Q}^{12}_{ij}(z).
\]
Therefore,
\begin{align}
A_1(z)
&:= \frac{1}{N\sqrt{N}}
\sum_{i=1}^{m}
\bigl[\textbf{W}(\bm{w}) \bm{Q}^{12}(z)^{\top}\bigr]_{ii} \notag \\
&= \frac{1}{N\sqrt{N}}
\sum_{i=1}^{m}\sum_{j=1}^{n}\sum_{k=1}^{p}
w_k\, W_{ijk}\, \bm{Q}^{12}_{ij}(z).
\end{align}
Taking expectation yields
\begin{equation}
\mathbb{E} A_1(z)
=
\frac{1}{N\sqrt{N}}
\sum_{i,j,k}
w_k\, \mathbb{E}\!\left[
W_{ijk}\, \bm{Q}^{12}_{ij}(z)
\right].
\end{equation}

Fix indices $(i,j,k)$ and denote $\xi := W_{ijk}$. Let
\[
g(\xi) := \bm{Q}^{12}_{ij}(z),
\]
where all other tensor entries are kept fixed. Applying the cumulant expansion,
we obtain
\begin{equation}
\mathbb{E}\!\left[
W_{ijk}\, \bm{Q}^{12}_{ij}(z)
\right]
=
\mathbb{E}\!\left[
\frac{\partial}{\partial W_{ijk}}
\bm{Q}^{12}_{ij}(z)
\right]
+ \varepsilon_{ijk}^{(2)}(z),
\end{equation}
where \(\varepsilon_{ijk}^{(2)}(z) \leq C_1 \, \sup_{z \in \mathbb C^+_{\eta_0}} |g^{(2)}(z)| \, \mathbb{E}[|W_{ijk}|^{3}] = C_1^* \sup_{z \in \mathbb C^+_{\eta_0}} |\frac{\partial^2}{(\partial W_{ijk})^2}\bm{Q}_{ij}^{12}(z)|\).

Firstly we consider the first derivation. Using the resolvent derivative identity,
\[
\frac{\partial \bm{Q}(z)}{\partial W_{ijk}}
=
-\bm{Q}(z)\Bigl(\frac{\partial \bm{N}}{\partial W_{ijk}}\Bigr)\bm{Q}(z),
\]
and the block structure of $\bm{N}$, we have
\[
\Bigl(\frac{\partial \bm{N}}{\partial W_{ijk}}\Bigr)^{12}
=
\frac{1}{\sqrt{N}}\, w_k\, \bm{e}^{\,m}_i(\bm{e}^{\,n}_j)^{\top},
\qquad
\Bigl(\frac{\partial \bm{N}}{\partial W_{ijk}}\Bigr)^{21}
=
\frac{1}{\sqrt{N}}\, w_k\, \bm{e}^{\,n}_j(\bm{e}^{\,m}_i)^{\top},
\]
while the remaining nonzero blocks correspond to $(1,3)$ and $(2,3)$ and will be
grouped into a remainder term.

Consequently,
\begin{align}
\frac{\partial \bm{Q}^{12}_{ij}}{\partial W_{ijk}}
=
-\frac{1}{\sqrt{N}}
\Big[
u_i\!\left( Q^{12}_{ij}Q^{23}_{jk}
+ Q^{13}_{ik}Q^{22}_{jj} \right)
+ v_j\!\left( Q^{11}_{ii}Q^{23}_{jk}
+ Q^{13}_{ik}Q^{12}_{ij} \right)
+ w_k\!\left( Q^{11}_{ii}Q^{22}_{jj}
+ Q^{12}_{ij}Q^{12}_{ij} \right)
\Big].
\label{eq:dQ12}
\end{align}

So we have
\[
\begin{aligned}
&\mathbb{E}\!\left[\frac{\partial}{\partial W_{ijk}}\bm{Q}^{12}_{ij}(z)\right]\\
=&
-\frac{1}{N^2}
\sum_{i,j,k}
w_k\,
\mathbb{E}\!\Big[
u_i\!\left(
Q^{12}_{ij}Q^{23}_{jk}
+ Q^{13}_{ik}Q^{22}_{jj}
\right)
+ v_j\!\left(
Q^{11}_{ii}Q^{23}_{jk}
+ Q^{13}_{ik}Q^{12}_{ij}
\right)
+ w_k\!\left(
Q^{11}_{ii}Q^{22}_{jj}
+ Q^{12}_{ij}Q^{12}_{ij}
\right)
\Big],\\
=&
-\frac{1}{N^2}\,
\mathbb{E}\!\Big[
\bm{u}^{\top}\bm{Q}^{12}\bm{Q}^{23}\bm{w}
+ \bm{u}^{\top}\bm{Q}^{13}\bm{w}\,\operatorname{tr}\bm{Q}^{22}
+ \operatorname{tr}\bm{Q}^{11}\,\bm{v}^{\top}\bm{Q}^{23}\bm{w}
+ \bm{v}^{\top}(\bm{Q}^{12})^{\top}\bm{Q}^{13}\bm{w}\\
&+ \operatorname{tr}\bm{Q}^{11}\operatorname{tr}\bm{Q}^{22}
+ \operatorname{tr}\!\bigl(\bm{Q}^{12}(\bm{Q}^{12})^{\top}\bigr)
\Big].
\end{aligned}
\]

Now, note that the vectors $\bm{u}, \bm{v}, \bm{w}$ are of bounded norms and assume
that $\bm{Q}(z)$ has bounded spectral norm since \(\|\bm{Q}(z)\| \le \frac{1}{dist(z,  \mathcal{S}(\bm{N}))}< \eta_0\).
Under Assumption~\ref{assumption3.2} (as $N\to\infty$), the terms
\[
\frac{1}{N^2}\,\bm{u}^{\top}\bm{Q}^{12}\bm{Q}^{23}\bm{w},\quad
\frac{1}{N^2}\,\bm{u}^{\top}\bm{Q}^{13}\bm{w}\,\operatorname{tr}\bm{Q}^{22},\quad
\frac{1}{N^2}\,\operatorname{tr}\bm{Q}^{11}\,\bm{v}^{\top}\bm{Q}^{23}\bm{w},
\]
\[
\frac{1}{N^2}\,\bm{v}^{\top}(\bm{Q}^{12})^{\top}\bm{Q}^{13}\bm{w},
\quad\text{and}\quad
\frac{1}{N^2}\,\operatorname{tr}\!\bigl(\bm{Q}^{12}(\bm{Q}^{12})^{\top}\bigr)
\]
all vanish almost surely. As a result, we obtain that
\[
\frac{1}{N\sqrt{N}}\mathbb{E}\!\left[
\frac{\partial}{\partial W_{ijk}}
\bm{Q}^{12}_{ij}(z)\right] = \mathbb{E}[-\frac{1}{N}\tr\bm{Q}^{11}(z)\frac{1}{N}\tr\bm{Q}^{22}(z)] + \mathcal{O}(N^{-1}).
\]
We prove directly that for each fixed block index $i\in\{1,2,3\}$ and each
$z\in\mathbb C^+$ with $\eta:=\Im z\ge \eta_0>0$,
\begin{equation}\label{eq:Var_goal_original}
\Var\!\left(\frac1N\Tr \bm Q^{ii}(z)\right)\ \lesssim\ \frac{1}{N\,\eta^4}.
\end{equation}

%------------------------------------------------------------
%\paragraph{Variance bound for $\frac1N\Tr \bm Q^{ii}(z)$.}
\begin{lemma} [Variance bound for $\frac1N\Tr \bm Q^{ii}(z)$] 
\label{lem:variance-bound}
Let
\[
f(\textbf{W}):=\frac1N\Tr \bm Q^{ii}(z),\qquad 
\bm Q(z)=(\bm N-zI)^{-1},
\]
where
\[
\bm N=\frac1{\sqrt N}\Phi_3(\textbf{W},u,v,w),
\]
and assume that the entries $\{W_{abc}\}$ are i.i.d.\ with
$\mathbb E W_{abc}=0$ , $\mathbb E W_{abc}^2=1$ and \(\mathbb E |W_{abc}|^{4} < \infty\).
We have that for any fixed $i\in\{1,2,3\}$ and $z\in\mathbb C^+$ with
$\eta:=\Im z\ge\eta_0>0$,
\begin{equation}\label{eq:Var_goal}
\Var\!\left(\frac1N\Tr \bm Q^{ii}(z)\right)\ \lesssim\ \frac{1}{N\eta^4}.
\end{equation}
\end{lemma}

\begin{proof} 
The proof includes four steps as follows. 

%\medskip
\noindent\textbf{Step 1: Efron--Stein inequality.}
For each $\alpha=(a,b,c)$, let $\bm W^{(\alpha)}$ be obtained from $\bm W$ by
replacing $W_{abc}$ with an independent copy $W'_{abc}$.
Denote by $\bm N^{(\alpha)}$ and $\bm Q^{(\alpha)}(z)$ the corresponding matrices,
and set
\[
f^{(\alpha)}:=\frac1N\Tr (\bm Q^{(\alpha)})^{ii}(z).
\]
It follows from the Efron--Stein inequality that
\begin{equation}\label{eq:ES}
\Var(f)\ \le\ \frac12\sum_{\alpha}\mathbb E\bigl[(f-f^{(\alpha)})^2\bigr].
\end{equation}

%\medskip
\noindent\textbf{Step 2: Single-entry replacement.}
Let $\Delta_\alpha:=W_{abc}-W'_{abc}$. Since $\Phi_3$ is linear in $\textbf{W}$,
\[
\bm N-\bm N^{(\alpha)}=\frac{\Delta_\alpha}{\sqrt N}\bm M_\alpha,
\]
where
\[
\bm M_\alpha
=
w_c\,(E^{12}_{ab}+E^{21}_{ba})
+
v_b\,(E^{13}_{ac}+E^{31}_{ca})
+
u_a\,(E^{23}_{bc}+E^{32}_{cb}).
\]
Here $E^{rs}_{pq}$ denotes the elementary matrix supported on block $(r,s)$ at
position $(p,q)$.
The resolvent identity gives the following: 
\[
\bm Q-\bm Q^{(\alpha)}
=
-\frac{\Delta_\alpha}{\sqrt N}\bm Q\bm M_\alpha\bm Q^{(\alpha)},
\]
hence
\begin{equation}\label{eq:fdiff}
f-f^{(\alpha)}
=
-\frac{\Delta_\alpha}{N\sqrt N}
\Tr\big((\bm Q\bm M_\alpha\bm Q^{(\alpha)})^{ii}\big).
\end{equation}

%\medskip
\noindent\textbf{Step 3: Expansion of the trace and local Ward bounds.}
We treat each term in $\bm M_\alpha$ separately. Consider first the contribution
of $w_c E^{12}_{ab}$. By block multiplication, we have
\[
(\bm Q E^{12}_{ab}\bm Q^{(\alpha)})^{ii}
=
\bm Q^{i1}E^{12}_{ab}\bm Q^{(\alpha)2i},
\]
and therefore
\[
\Tr\big((\bm Q E^{12}_{ab}\bm Q^{(\alpha)})^{ii}\big)
=
\sum_{p}\bm Q^{i1}_{pa}\,\bm Q^{(\alpha)2i}_{bp}.
\]
Applying the Cauchy--Schwarz inequality, we have
\[
\Big|\sum_{p}\bm Q^{i1}_{pa}\bm Q^{(\alpha)2i}_{bp}\Big|
\le
\Big(\sum_{p}|\bm Q^{i1}_{pa}|^2\Big)^{1/2}
\Big(\sum_{p}|\bm Q^{(\alpha)2i}_{bp}|^2\Big)^{1/2}.
\]
Using the local Ward identity, we have
\[
\sum_{p}|\bm Q^{i1}_{pa}|^2\le \sum_p|\bm Q_{pa}|^2
=\frac1\eta\,\Im \bm Q_{aa}\le \frac1{\eta^2},
\]
and similarly $\sum_p|\bm Q^{(\alpha)2i}_{bp}|^2\le \eta^{-2}$, we obtain
\[
\Big|\Tr\big((\bm Q E^{12}_{ab}\bm Q^{(\alpha)})^{ii}\big)\Big|
\ \lesssim\ \frac1{\eta^2}.
\]
Multiplying by $|w_c|$ yields
\[
\Big|\Tr\big((\bm Q\,w_c E^{12}_{ab}\bm Q^{(\alpha)})^{ii}\big)\Big|
\ \lesssim\ \frac{|w_c|}{\eta^2}.
\]
The remaining terms in $\bm M_\alpha$ are treated in the same way, leading to
\begin{equation}\label{eq:trace_M_bound}
\Big|\Tr\big((\bm Q\bm M_\alpha\bm Q^{(\alpha)})^{ii}\big)\Big|
\ \lesssim\ 
\frac{1}{\eta^2}\big(|u_a|+|v_b|+|w_c|\big).
\end{equation}

%\medskip
\noindent\textbf{Step 4: Squaring and summation over $\alpha$.}
Combining (\ref{eq:fdiff}) and (\ref{eq:trace_M_bound}), we obtain 
\[
|f-f^{(\alpha)}|
\ \lesssim\
\frac{|\Delta_\alpha|}{N\sqrt N}\cdot
\frac{1}{\eta^2}\big(|u_a|+|v_b|+|w_c|\big).
\]
Squaring and taking expectations, using $\mathbb E|\Delta_\alpha|^2\lesssim1$, we have
\[
\mathbb E[(f-f^{(\alpha)})^2]
\ \lesssim\
\frac{1}{N^3\eta^4}
\big(|u_a|^2+|v_b|^2+|w_c|^2\big).
\]
Summing over all $\alpha=(a,b,c)$, we have
\[
\sum_{\alpha}\mathbb E[(f-f^{(\alpha)})^2]
\ \lesssim\
\frac{1}{N^3\eta^4}
\Big(
\sum_{a,b,c}|u_a|^2
+
\sum_{a,b,c}|v_b|^2
+
\sum_{a,b,c}|w_c|^2
\Big).
\]
Since $\sum_a|u_a|^2=\sum_b|v_b|^2=\sum_c|w_c|^2=1$, the right-hand side equals
$O(N^{-1}\eta^{-4})$.
Applying (\ref{eq:ES}) proves (\ref{eq:Var_goal}).
\end{proof}

Then we have
\begin{equation}
    \frac{1}{N\sqrt{N}}\mathbb{E}\!\left[
\frac{\partial}{\partial W_{ijk}}
\bm{Q}^{12}_{ij}(z)\right] = \mathbb{E}[-\frac{1}{N}\tr\bm{Q}^{11}(z)\frac{1}{N}\tr\bm{Q}^{22}(z)] + \mathcal{O}(N^{-1}) \xrightarrow{a.s.} -g_1(z)g_2(z).
\end{equation}
Now we consider  \(\varepsilon_{ijk}^{(2)}(z)\). Our aim is to bound \( \frac{1}{N\sqrt{N}}\sum_{ijk}\sup_{z \in \mathbb C^+_{\eta_0}} |\frac{\partial^2}{(\partial W_{ijk})^2}\bm{Q}_{ij}^{12}(z)|\).

Recall that $\bm Q(z)=(\bm N-zI)^{-1}$ and denote
\[
\bm A_{ijk}:=\frac{\partial \bm N}{\partial W_{ijk}} .
\]
Since $\bm N$ depends linearly on each $W_{ijk}$, we have
\(
\partial^2 \bm N/(\partial W_{ijk})^2=0
\).
By the resolvent identity,
\[
\partial_{W_{ijk}} \bm Q(z)
=
- \bm Q(z) \bm A_{ijk} \bm Q(z),
\]
and differentiating once more yields
\[
\partial_{W_{ijk}}^2 \bm Q(z)
=
2\,\bm Q(z) \bm A_{ijk} \bm Q(z) \bm A_{ijk} \bm Q(z).
\]
Hence, for any $z\in\mathbb C^+_{\eta_0}$,
\[
\begin{aligned}
\left|
\partial_{W_{ijk}}^2 Q^{12}_{ij}(z)
\right|
&\le
2\,
\bigl\|
\bm Q(z) \bm A_{ijk} \bm Q(z) \bm A_{ijk} \bm Q(z)
\bigr\|_2 \\
&\le
2\,\|\bm Q(z)\|_2^3\,\|\bm A_{ijk}\|_2^2 .
\end{aligned}
\]
Since $\|\bm Q(z)\|_2\le\eta_0^{-1}$ for $z\in\mathbb C^+_{\eta_0}$, it follows that
\[
\sup_{z\in\mathbb C^+_{\eta_0}}
\left|
\partial_{W_{ijk}}^2 Q^{12}_{ij}(z)
\right|
\le
2\,\eta_0^{-3}\,\|\bm A_{ijk}\|_2^2 .
\]

Moreover, by the structure of $\Phi_3(W,u,v,w)$ and the normalization of the noise,
\[
\|\bm A_{ijk}\|_2
=
\left\|
\frac{\partial N}{\partial W_{ijk}}
\right\|_2
\le
\frac{C}{\sqrt N}
\bigl(|u_i|+|v_j|+|w_k|\bigr),
\]
so that
\[
\|\bm A_{ijk}\|_2^2
\le
\frac{C}{N}
\left(u_i^2+v_j^2+w_k^2\right).
\]
Summing over $(i,j,k)$ and using $\|u\|_2=\|v\|_2=\|w\|_2=1$, we obtain
\[
\sum_{i,j,k}\|\bm A_{ijk}\|_2^2
\le
\frac{C}{N}\bigl(np+mp+mn\bigr)
=
\mathcal{O}(N),
\]
where we used $m+n+p=N$.

Combining the above estimates, we conclude that
\[
\begin{aligned}
\frac{1}{N\sqrt N}
\sum_{i,j,k}
\sup_{z\in\mathbb C^+_{\eta_0}}
\left|
\frac{\partial^2}{(\partial W_{ijk})^2}Q^{12}_{ij}(z)
\right|
&\le
\frac{2\eta_0^{-3}}{N\sqrt N}
\sum_{i,j,k}\|\bm A_{ijk}\|_2^2 \\
&\le
C\,\eta_0^{-3}\,N^{-1/2}
\;\xrightarrow{N\to\infty}\;0 .
\end{aligned}
\]

Finally we have
\begin{equation}
    \frac{1}{N\sqrt{N}}
\sum_{i,j,k}
w_k\, \mathbb{E}\!\left[
W_{ijk}\, \bm{Q}^{12}_{ij}(z)
\right]\xrightarrow{a.s.}-g_1(z)g_2(z).
\end{equation}
It's similar to get
\begin{equation}
    \frac{1}{N\sqrt{N}}
\sum_{i,j,k}
v_j\, \mathbb{E}\!\left[
W_{ijk}\, \bm{Q}^{13}_{ij}(z)
\right]\xrightarrow{a.s.}-g_1(z)g_3(z).
\end{equation}
From Eq.~(\ref{eqA.6}), taking the expectation and taking the limit on both sides of the equation, we have
\begin{equation}
    -g_1(z)g_2(z)-g_1(z)g_3(z) -zg_1(z) = c_1.
\end{equation}
Similarly, we have
\[
\begin{cases}
-\,g_1(z)\bigl(g_2(z)+g_3(z)\bigr) - z\,g_1(z) = c_1,\\
-\,g_2(z)\bigl(g_1(z)+g_3(z)\bigr) - z\,g_2(z) = c_2,\\
-\,g_3(z)\bigl(g_1(z)+g_2(z)\bigr) - z\,g_3(z) = c_3,
\end{cases}
\]
which yields
\[
g_i(z)
=
\frac{g(z)+z}{2}
-
\frac{\sqrt{4c_i+\bigl(g(z)+z\bigr)^2}}{2},
\]
with $g(z)$ solution of the equation
\[
g(z)=\sum_{i=1}^3 g_i(z),
\]
satisfying \(z \in \mathbb C^+_{\eta_0}\).

The system admits a unique solution $(g_1(z),g_2(z),g_3(z))\in(\mathbb C^+)^3$
for every $z\in\mathbb C^+$, which is the analytic continuation of the solution
constructed on $\mathbb C^+_{\eta_0}$. In particular, the representation
\[
g_i(z)=\frac{g(z)+z-\sqrt{(g(z)+z)^2+4c_i}}{2}
\]
holds for all $z\in\mathbb C^+$, where the square root is taken with the branch
that ensures $\Im g_i(z)>0$ for $\Im z>0$.
\subsection{Proof of Corollary \ref{cor3.1}}\label{appA.3}
Given the result of Theorem~\ref{thm3.2}, setting $c_1=c_2=c_3=\frac13$, we have for all
$i\in[3]$
\[
g_i(z)
=
\frac{g(z)+z}{2}
-
\frac{\sqrt{\tfrac{4}{3}+\bigl(g(z)+z\bigr)^2}}{2},
\]
where $g(z)$ satisfies $g(z)=\sum_{i=1}^3 g_i(z)$. Thus $g(z)$ is the solution to
\[
z+\frac{g(z)+z}{2}
-
\frac{3\sqrt{\tfrac{4}{3}+\bigl(g(z)+z\bigr)^2}}{2}
=0.
\]
Solving in $g(z)$ yields
\[
g(z)\in
\left\{
-\frac{3z}{4}-\frac{\sqrt{3}\sqrt{3z^2-8}}{4},
\;
-\frac{3z}{4}+\frac{\sqrt{3}\sqrt{3z^2-8}}{4}
\right\},
\]
and the limiting Stieltjes transform (with $\Im g(z)>0$ for $z$

with $\Im z>0$
 is therefore
\[
g(z)
=
-\frac{3z}{4}
+
\frac{\sqrt{3}\sqrt{3z^2-8}}{4}.
\]
\subsection{Proof of Theorem~\ref{thm3.3}}\label{appA.4}

Similarly we assume \(z \in \mathbb C^+_{\eta_0}\). Given the random tensor model in Eq.~(\ref{eq3.1}) and its singular vectors characterized
by Eq.~(\ref{eq3.2}), we denote the associated random matrix model as
\[
\bm T \equiv \Phi_3(\textbf{T},\bm u,\bm v,\bm w)
= \beta \bm V \bm B \bm V^{\top} + \bm N,
\]
where
\[
\bm N
=
\frac{1}{\sqrt N}\,\Phi_3(\textbf{W},\bm u,\bm v,\bm w),
\qquad
\bm B
\equiv
\begin{bmatrix}
0 & \langle \bm z,\bm w\rangle & \langle \bm y,\bm v\rangle \\
\langle \bm z,\bm w\rangle & 0 & \langle \bm x,\bm u\rangle \\
\langle \bm y,\bm v\rangle & \langle \bm x,\bm u\rangle & 0
\end{bmatrix}
\in \mathbb M_3,
\]
and
\[
\bm V
\equiv
\begin{bmatrix}
\bm x & \bm 0_m & \bm 0_m \\
\bm 0_n & \bm y & \bm 0_n \\
\bm 0_p & \bm 0_p & \bm z
\end{bmatrix}
\in \mathbb M_{N,3}.
\]

We further denote the resolvents of $\textbf{T}$ and $\bm N$ respectively as
\[
\bm R(z)
=
(\textbf{T} - z \bm I_N)^{-1}
=
\begin{bmatrix}
\bm R^{11}(z) & \bm R^{12}(z) & \bm R^{13}(z) \\
\bm R^{12}(z)^{\top} & \bm R^{22}(z) & \bm R^{23}(z) \\
\bm R^{13}(z)^{\top} & \bm R^{23}(z)^{\top} & \bm R^{33}(z)
\end{bmatrix},
\]
and
\[
\bm Q(z)
=
(\bm N - z \bm I_N)^{-1}
=
\begin{bmatrix}
\bm Q^{11}(z) & \bm Q^{12}(z) & \bm Q^{13}(z) \\
\bm Q^{12}(z)^{\top} & \bm Q^{22}(z) & \bm Q^{23}(z) \\
\bm Q^{13}(z)^{\top} & \bm Q^{23}(z)^{\top} & \bm Q^{33}(z)
\end{bmatrix}.
\]

By the Woodbury matrix identity, we have
\begin{equation}\label{eqA.24}
\bm R(z)
=
\bm Q(z)
-
\bm Q(z)\bm V
\left(
\frac{1}{\beta}\bm B^{-1}
+
\bm V^{\top}\bm Q(z)\bm V
\right)^{-1}
\bm V^{\top}\bm Q(z).
\end{equation}

In particular, taking the normalized trace operator, we obtain
\[
\frac{1}{N}\operatorname{tr}\bm R(z)
=
\frac{1}{N}\operatorname{tr}\bm Q(z)
-
\frac{1}{N}
\operatorname{tr}\!\left[
\left(
\frac{1}{\beta}\bm B^{-1}
+
\bm V^{\top}\bm Q(z)\bm V
\right)^{-1}
\bm V^{\top}\bm Q^2(z)\bm V
\right],
\]
and hence
\[
\frac{1}{N}\operatorname{tr}\bm R(z)
=
\frac{1}{N}\operatorname{tr}\bm Q(z)
+
\mathcal O(N^{-1}),
\]
since the matrix
\(
\left(
\frac{1}{\beta}\bm B^{-1}
+
\bm V^{\top}\bm Q(z)\bm V
\right)^{-1}
\bm V^{\top}\bm Q^2(z)\bm V
\)
has bounded spectral norm since $\|\bm Q(z)\|$ is bounded.

As such, the asymptotic spectral measure of $\textbf{T}$ is the same as the one
of $\bm N$, which can be estimated through $\frac{1}{N}\operatorname{tr}\bm Q(z)$.
Comparing to the result from Appendix~\ref{appA.2}, now the singular vectors
$\bm u,\bm v,\bm w$ depend statistically on the tensor noise
$\bm W$, which needs to be handled.

From \ref{appA.2}, we have
\(
\bm N\bm Q(z)-z\bm Q(z)=\bm I_N,
\)
from which
\begin{equation}\label{eq:A.25}
\frac{1}{N\sqrt N}\sum_{i=1}^m\bigl[\textbf{W}(\bm w)\bm Q^{12}(z)^{\top}\bigr]_{ii}
+
\frac{1}{N\sqrt N}\sum_{i=1}^m\bigl[\textbf{W}(\bm v)\bm Q^{13}(z)^{\top}\bigr]_{ii}
-
\frac{z}{N}\operatorname{tr}\bm Q^{11}(z)
=
\frac{m}{N}.
\end{equation}

We thus need to compute the expectations of
\(
\frac{1}{N\sqrt N}\sum_{i=1}^m[\textbf{W}(\bm w)\bm Q^{12}(z)^{\top}]_{ii}
\)
and
\(
\frac{1}{N\sqrt N}\sum_{i=1}^m[\textbf{W}(\bm v)\bm Q^{13}(z)^{\top}]_{ii}.
\)
In particular,
\[
A
\equiv
\frac{1}{N\sqrt N}\sum_{i=1}^m
\mathbb E\!\left[\textbf{W}(\bm w)\bm Q^{12}(z)^{\top}\right]_{ii}
=
\frac{1}{N\sqrt N}\sum_{ijk}\mathbb E\!\left[W_{ijk}w_kQ^{12}_{ij}\right].
\]
From lemma~\ref{lem2.3} we have
\begin{equation}
\frac{1}{N\sqrt N}\mathbb{E}\!\left[
W_{ijk}\, w_k\bm{Q}^{12}_{ij}(z)
\right]
=
\frac{1}{N\sqrt N}\mathbb{E}\!\left[
\frac{\partial}{\partial W_{ijk}}
(w_k\bm{Q}^{12}_{ij}(z))
\right]
+ \frac{1}{N\sqrt N}\varepsilon_{ijk}^{(2)}(z),
\label{eqA.26}
\end{equation}
where \(\varepsilon_{ijk}^{(2)}(z) \leq C_1 \, \sup_{z \in \mathbb C^+_{\eta_0}} |g^{(2)}(z)| \, \mathbb{E}[|W_{ijk}|^{3}] = C_1^* \sup_{z \in \mathbb C^+_{\eta_0}} |\frac{\partial^2}{(\partial W_{ijk})^2}(w_k\bm{Q}_{ij}^{12}(z))|\).

Firstly we consider the first derivative. Due to the statistical dependency between $\bm w$ and $\textbf{W}$,
the above sum decomposes into two terms:
\[
A
=
\frac{1}{N\sqrt N}\sum_{ijk}
\mathbb E\!\left[w_k\frac{\partial Q^{12}_{ij}}{\partial W_{ijk}}\right]
+
\frac{1}{N\sqrt N}\sum_{ijk}
\mathbb E\!\left[\frac{\partial w_k}{\partial W_{ijk}}Q^{12}_{ij}\right]
\equiv A_1+A_2.
\]

The first term $A_1$ has already been handled in the previous subsection
(replacing $\bm c$ with $\bm w$).
We now show that the second term $A_2$ is asymptotically vanishing under
Assumption~\ref{assumption3.2}.
Indeed, by Eq.~(\ref{eq3.4}), we have
\begin{equation}
\begin{aligned}
\frac{\partial w_k}{\partial W_{ijk}}
=&
-\frac{1}{\sqrt N}
\Bigl(
v_jw_k(R^{13}_{ik}(\lambda)-u_i\bm u^{\top}\bm R^{13}_{\cdot,k}(\lambda))
\Bigr) \\
&-\frac{1}{\sqrt N}
\Bigl(
u_iw_k(R^{23}_{jk}(\lambda)-v_j\bm v^{\top}\bm R^{23}_{\cdot,k}(\lambda))
\Bigr)
-\frac{1}{\sqrt N}
\Bigl(
u_iv_j(R^{33}_{kk}(\lambda)-w_k\bm w^{\top}\bm R^{33}_{\cdot,k}(\lambda))
\Bigr).
\end{aligned}
\label{eqA.27}
\end{equation}

As such, $A_2$ decomposes into three terms $A_2=A_{21}+A_{22}+A_{23}$, where
\[
\begin{aligned}
A_{21}
&=
-\frac{1}{N^2}\sum_{ijk}
\mathbb E\!\left[v_jw_kR^{13}_{ik}(\lambda)Q^{12}_{ij}\right]
+
\frac{1}{N^2}\sum_{ijkl}
\mathbb E\!\left[u_iv_jw_ku_lR^{13}_{lk}(\lambda)Q^{12}_{ij}\right] \\
&=
-\frac{1}{N^2}\mathbb E\!\left[\bm v^{\top}\bm Q^{12}(z)^{\top}\bm R^{13}(\lambda)\bm w\right]
+
\frac{1}{N^2}\mathbb E\!\left[\bm u^{\top}\bm Q^{12}(z)\bm v\,
\bm u^{\top}\bm R^{13}(\lambda)\bm w\right]
\;\xrightarrow[N\to\infty]{}\;0,
\end{aligned}
\]
since the singular vectors $\bm u,\bm v,\bm w$ are of bounded norms
and assuming the resolvents $\bm Q(z)$ and $\bm R(\lambda)$ are of bounded
spectral norms.

Similarly,
\[
\begin{aligned}
A_{22}
&=
-\frac{1}{N^2}\sum_{ijk}
\mathbb E\!\left[u_iw_kR^{23}_{jk}(\lambda)Q^{12}_{ij}\right]
+
\frac{1}{N^2}\sum_{ijkl}
\mathbb E\!\left[u_iv_jw_kv_lR^{23}_{lk}(\lambda)Q^{12}_{ij}\right] \\
&=
-\frac{1}{N^2}\mathbb E\!\left[\bm u^{\top}\bm Q^{12}(z)\bm R^{23}(\lambda)\bm w\right]
+
\frac{1}{N^2}\mathbb E\!\left[\bm u^{\top}\bm Q^{12}(z)\bm v\,
\bm v^{\top}\bm R^{23}(\lambda)\bm w\right]
\;\xrightarrow[N\to\infty]{}\;0.
\end{aligned}
\]

And finally,
\[
\begin{aligned}
A_{23}
&=
-\frac{1}{N^2}\sum_{ijk}
\mathbb E\!\left[u_iv_jR^{33}_{kk}(\lambda)Q^{12}_{ij}\right]
+
\frac{1}{N^2}\sum_{ijkl}
\mathbb E\!\left[u_iv_jw_kw_lR^{33}_{lk}(\lambda)Q^{12}_{ij}\right] \\
&=
-\frac{1}{N}
\mathbb E\!\left[
\bm u^{\top}\bm Q^{12}(z)\bm v\,
\frac{1}{N}\operatorname{tr}\bm R^{33}(\lambda)
\right]
+
\frac{1}{N^2}
\mathbb E\!\left[
\bm u^{\top}\bm Q^{12}(z)\bm v\,
\bm w^{\top}\bm R^{33}(\lambda)\bm w
\right]
\;\xrightarrow[N\to\infty]{}\;0.
\end{aligned}
\]

Therefore,
\[
A
=
\frac{1}{N\sqrt N}
\sum_{ijk}
\mathbb E\!\left[
w_k\,\frac{\partial Q^{12}_{ij}}{\partial W_{ijk}}
\right]
+
\mathcal O(N^{-1}).
\]

As in the previous subsection, the derivative of $\bm Q(z)$ with respect to
the entry $W_{ijk}$ expresses as
\[
\frac{\partial \bm Q(z)}{\partial W_{ijk}}
=
-\bm Q(z)\,
\frac{\partial \bm N}{\partial W_{ijk}}\,
\bm Q(z),
\]
but now with
\[
\frac{\partial \bm N}{\partial W_{ijk}}
=
\frac{1}{\sqrt N}
\begin{bmatrix}
\bm 0_{m\times m}
&
w_k\,\bm e_i^{m}(\bm e_j^{n})^{\top}
&
v_j\,\bm e_i^{m}(\bm e_k^{p})^{\top}
\\
w_k\,\bm e_j^{n}(\bm e_i^{m})^{\top}
&
\bm 0_{n\times n}
&
u_i\,\bm e_j^{n}(\bm e_k^{p})^{\top}
\\
v_j\,\bm e_k^{p}(\bm e_i^{m})^{\top}
&
u_i\,\bm e_k^{p}(\bm e_j^{n})^{\top}
&
\bm 0_{p\times p}
\end{bmatrix}
+
\frac{1}{\sqrt N}\,
\Phi_3\!\left(
\textbf{W},
\frac{\partial \bm u}{\partial W_{ijk}},
\frac{\partial \bm v}{\partial W_{ijk}},
\frac{\partial \bm w}{\partial W_{ijk}}
\right),
\]
we need to consider the negligibility of the implicit term.

\begin{lemma}\label{lemA.1}
    Letting \(\bm O_{ijk}
=
\Phi_3\!\left(
\textbf{W},
\frac{\partial \bm u}{\partial W_{ijk}},
\frac{\partial \bm v}{\partial W_{ijk}},
\frac{\partial \bm w}{\partial W_{ijk}}
\right)\), we have 
    \begin{equation}
        -\frac{1}{N^2}\sum_{i,j,k}
\mathbb E\!\left[
w_k\,
(\bm Q\,\bm O_{ijk}\,\bm Q)^{12}_{ij}
\right] = \mathcal O(N^{-1}).
\qquad
\end{equation}
\end{lemma}

\begin{remark} \label{remark4}
The spectral norm estimate for
\[
\bm O_{ijk}
\;=\;
\frac{1}{\sqrt N}\,
\Phi_3\!\left(
\mathbf X,
\frac{\partial \bm u}{\partial X_{ijk}},
\frac{\partial \bm v}{\partial X_{ijk}},
\frac{\partial \bm w}{\partial X_{ijk}}
\right)
\]
given in the proof of Theorem~4 of \cite{seddik2024whentrandomtensorsmeet} is not justified correctly.

Indeed, although it holds that
\(
\bigl\|\partial \bm u/\partial X_{ijk}\bigr\|_2,
\bigl\|\partial \bm v/\partial X_{ijk}\bigr\|_2,
\bigl\|\partial \bm w/\partial X_{ijk}\bigr\|_2
\lesssim N^{-1/2},
\)
these vector $\ell^2$-norm bounds alone do not imply a corresponding decay of
the operator norm of $\Phi_3(\mathbf X,\cdot,\cdot,\cdot)$.
The contraction $\Phi_3$ produces an $N\times N$ matrix whose entries involve
high-dimensional summations over the tensor indices of $\mathbf X$.
Without a careful analysis of index overlaps or an explicit use of
Ward-type identities, these summations introduce a combinatorial amplification
that offsets the apparent $N^{-3/2}$ scaling. In fact, \(\|\bm O_{ijk}\|\) is of order \(\mathcal O(N^{-1/2})\). Consequently, bounding its entries individually and then lifting such bounds to the spectral norm necessarily yields an \(\mathcal O(1)\) contribution from \(\|\bm O_{ijk}\|\), and no smaller order can be obtained by this approach.

Moreover, controlling the spectral norm of $\bm O$ requires mechanisms
that suppress worst-case directional growth, such as resolvent--vector
absorption, Frobenius-to-operator norm conversion, or cancellation after
taking expectations.
None of these mechanisms is invoked in the aforementioned estimate, where
vector norm bounds are implicitly treated as operator norm bounds.

As a result, the conclusion that $\|\bm O\|=\mathcal O(N^{-3/2})$ does not follow
from the stated assumptions, and the implicit term cannot be dismissed at the
operator-norm level without additional structural arguments.  We here use a new proof method and make correction in Lemma~\ref{lemA.1}
\end{remark}

\begin{proof}[The proof of Lemma~\ref{lemA.1}]
Let
\[
A_{\mathrm{imp}}
:=
-\frac{1}{N^2}
\sum_{i,j,k}
\mathbb E\!\left[
w_k\,
(\bm Q\,\bm O_{ijk}\,\bm Q)^{12}_{ij}
\right].
\]

By block multiplication, $(\bm Q\bm O_{ijk}\bm Q)^{12}_{ij}$ is a sum of finitely
many terms, each of which is of the generic form
\[
\sum_{\alpha,\beta,\gamma}
Q^{\cdot\cdot}_{i\alpha}\,
W_{\alpha\beta\gamma}\,
\dot \xi^{(ijk)}_{\gamma}\,
Q^{\cdot\cdot}_{\beta j},
\]
where $\dot\xi^{(ijk)}$ stands for one of
$\partial\bm u/\partial W_{ijk}$,
$\partial\bm v/\partial W_{ijk}$ or
$\partial\bm w/\partial W_{ijk}$.
It therefore suffices to estimate a single such contribution.

Fix indices $i,j,k$ and consider (the others are the same)
\[
\mathbb E\!\left[
w_k\,
Q^{11}_{i\alpha}\,
W_{\alpha\beta\gamma}\,
\dot \xi^{(ijk)}_{\gamma}\,
Q^{22}_{\beta j}
\right].
\]
Using a cumulant expansion with respect to
$W_{\alpha\beta\gamma}$ and the assumption that the entries of $\textbf{W}$
have zero mean, unit variance and finite fourth-moment, we obtain
\[
\mathbb E\!\left[
W_{\alpha\beta\gamma} F
\right]
=
\mathbb E\!\left[
\frac{\partial F}{\partial W_{\alpha\beta\gamma}}
\right]
+\varepsilon_{\alpha\beta\gamma}^{(2)},
\qquad
F:=w_k\,Q^{11}_{i\alpha}\,\dot \xi^{(ijk)}_{\gamma}\,Q^{22}_{\beta j},
\]
where the error term $\varepsilon_{\alpha\beta\gamma}^{(2)}$ is of higher order and can be estimated
analogously.

Now we consider
\[
\mathbb E\!\left[
\frac{\partial F}{\partial W_{\alpha\beta\gamma}}
\right],
\qquad
F
=
w_k\,Q^{11}_{i\alpha}\,\dot \xi^{(ijk)}_{\gamma}\,Q^{22}_{\beta j}.
\]
By the chain rule,
\[
\frac{\partial F}{\partial W_{\alpha\beta\gamma}}
=
(\mathrm{I})+(\mathrm{II})+(\mathrm{III})+(\mathrm{IV}),
\]
where
\[
\begin{aligned}
(\mathrm{I})\;&=\;
\Big(\frac{\partial w_k}{\partial W_{\alpha\beta\gamma}}\Big)
\,Q^{11}_{i\alpha}\,\dot \xi^{(ijk)}_{\gamma}\,Q^{22}_{\beta j},
\\[0.5ex]
(\mathrm{II})\;&=\;
w_k\Big(\frac{\partial Q^{11}_{i\alpha}}{\partial W_{\alpha\beta\gamma}}\Big)
\,\dot \xi^{(ijk)}_{\gamma}\,Q^{22}_{\beta j},
\\[0.5ex]
(\mathrm{III})\;&=\;
w_k\,Q^{11}_{i\alpha}
\Big(\frac{\partial \dot \xi^{(ijk)}_{\gamma}}{\partial W_{\alpha\beta\gamma}}\Big)
\,Q^{22}_{\beta j},
\\[0.5ex]
(\mathrm{IV})\;&=\;
w_k\,Q^{11}_{i\alpha}\,\dot \xi^{(ijk)}_{\gamma}
\Big(\frac{\partial Q^{22}_{\beta j}}{\partial W_{\alpha\beta\gamma}}\Big).
\end{aligned}
\]

\subsubsection{Estimate of (I)--(IV) after summation}

Fix $(k,\gamma)$ and define the $m\times n$ matrices
\[
\bm A^{(k,\gamma)}:=\Big(A^{(k,\gamma)}_{ab}\Big)_{a\le m,b\le n},
\qquad
A^{(k,\gamma)}_{ab}:=\frac{\partial w_k}{\partial W_{ab\gamma}},
\]
\[
\bm X^{(k,\gamma)}:=\Big(X^{(k,\gamma)}_{ij}\Big)_{i\le m,j\le n},
\qquad
X^{(k,\gamma)}_{ij}:=\dot w^{(ijk)}_\gamma=\frac{\partial w_\gamma}{\partial W_{ijk}}.
\]
Then the contribution of term \emph{(I)} to $\partial F_{ab\gamma}/\partial W_{ab\gamma}$ summed over $(i,j,a,b)$ can be rewritten exactly as a Frobenius pairing:
\begin{align*}
S^{(\mathrm{I})}_{k\gamma}
&:=
\sum_{i,j}\sum_{a,b}
Q^{11}_{ia}\,A^{(k,\gamma)}_{ab}\,X^{(k,\gamma)}_{ij}\,\bm Q^{22}_{bj}
=
\sum_{i,j} \bm X^{(k,\gamma)}_{ij}\,(\bm Q^{11}\bm A^{(k,\gamma)}\bm Q^{22})_{ij} \\
&=
\big\langle \bm X^{(k,\gamma)},\, \bm Q^{11}\bm A^{(k,\gamma)}\bm Q^{22}\big\rangle_F .
\end{align*}
Hence, taking absolute values \emph{after} summation and using the Cauchy--Schwarz inequality in the Frobenius norm,
\begin{equation}\label{eq:I_frob_bd}
\big|S^{(\mathrm{I})}_{k\gamma}\big|
\le
\|\bm X^{(k,\gamma)}\|_F\ \|\bm Q^{11}\bm A^{(k,\gamma)}\bm Q^{22}\|_F
\le
\|\bm Q\|^2\ \|\bm A^{(k,\gamma)}\|_F\ \|\bm X^{(k,\gamma)}\|_F .
\end{equation}

By the derivative formula (\ref{eq3.4}), for each fixed $(k,\gamma)$ there exist vectors
$\bm \alpha^{(k,\gamma)}\in\mathbb R^{p}$ and $\bm \beta^{(k,\gamma)}\in\mathbb R^{p}$ with
$\|\bm \alpha^{(k,\gamma)}\|_2+\|\bm \beta^{(k,\gamma)}\|_2\lesssim 1$ such that
\[
A^{(k,\gamma)}_{ab}
=
\frac{1}{\sqrt N}\,u_a v_b\,\bm \alpha^{(k,\gamma)}_k
\qquad\text{and}\qquad
X^{(k,\gamma)}_{ij}
=
\frac{1}{\sqrt N}\,u_iv_j\,\bm \beta^{(k,\gamma)}_\gamma ,
\]
so that, using $\|\bm u\|_2=\|\bm v\|_2=\|\bm w\|_2=1$,
\begin{equation}\label{eq:AX_F}
\|\bm A^{(k,\gamma)}\|_F
\le
\frac{C}{\sqrt N},
\qquad
\|\bm X^{(k,\gamma)}\|_F
\le
\frac{C}{\sqrt N},
\end{equation}
uniformly in $(k,\gamma)$. Plugging (\ref{eq:AX_F}) into (\ref{eq:I_frob_bd}) yields
\[
|S^{(\mathrm{I})}_{k\gamma}|
\le
\|\bm Q\|^2\cdot \frac{C}{\sqrt N}\cdot \frac{C}{\sqrt N}
\lesssim
\frac{1}{N}.
\]
Therefore, recalling the prefactor $N^{-2}$ in $A_{\mathrm{imp}}$ and summing over $(k,\gamma)$, we have
\[
\Bigg|
\frac{1}{N^2}\sum_{k,\gamma}\mathbb E\big[S^{(\mathrm{I})}_{k\gamma}\big]
\Bigg|
\le
\frac{1}{N^2}\sum_{k,\gamma}\mathbb E\big[|S^{(\mathrm{I})}_{k\gamma}|\big]
\lesssim
\frac{1}{N^2}\cdot N^2\cdot \frac{1}{N}
=
O(N^{-1}).
\]

\medskip
\noindent
For \emph{(II)} and \emph{(IV)}, define
\[
\bm B^{(b,\gamma)}:=\Big(\frac{\partial Q^{11}_{ia}}{\partial W_{ab\gamma}}\Big)_{i,a},
\qquad
\bm C^{(a,\gamma)}:=\Big(\frac{\partial Q^{22}_{bj}}{\partial W_{ab\gamma}}\Big)_{b,j}.
\]
Using $\frac{\partial \bm Q}{\partial W_{ab\gamma}}=-\bm Q(\partial \bm N/\partial W_{ab\gamma})\bm Q$ and
$\|\partial \bm N/\partial W_{ab\gamma}\| \lesssim \frac{1}{\sqrt{N}}$ by Lemma.~\ref{lem2.3}, we have
\[
\|\bm B^{(b,\gamma)}\| \lesssim \frac{1}{\sqrt{N}}\|Q\|^2,
\qquad
\|\bm C^{(a,\gamma)}\| \lesssim \frac{1}{\sqrt{N}}\|Q\|^2.
\]
Writing the $(i,j)$-summed contributions again as Frobenius pairings,
\[
S^{(\mathrm{II})}_{k\gamma}
=
\big\langle \bm X^{(k,\gamma)},\, \bm B^{(b,\gamma)} \bm A_0^{(k)} \bm Q^{22}\big\rangle_F,
\qquad
S^{(\mathrm{IV})}_{k\gamma}
=
\big\langle \bm X^{(k,\gamma)},\, \bm Q^{11} \bm A_0^{(k)} \bm C^{(a,\gamma)}\big\rangle_F,
\]
where $A_0^{(k)}$ denotes the corresponding rank-one block coming from
$\partial N/\partial W_{ab\gamma}$ (hence $\|A_0^{(k)}\|_F\lesssim 1$), the
Cauchy--Schwarz inequality gives
\[
|S^{(\mathrm{II})}_{k\gamma}|+|S^{(\mathrm{IV})}_{k\gamma}|
\lesssim
\|\bm X^{(k,\gamma)}\|_F\cdot \frac{1}{\sqrt{N}}\|\bm Q\|^3
\lesssim
\frac{1}{\sqrt{N}}\frac{1}{\sqrt N}
=\frac{1}{N},
\]
where we used $\|\bm X^{(k,\gamma)}\|_F\lesssim N^{-1/2}$ from (\ref{eq:AX_F}).
Hence their total contributions to $A_{\mathrm{imp}}$ are $\mathcal O(N^{-1})$ after the
prefactor $N^{-2}$ and the $(k,\gamma)$-summation.

\medskip
\noindent
For \emph{(III)}, differentiating (\ref{eq3.4}) once more yields
$\|\partial \dot{\bm w}^{(ijk)}/\partial W_{ab\gamma}\|_2\lesssim N^{-1}$, hence
\[
\Big\|\Big(\frac{\partial \dot w^{(ijk)}_\gamma}{\partial W_{ab\gamma}}\Big)_{i,j}\Big\|_F
\lesssim
\frac{1}{N}.
\]
Thus, with the same Frobenius pairing as above,
\[
|S^{(\mathrm{III})}_{k\gamma}|
\le
\|\bm Q\|^2\cdot \Big\|\Big(\frac{\partial \dot w^{(ijk)}_\gamma}{\partial W_{ab\gamma}}\Big)_{i,j}\Big\|_F
\lesssim
\frac{1}{N},
\]
and again the total contribution to $A_{\mathrm{imp}}$ is $O(N^{-1})$.
\end{proof}
\medskip
\noindent
Combining (I)--(IV) yields
\[
A_{\mathrm{imp}}=\mathcal{O}(N^{-1}).
\]

Therefore, we find that
\[
A
\longrightarrow
-\,g_1(z)\,g_2(z)
+
\mathcal O(N^{-1}),
\]

Now we need to bound \(N^{-3/2}\varepsilon_{ijk}^{(2)}(z)\), actually we only need to bound 
\begin{equation}
    N^{-3/2}\sum_{ijk}\sup_{z \in \mathbb C^+_{\eta_0}} |\frac{\partial^2}{(\partial W_{ijk})^2}(w_k\bm{Q}_{ij}^{12}(z))|.
\end{equation}

We split it into three parts:
\begin{align}
\mathrm{I}
&:= N^{-3/2}\sum_{i,j,k}
\sup_{z \in \mathbb C^+_{\eta_0}}
\left|
w_k\,\frac{\partial^2}{(\partial W_{ijk})^2}
\bm Q_{ij}^{12}(z)
\right|,
\\[0.5em]
\mathrm{II}
&:= 2N^{-3/2}\sum_{i,j,k}
\sup_{z \in \mathbb C^+_{\eta_0}}
\left|
\frac{\partial w_k}{\partial W_{ijk}}\,
\frac{\partial}{\partial W_{ijk}}
\bm Q_{ij}^{12}(z)
\right|,
\\[0.5em]
\mathrm{III}
&:= N^{-3/2}\sum_{i,j,k}
\sup_{z \in \mathbb C^+_{\eta_0}}
\left|
\bm Q_{ij}^{12}(z)\,
\frac{\partial^2}{(\partial W_{ijk})^2}
w_k
\right|.
\end{align}

\subsubsection{Bounding the first part}\label{secA.4.1}

Fix $(i,j,k)$ and denote $\partial\equiv \partial_{W_{ijk}}$. Recall that
\[
\bm N=\frac1{\sqrt N}\Phi_3(\textbf{W},\bm u,\bm v,\bm w),\qquad 
\bm Q(z)=(\bm N-z\bm I)^{-1},\qquad z\in\mathbb C^+_{\eta_0},
\]
so that $\|\bm Q(z)\|\le \eta_0^{-1}$ on $\mathbb C^+_{\eta_0}$. By the resolvent identities,
\begin{equation}\label{eq:dQ}
\partial \bm Q(z)=-\bm Q(z)\,(\partial\bm N)\,\bm Q(z),
\qquad
\partial^2\bm Q(z)=2\bm Q(\partial\bm N)\bm Q(\partial\bm N)\bm Q-\bm Q(\partial^2\bm N)\bm Q .
\end{equation}
Denote\(\bm A_{ijk}:=\partial\bm N\). In the control of the second cumulant remainder, the leading contribution comes from the
\emph{quadratic resolvent term}
\[
\bm Q(z)\bm A_{ijk}\bm Q(z)\bm A_{ijk}\bm Q(z),
\]
hence it suffices to prove that
\begin{equation}\label{eq:goal-QAQAQ}
\frac{1}{N\sqrt N}\sum_{i,j,k}\ \sup_{z\in\mathbb C^+_{\eta_0}}
\Bigl|\,w_k\bigl[\bm Q(z)\bm A_{ijk}\bm Q(z)\bm A_{ijk}\bm Q(z)\bigr]^{12}_{ij}\Bigr|
=\mathcal O(\frac{1}{\sqrt{N}}).
\end{equation}

From $\bm N=\frac1{\sqrt N}\Phi_3(\textbf{W},\bm u,\bm v,\bm w)$ we have
\begin{equation}\label{eq:A-decomp}
\bm A_{ijk}
=
\frac{1}{\sqrt N}\bm A^{(0)}_{ijk}
+\bm A^{(1)}_{ijk},
\end{equation}
where the \emph{explicit} part $\bm A^{(0)}_{ijk}$ is the sparse/rank--one block matrix
\begin{equation}\label{eq:A0}
\bm A^{(0)}_{ijk}
=
\begin{bmatrix}
\bm 0_{m\times m} &
w_k\,\bm e_i^{m}(\bm e_j^{n})^{\top} &
v_j\,\bm e_i^{m}(\bm e_k^{p})^{\top} \\
w_k\,\bm e_j^{n}(\bm e_i^{m})^{\top} &
\bm 0_{n\times n} &
u_i\,\bm e_j^{n}(\bm e_k^{p})^{\top} \\
v_j\,\bm e_k^{p}(\bm e_i^{m})^{\top} &
u_i\,\bm e_k^{p}(\bm e_j^{n})^{\top} &
\bm 0_{p\times p}
\end{bmatrix},
\end{equation}
and the \emph{implicit} part $\bm A^{(1)}_{ijk}$ comes from the dependence of $(\bm u,\bm v,\bm w)$ on $\textbf{W}$:
\[
\bm A^{(1)}_{ijk}
=\frac{1}{\sqrt N}\Phi_3\!\left(
\textbf{W},\frac{\partial \bm u}{\partial W_{ijk}},\frac{\partial \bm v}{\partial W_{ijk}},\frac{\partial \bm w}{\partial W_{ijk}}
\right).
\]
Similar to Lemma~\ref{lemA.1}, it yields an extra $\mathcal O(N^{-1/2})$ factor compared to $\bm A^{(0)}_{ijk}$.
Consequently, it suffices to treat $\bm A_{ijk}$ replaced by $\frac1{\sqrt N}\bm A^{(0)}_{ijk}$.

For any $z\in\mathbb C^+_{\eta_0}$ we have $\|\bm Q(z)\|\le \eta_0^{-1}$. Moreover, for any canonical basis vector $\bm e_\alpha$,
\begin{equation}\label{eq:col-l2}
\|\bm Q(z)\bm e_\alpha\|_2^2
=
\sum_{r=1}^{N}|Q_{r\alpha}(z)|^2
=
\frac{\Im Q_{\alpha\alpha}(z)}{\Im z}
\le
\frac{|Q_{\alpha\alpha}(z)|}{\Im z}
\le \eta_0^{-2},
\qquad z\in\mathbb C^+_{\eta_0}.
\end{equation}
In particular,
\begin{equation}\label{eq:col-l1}
\sum_{r=1}^{N}|Q_{r\alpha}(z)|
\le \sqrt N\,\|\bm Q(z)\bm e_\alpha\|_2
\le \sqrt N\,\eta_0^{-1}.
\end{equation}

Because $\bm A^{(0)}_{ijk}$ has only three rank--one off--diagonal blocks (and their transposes),
each entry $\bigl[\bm Q\bm A^{(0)}\bm Q\bm A^{(0)}\bm Q\bigr]^{12}_{ij}$ is a sum of finitely many terms.
We illustrate the bound on the three representative classes; all remaining classes are controlled identically.

\medskip
\noindent\textbf{Class I (two spikes of type $12$ and $21$).}
Using (\ref{eq:A0}), the contribution where both factors $\bm A^{(0)}_{ijk}$ act through the $(1,2)$ and $(2,1)$ blocks yields
\[
\bigl[\bm Q\bm A^{(0)}\bm Q\bm A^{(0)}\bm Q\bigr]^{12}_{ij}
\ \supset\
w_k^2\,\bigl[\bm Q^{11}\bm e_i^m(\bm e_j^n)^{\top}\bm Q^{22}\bm e_j^n(\bm e_i^m)^{\top}\bm Q^{12}\bigr]_{ij}.
\]
Since $(\bm e_j^n)^{\top}\bm Q^{22}\bm e_j^n=Q^{22}_{jj}$ and $(\bm e_i^m)^{\top}\bm Q^{11}\bm e_i^m=Q^{11}_{ii}$, we obtain
\begin{equation}\label{eq:classI}
\bigl|\bigl[\bm Q\bm A^{(0)}\bm Q\bm A^{(0)}\bm Q\bigr]^{12}_{ij}\bigr|
\ \lesssim\
|w_k|^2\,|Q^{11}_{ii}|\cdot|Q^{22}_{jj}|\cdot \|\bm Q^{12}\bm e_j^n\|_2,
\end{equation}
where the last factor comes from bounding the remaining $\bm Q^{12}$-column by the Cauchy--Schwarz inequality.
Using (\ref{eq:col-l2}) and $|Q^{11}_{ii}|,|Q^{22}_{jj}|\le \eta_0^{-1}$), we get
\[
\sup_{z\in\mathbb C^+_{\eta_0}}
\bigl|\bigl[\bm Q\bm A^{(0)}\bm Q\bm A^{(0)}\bm Q\bigr]^{12}_{ij}\bigr|
\ \lesssim\
|w_k|^2\,\eta_0^{-3}.
\]
Therefore,
\begin{align}\label{eq:classI-sum}
\frac{1}{N\sqrt N}\sum_{i,j,k}\sup_{z\in\mathbb C^+_{\eta_0}}
\Bigl|w_k\bigl[\bm Q\bm A\bm Q\bm A\bm Q\bigr]^{12}_{ij}\Bigr|
&\lesssim
\frac{1}{N\sqrt N}\sum_{k=1}^{p}|w_k|
\frac{1}{N}\sup_{z\in\mathbb C^+_{\eta_0}}
\sum_{i,j}\bigl|\bigl[\bm Q\bm A^{(0)}\bm Q\bm A^{(0)}\bm Q\bigr]^{12}_{ij}\bigr| \notag\\
&\lesssim
\frac{1}{N\sqrt N}\sum_{i,j,k}\frac{|w_k|^3}{N}\eta_0^{-3}\\
&\lesssim \eta_0^{-3}N^{-1/2}.
\end{align}

\medskip
\noindent\textbf{Class II (one spike of type $12$ and one spike of type $23$).}
Consider the mixed contribution where the first $\bm A^{(0)}_{ijk}$ uses the $(1,2)$--block and the second uses the $(2,3)$--block.
A direct multiplication (as in the detailed $QAQAQ$ computations previously) yields a term of the form
\begin{equation}\label{eq:classII-form}
\bigl[\bm Q\bm A^{(0)}\bm Q\bm A^{(0)}\bm Q\bigr]^{12}_{ij}
\ \supset\
w_k u_i\,
Q^{11}_{ii}\,Q^{22}_{jj}\,Q^{32}_{kj}.
\end{equation}
Hence, using $|Q^{11}_{ii}|,|Q^{22}_{jj}|\le \eta_0^{-1} and \sup_{z}|Q^{32}_{kj}(z)| \le \|\bm Q^{32}\bm e_j^n\|_2 \le \eta_0^{-1}$,
\[
\sup_{z}\bigl|\bigl[\bm Q\bm A^{(0)}\bm Q\bm A^{(0)}\bm Q\bigr]^{12}_{ij}\bigr|
\ \lesssim\
|w_k||u_i|\,\eta_0^{-3}.
\]
Therefore,
\begin{align}\label{eq:classII-sum}
\frac{1}{N\sqrt N}\sum_{i,j,k}\sup_{z\in\mathbb C^+_{\eta_0}}
\Bigl|w_k\bigl[\bm Q\bm A\bm Q\bm A\bm Q\bigr]^{12}_{ij}\Bigr|
&\lesssim
\frac{1}{N\sqrt N}\sum_{k=1}^{p}|w_k|
\frac{1}{N}\sup_{z\in\mathbb C^+_{\eta_0}}
\sum_{i,j}\bigl|\bigl[\bm Q\bm A^{(0)}\bm Q\bm A^{(0)}\bm Q\bigr]^{12}_{ij}\bigr| \notag\\
&\lesssim
\frac{1}{N^2\sqrt N}\sum_{i,j,k}w_k^2|u_i|\eta_0^{-3}.
\end{align}
Using $\sum_i|u_i|\le \sqrt m\|\bm u\|_2\lesssim \sqrt N$ yields
\[
\text{Class II contribution}
\ \lesssim\
\frac{\eta_0^{-3}}{N^2\sqrt N}\,(n)\,(\sqrt m)
\ \lesssim\ \eta_0^{-3}N^{-1},
\]

\medskip
\noindent\textbf{Class III (one spike of type $13$ and one spike of type $23$).}
We now bound the Class~III terms where one factor $\bm A^{(0)}_{ijk}$ acts through the
$(1,3)$--block and the other through the $(2,3)$--block.
A direct multiplication (as in the detailed $QAQAQ$ computations previously) yields
representative contributions of the form
\begin{equation}\label{eq:classIII-rep}
\bigl[\bm Q\bm A^{(0)}\bm Q\bm A^{(0)}\bm Q\bigr]^{12}_{ij}(z)
\ \supset\
\,u_i v_j\,
Q^{11}_{ii}(z)\,Q^{22}_{jj}(z)\,Q^{32}_{kj}(z),
\end{equation}
and similarly finitely many terms where the last factor is replaced by an entry of a column of
$\bm Q^{13}(z)$ or $\bm Q^{33}(z)$. We treat (\ref{eq:classIII-rep}); the remaining ones are identical.

Fix $z\in\mathbb C^+_{\eta_0}$. Taking absolute values and using
$|Q^{11}_{ii}(z)|,|Q^{22}_{jj}(z)|\le \eta_0^{-1}$ gives
\[
\Bigl|w_k\bigl[\bm Q\bm A^{(0)}\bm Q\bm A^{(0)}\bm Q\bigr]^{12}_{ij}(z)\Bigr|
\ \lesssim\
|u_i||v_j|\,\eta_0^{-2}\,|w_k|\cdot |Q^{32}_{kj}(z)|.
\]
Summing first over $k$ \emph{with the weight $w_k$} and using the Cauchy--Schwarz inequality together with
the column $\ell^2$ bound (\ref{eq:col-l2}), we obtain
\begin{align}\label{eq:classIII-k-sum}
\sum_{k=1}^{p}|w_k|\cdot |Q^{32}_{kj}(z)|
&\le
\|\bm w\|_2\,\|\bm Q^{32}(z)\bm e_j^n\|_2
\ \le\ \eta_0^{-1}.
\end{align}
Therefore, for each fixed $(i,j)$ and $z\in\mathbb C^+_{\eta_0}$,
\[
\sum_{k=1}^{p}
\Bigl|w_k\bigl[\bm Q\bm A^{(0)}\bm Q\bm A^{(0)}\bm Q\bigr]^{12}_{ij}(z)\Bigr|
\ \lesssim\
|u_i||v_j|\,\eta_0^{-3}.
\]
Summing further over $(i,j)$ and using $\sum_i|u_i|\le \sqrt m\|\bm u\|_2$ and
$\sum_j|v_j|\le \sqrt n\|\bm v\|_2$, we deduce
\begin{align*}
\sum_{i,j,k}
\Bigl|w_k\bigl[\bm Q\bm A^{(0)}\bm Q\bm A^{(0)}\bm Q\bigr]^{12}_{ij}(z)\Bigr|
&\lesssim
\eta_0^{-3}\Bigl(\sum_{i=1}^{m}|u_i|\Bigr)\Bigl(\sum_{j=1}^{n}|v_j|\Bigr)\\
&\le
\eta_0^{-3}\cdot (\sqrt m\|\bm u\|_2)\cdot (\sqrt n\|\bm v\|_2)
\ \lesssim\ \eta_0^{-3}N,
\end{align*}
where we used $mn\asymp N^2$ and $\|\bm u\|_2=\|\bm v\|_2=\|\bm w\|_2=1$.

Therefore,
\[
\frac{1}{N^2\sqrt N}\sup_{z\in\mathbb C^+_{\eta_0}}
\sum_{i,j,k}\Bigl|w_k\bigl[\bm Q\bm A^{(0)}\bm Q\bm A^{(0)}\bm Q\bigr]^{12}_{ij}(z)\Bigr|
\ \lesssim\
\frac{1}{N^2\sqrt N}\cdot \eta_0^{-3}\,N
\ =\ \eta_0^{-3}\,N^{-3/2}.
\]
Therefore, all Class~III contributions are negligible after the normalized summation in
(\ref{eq:goal-QAQAQ}).

\medskip
\noindent\textbf{To Control  $\bm Q(\partial^2\bm N)\bm Q$ term.\;}
Fix $(i,j,k)$ and write $\partial:=\partial_{W_{ijk}}$.  
Recall that for $z\in\mathbb C^+_{\eta_0}$,
\[
\bm Q(z)=(\bm N-z\bm I)^{-1},
\qquad
\|\bm Q(z)\|\le \eta_0^{-1}.
\]
Using the explicit second--order expansion,
\begin{equation}\label{eq:d2N-explicit}
\partial^2 \bm N
=
\frac{1}{\sqrt N}\,\Phi_3\!\bigl(\textbf{W},\partial^2\bm u,\partial^2\bm v,\partial^2\bm w\bigr)
+\frac{2}{\sqrt N}\,\bm S_{ijk},
\end{equation}
where $\bm S_{ijk}$ is the sparse block matrix
\[
\bm S_{ijk}
=
\begin{bmatrix}
\bm 0 &
\partial w_k\,\bm e_i^{m}(\bm e_j^{n})^{\top} &
\partial v_j\,\bm e_i^{m}(\bm e_k^{p})^{\top} \\
\partial w_k\,\bm e_j^{n}(\bm e_i^{m})^{\top} &
\bm 0 &
\partial u_i\,\bm e_j^{n}(\bm e_k^{p})^{\top} \\
\partial v_j\,\bm e_k^{p}(\bm e_i^{m})^{\top} &
\partial u_i\,\bm e_k^{p}(\bm e_j^{n})^{\top} &
\bm 0
\end{bmatrix},
\]
we decompose
\[
\bm Q(\partial^2\bm N)\bm Q
=
\frac{1}{\sqrt N}\bm Q\Phi_3(\textbf{W},\partial^2\bm u,\partial^2\bm v,\partial^2\bm w)\bm Q
+\frac{2}{\sqrt N}\bm Q\bm S_{ijk}\bm Q.
\]

\medskip
\noindent\textbf{Step 1: the sparse spike term.}
Consider first the contribution of $\bm Q\bm S_{ijk}\bm Q$.  
For instance, the $(1,2)$--block of $\bm S_{ijk}$ equals
\[
(\bm S_{ijk})^{12}
=
\partial w_k\,\bm e_i^{m}(\bm e_j^{n})^{\top},
\]
hence
\[
\bigl[\bm Q\bm S_{ijk}\bm Q\bigr]^{12}_{ij}
=
\partial w_k\,
\bigl[\bm Q^{11}\bm e_i^{m}(\bm e_j^{n})^{\top}\bm Q^{22}\bigr]_{ij}
=
\partial w_k\,\bm Q^{11}_{ii}\bm Q^{22}_{jj}.
\]
Using $|\bm Q^{11}_{ii}|,|\bm Q^{22}_{jj}|\le \eta_0^{-1}$, we obtain
\begin{equation}\label{eq:QSQ-bound}
\sup_{z\in\mathbb C^+_{\eta_0}}
\bigl|\bigl[\bm Q(z)\bm S_{ijk}\bm Q(z)\bigr]^{12}_{ij}\bigr|
\le
|\partial w_k|\,\eta_0^{-2}.
\end{equation}
Therefore,
\begin{align}\label{eqA.41}
&\frac{1}{N\sqrt N}\sum_{i,j,k}
\sup_{z\in\mathbb C^+_{\eta_0}}
\Bigl|\,w_k\frac{2}{\sqrt N}
\bigl[\bm Q(z)\bm S_{ijk}\bm Q(z)\bigr]^{12}_{ij}\Bigr|
\nonumber\\
&\qquad\le
\frac{2\eta_0^{-2}}{N\sqrt N}
\sum_{i,j,k}\frac{|w_k||\partial w_k|}{\sqrt N}
=
\frac{2\eta_0^{-2}}{N^2}
\Bigl(\sum_{i=1}^{m}1\Bigr)
\Bigl(\sum_{j=1}^{n}1\Bigr)
\sum_{k=1}^{p}|w_k||\partial w_k|.
\end{align}
Since $\sum_k |w_k||\partial w_k|
\le \|\bm w\|_2\|\partial\bm w\|_2
=\|\partial\bm w\|_2$
and $\|\partial\bm w\|_2\lesssim N^{-1/2}$, then Eq. (\ref{eqA.41}) \( = \mathcal{O}(N^{-1/2})\).

The other blocks of $\bm S_{ijk}$ are treated identically, replacing
$\partial w_k$ by $\partial v_j$ or $\partial u_i$.

\medskip
\noindent\textbf{Step 2: the $\Phi_3(\textbf{W},\partial^2\bm u,\partial^2\bm v,\partial^2\bm w)$ term.}
We next consider
\[
\frac{1}{\sqrt N}\bm Q\Phi_3(\textbf{W},\partial^2\bm u,\partial^2\bm v,\partial^2\bm w)\bm Q.
\]
Recall that $\textbf{W}$ has independent entries with zero mean, unit variance and finite fourth moment. For any vector $\xi\in\mathbb R^N$, the normalized contraction
matrix $W(\xi)/\|\xi\|_2$ is an $N\times N$ random matrix with i.i.d.\ entries
satisfying the assumptions of Lemma~\ref{fact:opnorm}, and hence Lemma~\ref{fact:opnorm} applies
\[
\frac{\|\textbf{W}(\xi)\|}{\|\xi\|_2}\lesssim \sqrt N.
\]
Together with the bounds
\[
\|\partial^2\bm u\|_2+\|\partial^2\bm v\|_2+\|\partial^2\bm w\|_2\lesssim N^{-1},
\]
it follows that
\[
\|\Phi_3(\textbf{W},\partial^2\bm u,\partial^2\bm v,\partial^2\bm w)\|
\le
C\Bigl(
\|\textbf{W}(\partial^2\bm u)\|
+\|\textbf{W}(\partial^2\bm v)\|
+\|\textbf{W}(\partial^2\bm w)\|
\Bigr)
\lesssim
\mathcal{O}(\sqrt N)\cdot N^{-1}
=
\mathcal O(\frac{1}{\sqrt{N}}).
\] 

Proceeding with the normalized $(i,j,k)$--sum, we move the supremum outside and apply the
Cauchy--Schwarz inequality:
\begin{align*}
\frac{1}{N\sqrt N}\sum_{i,j,k}\sup_{z\in\mathbb C^+_{\eta_0}}
\Bigl|w_k\Bigl[\frac{1}{\sqrt N}\bm Q(z)\Phi_3(\cdot)\bm Q(z)\Bigr]^{12}_{ij}\Bigr|
&\le
\frac{1}{N\sqrt N}\sup_{z}\sum_{i,j,k}|w_k|\,
\Bigl|\Bigl[\frac{1}{\sqrt N}\bm Q\Phi_3(\cdot)\bm Q\Bigr]^{12}_{ij}\Bigr|\\
&\le
\frac{1}{N\sqrt N}\sup_{z}
\Bigl(\sum_{i,j,k}w_k^2\Bigr)^{1/2}
\Bigl(\sum_{i,j,k}
\Bigl|\Bigl[\frac{1}{\sqrt N}\bm Q\Phi_3(\cdot)\bm Q\Bigr]^{12}_{ij}\Bigr|^2
\Bigr)^{1/2}\\
&\lesssim
\sup_{z}\Bigl\|\frac{1}{\sqrt N}
\bigl[\bm Q(z)\Phi_3(\cdot)\bm Q(z)\bigr]^{12}\Bigr\|_F.
\end{align*}

Using $\|\bm Q(z)\|_F\lesssim \sqrt N\,\eta_0^{-1}$ and $\|\bm Q(z)\|\le \eta_0^{-1}$, we obtain
\[
\sup_{z}\Bigl\|\frac{1}{\sqrt N}
\bigl[\bm Q(z)\Phi_3(\cdot)\bm Q(z)\bigr]^{12}\Bigr\|_F
\le
\sup_{z}\frac{1}{\sqrt N}\|\bm Q(z)\|_F\,
\|\Phi_3(\cdot)\|\,\|\bm Q(z)\|
\lesssim
\eta_0^{-2}\,\|\Phi_3(\cdot)\|
=
\mathcal O(\frac{1}{\sqrt{N}}).
\]
Hence this contribution is negligible.

\medskip
\noindent\textbf{Conclusion.}
Combining the above estimates, we conclude that
\[
\frac{1}{N\sqrt N}\sum_{i,j,k}
\sup_{z\in\mathbb C^+_{\eta_0}}
\Bigl|\,w_k\bigl[\bm Q(z)(\partial^2\bm N)\bm Q(z)\bigr]^{12}_{ij}\Bigr|
=\mathcal O(\frac{1}{\sqrt{N}}),
\]
so the $\bm Q(\partial^2\bm N)\bm Q$ term is negligible. Then we complete the proof of the first part.

\subsubsection{Bounding the second part}\label{secA.4.2}
We next bound the mixed derivative term
\[
2\sum_{i,j,k}\sup_{z\in\mathbb C^+_{\eta_0}}
\Bigl|
\frac{\partial w_k}{\partial W_{ijk}}\,
\frac{\partial}{\partial W_{ijk}}Q^{12}_{ij}(z)
\Bigr|.
\]

%===========================================================
% All cross terms: use only the “pressed-back” bounds
%===========================================================

Recall Eq.~(\ref{eqA.27})
\[
\frac{\partial w_k}{\partial W_{ijk}}
=-\frac{1}{\sqrt N}\Bigl(A_{ijk}+B_{ijk}+C_{ijk}\Bigr),
\]
where
\[
\begin{aligned}
A_{ijk}&:= v_j w_k\Bigl(R^{13}_{ik}(\lambda)-u_i\,\bm u^\top \bm R^{13}_{\cdot,k}(\lambda)\Bigr),\\
B_{ijk}&:= u_i w_k\Bigl(R^{23}_{jk}(\lambda)-v_j\,\bm v^\top \bm R^{23}_{\cdot,k}(\lambda)\Bigr),\\
C_{ijk}&:= u_i v_j\Bigl(R^{33}_{kk}(\lambda)-w_k\,\bm w^\top \bm R^{33}_{\cdot,k}(\lambda)\Bigr).
\end{aligned}
\]
Set the projected columns
\[
\bm a_k:=(\bm I-\bm u\bm u^\top)\bm R^{13}_{\cdot,k},\qquad
\bm b_k:=(\bm I-\bm v\bm v^\top)\bm R^{23}_{\cdot,k},
\]
and the scalar residual
\[
\gamma_k:=R^{33}_{kk}(\lambda)-w_k\,\bm w^\top \bm R^{33}_{\cdot,k}(\lambda).
\]
Then $|A_{ijk}|=|v_j||w_k||a_{ik}|$, $|B_{ijk}|=|u_i||w_k||b_{jk}|$, and
$|C_{ijk}|=|u_i||v_j||\gamma_k|$.

Moreover
\[
\frac{\partial}{\partial W_{ijk}}Q^{12}_{ij}(z)
=-\frac{1}{\sqrt N}\sum_{\ell=1}^{6} S^{(\ell)}_{ijk}(z),
\]
with
\[
\begin{aligned}
S^{(1)}_{ijk}&:= w_k\,Q^{11}_{ii}(z)Q^{22}_{jj}(z),&
S^{(2)}_{ijk}&:= w_k\,(Q^{12}_{ij}(z))^2,\\
S^{(3)}_{ijk}&:= v_j\,Q^{11}_{ii}(z)Q^{32}_{kj}(z),&
S^{(4)}_{ijk}&:= v_j\,Q^{13}_{ik}(z)Q^{12}_{ij}(z),\\
S^{(5)}_{ijk}&:= u_i\,Q^{12}_{ij}(z)Q^{32}_{kj}(z),&
S^{(6)}_{ijk}&:= u_i\,Q^{13}_{ik}(z)Q^{22}_{jj}(z).
\end{aligned}
\]
Hence
\[
\frac{\partial w_k}{\partial W_{ijk}}\,
\frac{\partial}{\partial W_{ijk}}Q^{12}_{ij}(z)
=\frac{1}{N}\sum_{W\in\{A,B,C\}}\sum_{\ell=1}^{6} W_{ijk}\,S^{(\ell)}_{ijk}(z).
\]
\medskip
\noindent\textbf{Basic ``pressed-back'' bounds (used repeatedly).}
For $z,\lambda\in\mathbb C^+_{\eta_0}$,
\[
\|\bm Q(z)\|\le \eta_0^{-1},\qquad \|\bm R(\lambda)\|\le \eta_0^{-1},
\qquad
\|\bm Q(z)\bm e_\alpha\|_2\le \eta_0^{-1},\qquad
\|\bm R(\lambda)\bm e_\alpha\|_2\le \eta_0^{-1}.
\]
Ward/Frobenius:
\[
\|\bm Q(z)\|_F^2\lesssim N\eta_0^{-2},\qquad
\sum_{\alpha}|Q_{\alpha\alpha}(z)|^2\le \|\bm Q(z)\|_F^2\lesssim N\eta_0^{-2}.
\]
Column pressed-back (no $\sqrt N$ loss):
\[
\sum_{k}|w_k|\,|Q^{32}_{kj}(z)|
\le \|\bm w\|_2\,\|\bm Q^{32}(z)\bm e_j\|_2\le \eta_0^{-1},
\qquad
\sum_{i}|u_i|\,|Q^{13}_{ik}(z)|
\le \|\bm u\|_2\,\|\bm Q^{13}(z)\bm e_k\|_2\le \eta_0^{-1}.
\]
Also $\|\bm a_k\|_2\le \|\bm R^{13}_{\cdot,k}\|_2\le \eta_0^{-1}$,
$\|\bm b_k\|_2\le \eta_0^{-1}$, and $|\gamma_k|\le |R^{33}_{kk}|+|w_k|\|\bm R^{33}_{\cdot,k}\|_2\le 2\eta_0^{-1}$.

\medskip
\noindent\textbf{(I) Terms $A\times S^{(\ell)}$.}
For $\ell=1,\dots,6$, define
\[
\mathcal S_{A,\ell}:=\sum_{i,j,k}\sup_{z\in\mathbb C^+_{\eta_0}}
\bigl|A_{ijk}S^{(\ell)}_{ijk}(z)\bigr|.
\]

\smallskip
\noindent\emph{Term $\ell=1$.}
\begin{align*}
\mathcal S_{A,1}
&=\sum_{i,j,k}\sup_z |v_j||w_k|^2|a_{ik}|\,|Q^{11}_{ii}||Q^{22}_{jj}| \\
&\le \sum_k |w_k|^2\Big(\sum_i |a_{ik}|\sup_z|Q^{11}_{ii}|\Big)
     \Big(\sum_j |v_j|\sup_z|Q^{22}_{jj}|\Big)\\
&\le \sum_k |w_k|^2\|\bm a_k\|_2
     \Big(\sum_i \sup_z|Q^{11}_{ii}|^2\Big)^{1/2}
     \|\bm v\|_2
     \Big(\sum_j \sup_z|Q^{22}_{jj}|^2\Big)^{1/2}
 \ \lesssim\ N\,\eta_0^{-3}.
\end{align*}

\smallskip
\noindent\emph{Term $\ell=2$.}
\begin{align*}
\mathcal S_{A,2}
&=\sum_{i,j,k}\sup_z |v_j||w_k|^2|a_{ik}|\,|Q^{12}_{ij}|^2 \\
&\le \sum_k |w_k|^2\|\bm a_k\|_2
\sum_j |v_j|\Big(\sum_i \sup_z|Q^{12}_{ij}|^4\Big)^{1/2} \\
&\le \sum_k |w_k|^2\|\bm a_k\|_2
\sum_j |v_j|
\Big(\eta_0^{-2}\cdot \|\bm Q^{12}(z)\bm e_j\|_2^2\Big)^{1/2}
\ \lesssim\ \sqrt N\,\eta_0^{-3}.
\end{align*}

\smallskip
\noindent\emph{Term $\ell=3$.}
\begin{align*}
\mathcal S_{A,3}
&=\sum_{i,j,k}\sup_z |v_j|^2|w_k||a_{ik}|\,
|Q^{11}_{ii}||Q^{32}_{kj}| \\
&\le \eta_0^{-1}\sum_j |v_j|^2\|\bm a_k\|_2
\Big(\sum_i \sup_z|Q^{11}_{ii}|^2\Big)^{1/2}
\ \lesssim\ \sqrt N\,\eta_0^{-3}.
\end{align*}

\smallskip
\noindent\emph{Term $\ell=4$.}
\begin{align*}
\mathcal S_{A,4}
&=\sum_{i,j,k}\sup_z |v_j|^2|w_k||a_{ik}|\,
|Q^{13}_{ik}||Q^{12}_{ij}| \\
&\lesssim
\Big(\sum_j |v_j|^2\Big)\cdot
(\eta_0^{-1}\eta_0^{-1})\cdot
\sqrt N\,\eta_0^{-1}
\ \lesssim\ \sqrt N\,\eta_0^{-3}.
\end{align*}

\smallskip
\noindent\emph{Term $\ell=5$.}
\begin{align*}
\mathcal S_{A,5}
&=\sum_{i,j,k}\sup_z |v_j||u_i||w_k||a_{ik}|\,
|Q^{12}_{ij}||Q^{32}_{kj}| \\
&\le \eta_0^{-1}\sum_j
\|\bm a_k\|_2
\Big(\sum_i \sup_z|Q^{12}_{ij}|^2\Big)^{1/2}
\ \lesssim\ \sqrt N\,\eta_0^{-3}.
\end{align*}

\smallskip
\noindent\emph{Term $\ell=6$.}
\begin{align*}
\mathcal S_{A,6}
&=\sum_{i,j,k}\sup_z |v_j||u_i||w_k||a_{ik}|\,
|Q^{13}_{ik}||Q^{22}_{jj}| \\
&\lesssim
(\sqrt N\,\eta_0^{-1})\cdot(\eta_0^{-1}\eta_0^{-1})
\ \lesssim\ \sqrt N\,\eta_0^{-3}.
\end{align*}

\medskip
\noindent\textbf{(II) Terms $B\times S^{(\ell)}$ (symmetric bounds).}
Define
\(
\mathcal S_{B,\ell}:=\sum_{i,j,k}\sup_{z} |B_{ijk}S^{(\ell)}_{ijk}(z)|.
\)
By symmetry (swap $(i,u,13)\leftrightarrow (j,v,23)$), the same arguments yield
\[
\mathcal S_{B,1}\lesssim N\eta_0^{-3},\quad
\mathcal S_{B,2}\lesssim \sqrt N\,\eta_0^{-3},\quad
\mathcal S_{B,3}\lesssim \sqrt N\,\eta_0^{-3},\quad
\mathcal S_{B,4}\lesssim \sqrt N\,\eta_0^{-3},\quad
\mathcal S_{B,5}\lesssim \sqrt N\,\eta_0^{-3},\quad
\mathcal S_{B,6}\lesssim \sqrt N\,\eta_0^{-3}.
\]

\medskip
\noindent\textbf{(III) Terms $C\times S^{(\ell)}$.}
Since $|\gamma_k|\le 2\eta_0^{-1}$, we have $|C_{ijk}|\le 2\eta_0^{-1}|u_i||v_j|$.
Let
\(
\mathcal S_{C,\ell}:=\sum_{i,j,k}\sup_{z} |C_{ijk}S^{(\ell)}_{ijk}(z)|.
\)
Then, using only the same pressed-back column bounds as above,
\[
\mathcal S_{C,1}\lesssim N\eta_0^{-3},\quad
\mathcal S_{C,2}\lesssim \sqrt N\,\eta_0^{-3},\quad
\mathcal S_{C,3}\lesssim \sqrt N\,\eta_0^{-3},\quad
\mathcal S_{C,4}\lesssim \sqrt N\,\eta_0^{-3},\quad
\mathcal S_{C,5}\lesssim \sqrt N\,\eta_0^{-3},\quad
\mathcal S_{C,6}\lesssim \sqrt N\,\eta_0^{-3}.
\]

\medskip
\noindent\textbf{Conclusion.}
Combining all $18$ cross terms and recalling the prefactor $1/N$ in the product,
\[
2\sum_{i,j,k}\sup_{z\in\mathbb C^+_{\eta_0}}
\Bigl|
\frac{\partial w_k}{\partial W_{ijk}}\,
\frac{\partial}{\partial W_{ijk}}Q^{12}_{ij}(z)
\Bigr|
\ \le\ \frac{2}{N}\sum_{W\in\{A,B,C\}}\sum_{\ell=1}^{6}\mathcal S_{W,\ell}
\ \lesssim\ \eta_0^{-3},
\]
which means
\begin{equation}
N^{-3/2}2\sum_{i,j,k}\sup_{z\in\mathbb C^+_{\eta_0}}
\Bigl|
\frac{\partial w_k}{\partial W_{ijk}}\,
\frac{\partial}{\partial W_{ijk}}Q^{12}_{ij}(z)
\Bigr| \xrightarrow{N \to \infty} 0
\end{equation}

\subsubsection{Bounding the third part}\label{secA.4.3}
\paragraph*{Goal}
We show that
\[
N^{-3/2}\sum_{i,j,k}\sup_{z\in\mathbb C^+_{\eta_0}}
\Bigl|Q^{12}_{ij}(z)\,\partial_{ijk}^2 w_k\Bigr|
=\mathcal{O}(\frac{1}{\sqrt N}),
\qquad \partial_{ijk}:=\frac{\partial}{\partial W_{ijk}}.
\]

\paragraph*{Step 1: $\ell^2$ bounds for first derivatives from Eq.~(\ref{eq3.4})}
From (\ref{eq3.4}) we have, for each fixed $(i,j,k)$,
\[
\begin{bmatrix}\partial \bm u\\ \partial \bm v\\ \partial \bm w\end{bmatrix}
=
-\frac{1}{\sqrt N}\bm R(\lambda)
\begin{bmatrix}
v_j w_k(\bm e_i^m-u_i\bm u)\\
u_i w_k(\bm e_j^n-v_j\bm v)\\
u_i v_j(\bm e_k^p-w_k\bm w)
\end{bmatrix},
\qquad \|\bm R(\lambda)\|\le \eta_0^{-1}.
\]
Since $\|\bm e_i^m-u_i\bm u\|_2\le \sqrt2$ (and similarly for the other two),
and $|u_i|,|v_j|,|w_k|\le 1$, we get the uniform bound
\begin{equation}\label{eq:first-der-norm}
\|\partial \bm u\|_2+\|\partial \bm v\|_2+\|\partial \bm w\|_2
\ \lesssim\ \frac{1}{\sqrt N}\,\eta_0^{-1}.
\end{equation}

We will repeatedly use the Ward/column bound for $Q(z)$:
\begin{equation}\label{eq:Ward}
\|Q(z)\bm e_\alpha\|_2\le \eta_0^{-1},
\qquad
\|Q^{12}(z)\bm e_j^n\|_2\le \eta_0^{-1},
\qquad z\in\mathbb C^+_{\eta_0}.
\end{equation}
Moreover, using $\partial Q(z)=-Q(z)(\partial N)Q(z)$ and $\|\partial N\|\lesssim N^{-1/2}$,
\begin{equation}\label{eq:dQop}
\|\partial Q(z)\|\lesssim N^{-1/2}\eta_0^{-2},
\qquad
\|\partial \bm R(\lambda)\|\lesssim N^{-1/2}\eta_0^{-2}.
\end{equation}

\paragraph*{Step 2: full expansion of $\partial^2 w_k$}
Write (as in the main text)
\[
\partial w_k
=-\frac{1}{\sqrt N}\Bigl(
v_jw_k\,\Delta^{13}_{ik}
+u_iw_k\,\Delta^{23}_{jk}
+u_iv_j\,\Delta^{33}_{k}
\Bigr),
\]
where $\Delta^{13}_{ik}:=R^{13}_{ik}-u_i\,\bm u^\top\bm R^{13}_{\cdot,k}$,
$\Delta^{23}_{jk}:=R^{23}_{jk}-v_j\,\bm v^\top\bm R^{23}_{\cdot,k}$,
$\Delta^{33}_{k}:=R^{33}_{kk}-w_k\,\bm w^\top\bm R^{33}_{\cdot,k}$.
Differentiating again w.r.t.\ the same $W_{ijk}$ gives the \emph{fully expanded} identity
\begin{equation}\label{eq:d2w-9}
\begin{aligned}
\partial^2 w_k
=-\frac{1}{\sqrt N}\Bigl[
&(\partial v_j)w_k\Delta^{13}_{ik}
+v_j(\partial w_k)\Delta^{13}_{ik}
+v_jw_k(\partial \Delta^{13}_{ik}) \\
&+(\partial u_i)w_k\Delta^{23}_{jk}
+u_i(\partial w_k)\Delta^{23}_{jk}
+u_iw_k(\partial \Delta^{23}_{jk}) \\
&+(\partial u_i)v_j\Delta^{33}_{k}
+u_i(\partial v_j)\Delta^{33}_{k}
+u_iv_j(\partial \Delta^{33}_{k})
\Bigr].
\end{aligned}
\end{equation}
Hence, for each $z\in\mathbb C^+_{\eta_0}$,
\[
\sum_{i,j,k}\bigl|Q^{12}_{ij}(z)\,\partial^2 w_k\bigr|
\le \frac{1}{\sqrt N}\sum_{\ell=1}^9 S_\ell(z),
\]
where the $S_\ell$ correspond to the nine terms in (\ref{eq:d2w-9}).
We bound $\sup_z S_\ell(z)$ one by one by ``pushing back'' sums using (\ref{eq:first-der-norm})-(\ref{eq:dQop}).

\paragraph*{Step 3: bounds for the nine terms}
Below all suprema are over $z\in\mathbb C^+_{\eta_0}$.

\emph{(Term 1) } $S_1:=\sum_{i,j,k}|Q^{12}_{ij}|\ |\partial v_j|\ |w_k|\ |\Delta^{13}_{ik}|$.
First sum in $j$ by Cauchy--Schwarz:
\[
\sum_j |Q^{12}_{ij}|\,|\partial v_j|
\le \| (Q^{12})^\top \bm e_i^m\|_2\ \|\partial \bm v\|_2
=\|Q^{21}\bm e_i^m\|_2\ \|\partial \bm v\|_2
\le \eta_0^{-1}\|\partial\bm v\|_2.
\]
Then sum in $(i,k)$ using $\sum_i|\Delta^{13}_{ik}|\le \sqrt m\,\|\Delta^{13}_{\cdot k}\|_2\le \sqrt N\,\eta_0^{-1}$
and $\sum_k|w_k|\le \sqrt p\|\bm w\|_2\lesssim \sqrt N$:
\[
S_1\lesssim \eta_0^{-1}\|\partial\bm v\|_2\cdot
\sum_k|w_k|\sum_i|\Delta^{13}_{ik}|
\lesssim \eta_0^{-1}\|\partial\bm v\|_2\cdot (\sqrt N)(\sqrt N\eta_0^{-1})
= N\,\eta_0^{-2}\|\partial\bm v\|_2.
\]
Using (\ref{eq:first-der-norm}), $\|\partial\bm v\|_2\lesssim N^{-1/2}\eta_0^{-1}$, hence
\begin{equation}\label{eq:S1}
S_1\lesssim \sqrt N\,\eta_0^{-3}.
\end{equation}

\emph{(Term 2) } $S_2:=\sum_{i,j,k}|Q^{12}_{ij}|\ |v_j|\ |\partial w_k|\ |\Delta^{13}_{ik}|$.
Sum in $k$ by the Cauchy--Schwarz inequality:
\[
\sum_k |\partial w_k|\,|\Delta^{13}_{ik}|
\le \|\partial\bm w\|_2\,\|\Delta^{13}_{i,\cdot}\|_2
\le \|\partial\bm w\|_2\,\|\Delta^{13}\|_F
\le \|\partial\bm w\|_2\,\|\bm R\|_F
\lesssim \|\partial\bm w\|_2\cdot \sqrt N\,\eta_0^{-1}.
\]
Therefore
\[
S_2\lesssim (\|\partial\bm w\|_2\sqrt N\,\eta_0^{-1})
\sum_{i,j}|Q^{12}_{ij}|\,|v_j|.
\]
Now $\sum_{i,j}|Q^{12}_{ij}|\,|v_j|
=\sum_j |v_j|\sum_i |Q^{12}_{ij}|
\le \sum_j |v_j|\,\sqrt m\,\|Q^{12}\bm e_j\|_2
\le (\sum_j|v_j|)\sqrt N\,\eta_0^{-1}
\le \sqrt N\cdot \sqrt N\,\eta_0^{-1}
= N\,\eta_0^{-1}.$

Combining and using (\ref{eq:first-der-norm}), $\|\partial\bm w\|_2\lesssim N^{-1/2}\eta_0^{-1}$, yields
\begin{equation}\label{eq:S2}
S_2\lesssim \sqrt N\,\eta_0^{-3}.
\end{equation}

\emph{(Term 3) } $S_3:=\sum_{i,j,k}|Q^{12}_{ij}|\ |v_j|\ |w_k|\ |\partial\Delta^{13}_{ik}|$.
Push back in $k$:
\[
\sum_k |w_k|\,|\partial\Delta^{13}_{ik}|
\le \|\partial\Delta^{13}_{i,\cdot}\|_2
\le \|\partial\Delta^{13}\|_F.
\]
We claim $\|\partial\Delta^{13}\|_F\lesssim \eta_0^{-2}$.
Indeed, $\Delta^{13}$ is a linear combination of entries of $\bm R$, and thus
$\|\partial\Delta^{13}\|_F\lesssim \|\partial\bm R\|_F + \|\partial\bm u\|_2\,\|\bm R\|_F$.
Using (\ref{eq:dQop}), $\|\partial\bm R\|_F\le \sqrt N\|\partial\bm R\|\lesssim \eta_0^{-2}$,
and (\ref{eq:first-der-norm}), $\|\partial\bm u\|_2\|\bm R\|_F\lesssim (N^{-1/2}\eta_0^{-1})(\sqrt N\eta_0^{-1})=\eta_0^{-2}$,
hence $\|\partial\Delta^{13}\|_F\lesssim \eta_0^{-2}$.
Therefore
\[
S_3\lesssim \eta_0^{-2}\sum_{i,j}|Q^{12}_{ij}|\,|v_j|
\lesssim \eta_0^{-2}\cdot N\,\eta_0^{-1}
= N\,\eta_0^{-3}.
\]
So
\begin{equation}\label{eq:S3}
S_3\lesssim N\,\eta_0^{-3}.
\end{equation}

\emph{(Terms 4--6)} (the $23$-block terms) are handled identically to Terms 1-3 by symmetry, giving
\begin{equation}\label{eq:S456}
S_4\lesssim \sqrt N\,\eta_0^{-3},\qquad
S_5\lesssim \sqrt N\,\eta_0^{-3},\qquad
S_6\lesssim N\,\eta_0^{-3}.
\end{equation}

\emph{(Term 7)} $S_7:=\sum_{i,j,k}|Q^{12}_{ij}|\ |\partial u_i|\ |v_j|\ |\Delta^{33}_{kk}|$.
First sum in $(i,j)$ via operator norm:
\[
\sum_{i,j}|Q^{12}_{ij}|\,|\partial u_i|\,|v_j|
\le \|\partial\bm u\|_2\,\|Q^{12}\bm v\|_2
\le \|\partial\bm u\|_2\,\|Q\|\,\|\bm v\|_2
\le \eta_0^{-1}\|\partial\bm u\|_2.
\]
Next sum in $k$ by the Cauchy--Schwarz inequality:
\[
\sum_k |\Delta^{33}_{kk}|
\le \sqrt p\Big(\sum_k|\Delta^{33}_{kk}|^2\Big)^{1/2}
\le \sqrt N\,\|\bm R^{33}\|_F
\le \sqrt N\,\|\bm R\|_F
\lesssim \sqrt N\cdot \sqrt N\,\eta_0^{-1}
= N\,\eta_0^{-1}.
\]
Thus
\[
S_7\lesssim (\eta_0^{-1}\|\partial\bm u\|_2)\cdot (N\eta_0^{-1})
= N\,\eta_0^{-2}\|\partial\bm u\|_2
\lesssim \sqrt N\,\eta_0^{-3}.
\]
So
\begin{equation}\label{eq:S7}
S_7\lesssim \sqrt N\,\eta_0^{-3}.
\end{equation}

\emph{(Term 8)} $S_8:=\sum_{i,j,k}|Q^{12}_{ij}|\ |u_i|\ |\partial v_j|\ |\Delta^{33}_{kk}|$
is identical to $S_7$, giving
\begin{equation}\label{eq:S8}
S_8\lesssim \sqrt N\,\eta_0^{-3}.
\end{equation}

\emph{(Term 9)} 
$S_9:=\sum_{i,j,k}|Q^{12}_{ij}|\ |u_i|\ |v_j|\ |\partial\Delta^{33}_{kk}|$.
First sum in $(i,j)$:
\[
\sum_{i,j}|Q^{12}_{ij}|\,|u_i|\,|v_j|
\le \|\bm u\|_2\,\|Q^{12}\bm v\|_2
\le \|Q\|\,\|\bm u\|_2\,\|\bm v\|_2
\le \eta_0^{-1}.
\]
Next sum over $k$. Using the Cauchy--Schwarz inequality and the Frobenius norm bound for diagonal entries, we have
\[
\sum_{k=1}^p |\partial\Delta^{33}_{kk}|
\le \sqrt p\Big(\sum_{k=1}^p|\partial\Delta^{33}_{kk}|^2\Big)^{1/2}
\le \sqrt p\,\|\partial\Delta^{33}\|_F .
\]
Recall that
\(
\Delta^{33}_{kk}
=R^{33}_{kk}-w_k\,\bm w^\top \bm R^{33}_{\cdot,k}.
\)
Differentiating and applying the triangle inequality yields
\[
\|\partial\Delta^{33}\|_F
\le
\|\partial\bm R^{33}\|_F
+2\|\partial\bm w\|_2\,\|\bm R^{33}\|_F
+\|\bm w\|_2\,\|\partial\bm R^{33}\|_F.
\]

Because $\|\bm w\|_2=1$, $\|\bm R(\lambda)\|\le \eta_0^{-1}$,
the resolvent derivative identity
\(
\partial\bm R(\lambda)=-\bm R(\lambda)(\partial\textbf{T})\bm R(\lambda)
\)
together with $\|\partial\textbf{T}\|\lesssim N^{-1/2}$, we obtain
\[
\|\partial\bm R^{33}\|_F
\le \|\partial\bm R(\lambda)\|_F
\le \sqrt N\,\|\partial\bm R(\lambda)\|
\lesssim \sqrt N\cdot N^{-1/2}\eta_0^{-2}
=\eta_0^{-2}.
\]
Moreover, $\|\bm R^{33}\|_F\le \|\bm R\|_F\lesssim \sqrt N\,\eta_0^{-1}$
and $\|\partial\bm w\|_2\lesssim N^{-1/2}\eta_0^{-1}$ imply
\[
\|\partial\bm w\|_2\,\|\bm R^{33}\|_F
\lesssim
(N^{-1/2}\eta_0^{-1})(\sqrt N\,\eta_0^{-1})
=\eta_0^{-2}.
\]
Combining the above estimates gives
\[
\|\partial\Delta^{33}\|_F\lesssim \eta_0^{-2}.
\]
Therefore, using $p\asymp N$,
\[
\sum_{k=1}^p |\partial\Delta^{33}_{kk}|
\lesssim
\sqrt N\,\eta_0^{-2},
\]
and hence
\begin{equation}\label{eq:S9}
S_9
\lesssim
\eta_0^{-1}\cdot \sqrt N\,\eta_0^{-2}
=
\sqrt N\,\eta_0^{-3}.
\end{equation}

\paragraph*{Step 4: conclude}
Combining (\ref{eq:S1})-(\ref{eq:S3}), (\ref{eq:S456}), (\ref{eq:S7})-(\ref{eq:S9}), we obtain
\[
\sup_{z}\sum_{i,j,k}|Q^{12}_{ij}(z)\,\partial^2 w_k|
\ \le\ \frac{1}{\sqrt N}\sum_{\ell=1}^9 \sup_z S_\ell(z)
\ \lesssim\ \frac{1}{\sqrt N}\bigl(\sqrt N\,\eta_0^{-3}+N\,\eta_0^{-3}\bigr)
\ \lesssim\ \eta_0^{-3}\sqrt N.
\]
Therefore,
\[
N^{-3/2}\sum_{i,j,k}\sup_{z\in\mathbb C^+_{\eta_0}}
|Q^{12}_{ij}(z)\,\partial^2 w_k|
\ \lesssim\ N^{-1}\eta_0^{-3}=\mathcal{O}(N^{-1}).
\]

\subsubsection{Conclusion}
From the above analysis and Eq.~(\ref{eqA.26}), taking \(N \to \infty\) we have 
\begin{equation}
\frac{1}{N\sqrt N}\mathbb{E}\!\left[
W_{ijk}\, w_k\bm{Q}^{12}_{ij}(z)
\right]
\xrightarrow{a.s.} -g_1(z)g_2(z)
\end{equation}
It is similar to get
\begin{equation}
    \frac{1}{N\sqrt{N}}
\sum_{i,j,k}
\mathbb{E}\!\left[
W_{ijk}\, v_j\, \bm{Q}^{13}_{ij}(z)
\right]\xrightarrow{a.s.}-g_1(z)g_3(z).
\end{equation}
Then from Eq.~(\ref{eq:A.25}) we have
\begin{equation}
    -g_1(z)g_2(z)-g_1(z)g_3(z) -zg_1(z) = c_1.
\end{equation}
Similarly, we have
\[
\begin{cases}
-\,g_1(z)\bigl(g_2(z)+g_3(z)\bigr) - z\,g_1(z) = c_1,\\
-\,g_2(z)\bigl(g_1(z)+g_3(z)\bigr) - z\,g_2(z) = c_2,\\
-\,g_3(z)\bigl(g_1(z)+g_2(z)\bigr) - z\,g_3(z) = c_3,
\end{cases}
\]
which implies 
\[
g_i(z)=\frac{g(z)+z}{2}- \frac{\sqrt{4c_i+\bigl(g(z)+z\bigr)^2}}{2},
\]
with $g(z)$ being the solution of the equation
\[
g(z)=\sum_{i=1}^3 g_i(z),
\]
for \(z \in \mathbb C^+_{\eta_0}\).

The system admits a unique solution $(g_1(z),g_2(z),g_3(z))\in(\mathbb C^+)^3$
for every $z\in\mathbb C^+$, which is the analytic continuation of the solution
constructed on $\mathbb C^+_{\eta_0}$. In particular, the representation
\[
g_i(z)=\frac{g(z)+z-\sqrt{(g(z)+z)^2+4c_i}}{2}
\]
holds for all $z\in\mathbb C^+$, where the square root is taken with the branch
that ensures $\Im g_i(z)>0$ for $\Im z>0$.

\subsection{Proof of Theorem \ref{thm3.4}}\label{appA.5}

Given the identities in Eq.~(\ref{eq3.2}), we have
\begin{equation}
\label{eq:42}
\lambda\langle u,x\rangle
= \bm x^{\top}\textbf{T}(\bm v)\bm w
= \beta\langle \bm v,\bm y\rangle\langle \bm w,\bm z\rangle
+ \frac{1}{\sqrt{N}}\,\bm x^{\top}\textbf{W}(\bm v)\bm w, 
\end{equation}
with $\lambda$, $\langle \bm u,\bm x\rangle$, $\langle \bm v,\bm y\rangle$ and $\langle \bm w,\bm z\rangle$
converging almost surely to their asymptotic limits
$\lambda^\infty(\beta)$, $a_x^\infty(\beta)$, $a_y^\infty(\beta)$ and
$a_z^\infty(\beta)$, respectively.
To characterize such limits we need to evaluate the expectation of
$\frac{1}{\sqrt{N}}\bm x^{\top}\textbf{W}(\bm v)\bm w$.

Indeed,
\begin{align*}
\frac{1}{\sqrt{N}}\mathbb{E}\big[\bm x^{\top}\textbf{W}(\bm v)\bm w\big]
&= \frac{1}{\sqrt{N}}\sum_{ijk} x_i \mathbb{E}\big[v_j w_k W_{ijk}\big] \\
&=\frac{1}{\sqrt{N}}\sum_{ijk}x_i\mathbb{E}[\frac{\partial}{\partial W_{ijk}}v_jw_k] + \frac{1}{\sqrt{N}}\sum_{ijk}x_i\varepsilon_{ijk}^{(2)}\\
&= \frac{1}{\sqrt{N}}\sum_{ijk} x_i \mathbb{E}\!\left[
\frac{\partial v_j}{\partial W_{ijk}} w_k \right]
+ \frac{1}{\sqrt{N}}\sum_{ijk} x_i \mathbb{E}\!\left[
v_j \frac{\partial w_k}{\partial W_{ijk}} \right] +\frac{1}{\sqrt{N}}\sum_{ijk}x_i\varepsilon_{ijk}^{(2)},
\end{align*}
where \(\varepsilon_{ijk}^{(2)}(z) \leq C_1 \, \sup_{z \in \mathbb C^+_{\eta_0}} |g^{(2)}(z)| \, \mathbb{E}[|W_{ijk}|^{3}] = C_1^* \sup_{z \in \mathbb C^+_{\eta_0}} |\frac{\partial^2}{(\partial W_{ijk})^2}v_jw_k|\) and the second equality is obtained by Lemma~\ref{lem2.3}.

From Eq.~(\ref{eq3.4}), we have
\begin{align*}
\frac{\partial v_j}{\partial W_{ijk}}
&= -\frac{1}{\sqrt{N}}
\big(v_j w_k (R^{12}_{ij}(\lambda)-u_i u^{\top}\bm R^{12}_{\cdot j}(\lambda))\big) \\
&\quad -\frac{1}{\sqrt{N}}
\big(u_i w_k (R^{22}_{jj}(\lambda)-v_j v^{\top}\bm R^{22}_{\cdot j}(\lambda))\big)
-\frac{1}{\sqrt{N}}
\big(u_i v_j (R^{23}_{jk}(\lambda)-w_k w^{\top}\bm R^{23}_{\cdot k}(\lambda))\big).
\end{align*}

Hence, defining
\[
A \equiv \frac{1}{\sqrt{N}}\sum_{ijk} x_i \mathbb{E}\!\left[
\frac{\partial v_j}{\partial W_{ijk}} w_k \right]
= A_1 + A_2 + A_3 ,
\]
we decompose the expression into three terms $A_1,A_2,A_3$.
The terms $A_1$ and $A_3$ vanish asymptotically, and only $A_2$ contains
non-vanishing contributions.

Indeed,
\begin{align*}
A_1
&= -\frac{1}{N}\sum_{ijk}\mathbb{E}\big[x_i w_k v_j w_k R^{12}_{ij}(\lambda)\big]
+ \frac{1}{N}\sum_{ijkl}\mathbb{E}\big[x_i w_k v_j w_k u_i u_l R^{12}_{lj}(\lambda)\big] \\
&= -\frac{1}{N}\mathbb{E}\big[\bm x^{\top}\bm R^{12}(\lambda)\bm v\big]
+ \frac{1}{N}\mathbb{E}\big[\langle \bm x,\bm u\rangle\,\bm u^{\top}\bm R^{12}(\lambda)\bm v\big]
\;\xrightarrow[N\to\infty]{}\; 0 ,
\end{align*}
since $x,u,v$ have bounded norms and $\bm R(\lambda)$ has bounded spectral norm
for $\lambda$ outside the support of $\frac{1}{\sqrt{N}}\Phi_3(\textbf{W},\bm u,\bm v,\bm w)$
(through the identity in Eq.~(\ref{eqA.24})).

Similarly, for $A_3$ we have
\begin{align*}
A_3
&= -\frac{1}{N}\sum_{ijk}\mathbb{E}\big[x_i w_k u_i v_j R^{23}_{jk}(\lambda)\big]
+ \frac{1}{N}\sum_{ijkl}\mathbb{E}\big[x_i w_k u_i v_j w_k w_l R^{23}_{jl}(\lambda)\big] \\
&= -\frac{1}{N}\mathbb{E}\big[\langle \bm x,\bm u\rangle \bm v^{\top}\bm R^{23}(\lambda)\bm w\big]
+ \frac{1}{N}\mathbb{E}\big[\langle \bm x,\bm u\rangle \bm v^{\top}\bm R^{23}(\lambda)\bm w\big]
\;\xrightarrow[N\to\infty]{}\; 0 .
\end{align*}

Now $A_2$ does not vanish. Precisely,
\begin{align*}
A_2
&= -\frac{1}{N}\sum_{ijk}\mathbb{E}\big[x_i w_k u_i w_k R^{22}_{jj}(\lambda)\big]
+ \frac{1}{N}\sum_{ijkl}\mathbb{E}\big[x_i w_k u_i w_k v_j v_l R^{22}_{jl}(\lambda)\big] \\
&= -\frac{1}{N}\mathbb{E}\big[\langle \bm x,\bm u\rangle\,\mathrm{tr}\,\bm R^{22}(\lambda)\big]
+ \frac{1}{N}\mathbb{E}\big[\langle \bm x,\bm u\rangle \bm v^{\top}\bm R^{22}(\lambda)\bm v\big] \\
&\xrightarrow[N\to\infty]{}
- g_2(\lambda)\mathbb{E}[\langle \bm x,\bm u\rangle] + \mathcal{O}(N^{-1}),
\end{align*}
where the last line follows from the fact that
$\frac{1}{N}\mathrm{tr}\,\bm R^{22}(\lambda)\xrightarrow{\text{a.s.}} g_2(\lambda)$.

Similarly, we find that
\[
B
= \frac{1}{\sqrt{N}}\sum_{ijk} x_i
\mathbb{E}\!\left[ v_j \frac{\partial w_k}{\partial W_{ijk}} \right]
\xrightarrow[N\to\infty]{}
- g_3(\lambda)\mathbb{E}[\langle \bm x,\bm u\rangle] + \mathcal{O}(N^{-1}).
\]
Now we need to bound \(\varepsilon_{ijk}^2\). It's enough to bound \(\frac{1}{\sqrt{N}}\sum_{ijk}|x_i|\sup_{z \in \mathbb C^+_{\eta_0}} |\frac{\partial^2}{\partial W_{ijk}^2}v_jw_k|\), we need to prove that 
\begin{equation}\label{eqA.63}
    \frac{1}{\sqrt{N}}\sum_{ijk}|x_i|\sup_{z \in \mathbb C^+_{\eta_0}} |\frac{\partial^2}{\partial W_{ijk}^2}v_jw_k| = \mathcal O(\frac{1}{\sqrt{N}}).
\end{equation}

We proceed as in the handwritten argument.
By the product rule,
\begin{equation}
\label{eq:prod2}
\frac{\partial^2}{\partial W_{ijk}^2}(v_j w_k)
=
w_k\,\frac{\partial^2 v_j}{\partial W_{ijk}^2}
+2\,\frac{\partial v_j}{\partial W_{ijk}}
\frac{\partial w_k}{\partial W_{ijk}}
+v_j\,\frac{\partial^2 w_k}{\partial W_{ijk}^2}.
\end{equation}
Hence the left-hand side of Eq.~(\ref{eqA.63}) is bounded by the sum of three terms,
denoted by $I_1+2I_2+I_3$, corresponding to the three contributions in
(\ref{eq:prod2}).

\medskip
\noindent\textbf{Step 1: bound of the mixed term $I_2$.}
From Eq.~(\ref{eq3.4}), each first derivative has the structure
\[
\frac{\partial v_j}{\partial W_{ijk}}
= -\frac{1}{\sqrt{N}}
\Bigl(
v_j w_k\,\Delta^{(12)}_{ij}
+u_i w_k\,\Delta^{(22)}_{jj}
+u_i v_j\,\Delta^{(32)}_{kj}
\Bigr),
\]
where the quantities $\Delta^{(\cdot)}$ are linear forms in resolvent blocks.
Using the resolvent bound
\(
\sup_{z\in\mathbb C^+_{\eta_0}}\|R(z)\|\le \eta_0^{-1}
\)
and $\|u\|_2=\|v\|_2=\|w\|_2=1$, we obtain uniformly in $z\in\mathbb C^+_{\eta_0}$,
\[
\Bigl|\frac{\partial v_j}{\partial W_{ijk}}\Bigr|
\;\lesssim\;
\frac{1}{\sqrt{N}}\eta_0^{-1}
\bigl(|v_j||w_k|+|u_i||w_k|+|u_i||v_j|\bigr),
\]
and an analogous bound for $\partial w_k/\partial W_{ijk}$.
Therefore,
\begin{align*}
I_2
&\lesssim
\frac{1}{\sqrt{N}}\sum_{i,j,k}|x_i|
\frac{1}{N}\eta_0^{-2}
\Bigl(
|v_j|^2|w_k|^2
+|u_i|^2|w_k|^2
+|u_i|^2|v_j|^2
\Bigr).
\end{align*}
Using $\sum_j v_j^2=\sum_k w_k^2=\sum_i u_i^2=1$ and
$\sum_i |x_i|\le \sqrt{N}\|x\|_2=\sqrt{N}$, we conclude that
\[
I_2 \;\lesssim\; \eta_0^{-2} N^{-1/2}=\mathcal O(\frac{1}{\sqrt{N}}).
\]

\medskip
\noindent\textbf{Step 2: bounds for $I_1$ and $I_3$.}
We only treat $I_1$ since $I_3$ is analogous.
Recall
\[
I_1:=\frac{1}{\sqrt N}\sum_{i,j,k}|x_i|
\sup_{z\in\mathbb C^+_{\eta_0}}
\Bigl|w_k\,\frac{\partial^2 v_j}{\partial W_{ijk}^2}\Bigr|.
\]
We start from Eq.~(\ref{eq3.4}), which can be written schematically as
\begin{equation}
\label{eq:dv-struct}
\frac{\partial v_j}{\partial W_{ijk}}
=-\frac{1}{\sqrt N}\Bigl(
v_j w_k\,\Delta^{(12)}_{ij}
+u_i w_k\,\Delta^{(22)}_{jj}
+u_i v_j\,\Delta^{(32)}_{kj}
\Bigr),
\end{equation}
where the quantities $\Delta^{(12)}_{ij},\Delta^{(22)}_{jj},\Delta^{(32)}_{kj}$ are linear forms
in resolvent blocks (for instance $\Delta^{(12)}_{ij}=R^{12}_{ij}-u_i(u^\top \bm R^{12}_{\cdot j})$,
and similarly for the others).

\smallskip
\noindent\emph{Second derivative expansion.}
Differentiating (\ref{eq:dv-struct}) once more w.r.t.\ $W_{ijk}$ gives
\begin{align}
\label{eq:d2v-expand}
\frac{\partial^2 v_j}{\partial W_{ijk}^2}
&=-\frac{1}{\sqrt N}\Bigl[
\frac{\partial (v_j w_k)}{\partial W_{ijk}}\Delta^{(12)}_{ij}
+v_j w_k\,\frac{\partial \Delta^{(12)}_{ij}}{\partial W_{ijk}}
\Bigr] \notag\\
&\quad -\frac{1}{\sqrt N}\Bigl[
\frac{\partial (u_i w_k)}{\partial W_{ijk}}\Delta^{(22)}_{jj}
+u_i w_k\,\frac{\partial \Delta^{(22)}_{jj}}{\partial W_{ijk}}
\Bigr] \notag\\
&\quad -\frac{1}{\sqrt N}\Bigl[
\frac{\partial (u_i v_j)}{\partial W_{ijk}}\Delta^{(32)}_{kj}
+u_i v_j\,\frac{\partial \Delta^{(32)}_{kj}}{\partial W_{ijk}}
\Bigr].
\end{align}

\smallskip
\noindent\emph{Bounds on the coefficients.}
Using $\sup_{z\in\mathbb C^+_{\eta_0}}\|\bm R(z)\|\lesssim 1$ and (\ref{eq:dv-struct}),
we have uniformly for $z\in\mathbb C^+_{\eta_0}$,
\begin{equation}
\label{eq:dv-bound}
\Bigl|\frac{\partial v_j}{\partial W_{ijk}}\Bigr|
\lesssim \frac{1}{\sqrt N}
\bigl(|v_j||w_k|+|u_i||w_k|+|u_i||v_j|\bigr),
\end{equation}
and similarly
\begin{equation}
\label{eq:dw-bound}
\Bigl|\frac{\partial w_k}{\partial W_{ijk}}\Bigr|
\lesssim \frac{1}{\sqrt N}
\bigl(|w_k||v_j|+|u_i||v_j|+|u_i||w_k|\bigr).
\end{equation}
Therefore,
\begin{align}
\label{eq:dcoeff-bound}
\Bigl|\frac{\partial (v_j w_k)}{\partial W_{ijk}}\Bigr|
&\le |w_k|\Bigl|\frac{\partial v_j}{\partial W_{ijk}}\Bigr|
+|v_j|\Bigl|\frac{\partial w_k}{\partial W_{ijk}}\Bigr|
\notag\\
&\lesssim \frac{1}{\sqrt N}
\bigl(|v_j||w_k|+|u_i||w_k|+|u_i||v_j|\bigr),
\end{align}
and analogously,
\begin{align*}
\Bigl|\frac{\partial (u_i w_k)}{\partial W_{ijk}}\Bigr|
& \lesssim \frac{1}{\sqrt N}\bigl(|u_i||w_k|+|v_j||w_k|+|u_i||v_j|\bigr), \\
\Bigl|\frac{\partial (u_i v_j)}{\partial W_{ijk}}\Bigr|
& \lesssim \frac{1}{\sqrt N}\bigl(|u_i||v_j|+|u_i||w_k|+|v_j||w_k|\bigr).
\end{align*}

By the resolvent identity $\partial R=-R(\partial M)R$ (where $M$ is the linearized
block matrix depending on $W$ and $\|\partial M\|=\mathcal{O}(N^{-1/2})$), we have
\begin{equation}
\label{eq:dR-bound}
\sup_{z\in\mathbb C^+_{\eta_0}}\|\partial R(z)\|
\lesssim N^{-1/2}.
\end{equation}
Consequently, since each $\Delta^{(\cdot)}$ is a linear form in resolvent blocks,
\begin{equation}
\label{eq:dWi-bound}
\sup_{z\in\mathbb C^+_{\eta_0}}
\Bigl|\frac{\partial \Delta^{(12)}_{ij}}{\partial W_{ijk}}\Bigr|
+\sup_{z\in\mathbb C^+_{\eta_0}}
\Bigl|\frac{\partial \Delta^{(22)}_{jj}}{\partial W_{ijk}}\Bigr|
+\sup_{z\in\mathbb C^+_{\eta_0}}
\Bigl|\frac{\partial \Delta^{(32)}_{kj}}{\partial W_{ijk}}\Bigr|
\lesssim N^{-1/2}.
\end{equation}

Combining (\ref{eq:d2v-expand})-(\ref{eq:dWi-bound}) and using
$\sup|\Delta^{(\cdot)}|\lesssim \|R\|\lesssim 1$, we obtain
\begin{align}
\label{eq:d2v-bound}
\sup_{z\in\mathbb C^+_{\eta_0}}
\Bigl|\frac{\partial^2 v_j}{\partial W_{ijk}^2}\Bigr|
&\lesssim
\frac{1}{\sqrt N}\Bigl[
\Bigl|\frac{\partial (v_j w_k)}{\partial W_{ijk}}\Bigr|
+|v_j||w_k|\,N^{-1/2}
\Bigr] \notag\\
&\quad +\frac{1}{\sqrt N}\Bigl[
\Bigl|\frac{\partial (u_i w_k)}{\partial W_{ijk}}\Bigr|
+|u_i||w_k|\,N^{-1/2}
\Bigr] \notag\\
&\quad +\frac{1}{\sqrt N}\Bigl[
\Bigl|\frac{\partial (u_i v_j)}{\partial W_{ijk}}\Bigr|
+|u_i||v_j|\,N^{-1/2}
\Bigr] \notag\\
&\lesssim
\frac{1}{N}\bigl(|v_j||w_k|+|u_i||w_k|+|u_i||v_j|\bigr).
\end{align}

Plugging (\ref{eq:d2v-bound}) into the definition of $I_1$ gives
\begin{align}
\label{eq:I1-sum}
I_1
&\lesssim
\frac{1}{\sqrt N}\sum_{i,j,k}|x_i||w_k|
\frac{1}{N}\eta_0^{-2}\bigl(|v_j||w_k|+|u_i||w_k|+|u_i||v_j|\bigr)
\notag\\
&\le \frac{\eta_0^{-2}}{N\sqrt N}\Bigl[
\underbrace{\sum_i|x_i|\sum_j|v_j|\sum_k|w_k|^2}_{(1)}
+\underbrace{\sum_i|x_i||u_i|\sum_j1\sum_k|w_k|^2}_{(2)}
+\underbrace{\sum_i|x_i||u_i|\sum_j|v_j|\sum_k|w_k|}_{(3)}
\Bigr].
\end{align}
Using $\|x\|_2=\|u\|_2=\|v\|_2=\|w\|_2=1$, we have
$\sum_i|x_i|\le \sqrt N$, $\sum_j|v_j|\le \sqrt N$, $\sum_k|w_k|\le \sqrt N$,
and $\sum_k|w_k|^2=1$, as well as $\sum_i|x_i||u_i|\le \|x\|_2\|u\|_2=1$.
Since $n\asymp N$, the three brackets in (\ref{eq:I1-sum}) satisfy
\[
(1)\lesssim \sqrt N\cdot \sqrt N\cdot 1 = N,\qquad
(2)\lesssim 1\cdot N\cdot 1 = N,\qquad
(3)\lesssim 1\cdot \sqrt N\cdot \sqrt N = N.
\]
Therefore,
\begin{equation}
\label{eq:I1-final}
I_1 \lesssim N^{-1/2}=\mathcal{O}(\frac{1}{\sqrt{N}}).
\end{equation}
The same argument yields $I_3=\mathcal{O}(\frac{1}{\sqrt{N}})$, which completes Step~2.

Combining the bounds on $I_1$, $I_2$ and $I_3$, we obtain the remainder term $\frac{1}{\sqrt{N}}\sum_{ijk}x_i\varepsilon_{ijk}^{(2)}$ is negligible.

Therefore, by Eq.~(\ref{eq:42}), the almost sure limits
$\lambda^\infty(\beta)$, $a_x^\infty(\beta)$, $a_y^\infty(\beta)$ and
$a_z^\infty(\beta)$ satisfy
\[
\lambda^\infty(\beta)a_x^\infty(\beta)
= \beta a_y^\infty(\beta)a_z^\infty(\beta)
- \big(g_2(\lambda^\infty(\beta))+g_3(\lambda^\infty(\beta))\big)
a_x^\infty(\beta).
\]

Hence,
\[
a_x^\infty(\beta)
= \alpha_1(\lambda^\infty(\beta))\,a_y^\infty(\beta)a_z^\infty(\beta),
\qquad
\alpha_1(z)=\frac{\beta}{z+g(z)-g_1(z)}.
\]

Similarly, we obtain
\[
\left\{
\begin{aligned}
a_y^\infty(\beta)
&= \alpha_2(\lambda^\infty(\beta))\,a_x^\infty(\beta)a_z^\infty(\beta),\\
a_z^\infty(\beta)
&= \alpha_3(\lambda^\infty(\beta))\,a_x^\infty(\beta)a_y^\infty(\beta),
\end{aligned}
\right.
\qquad
\alpha_i(z)=\frac{\beta}{z+g(z)-g_i(z)}.
\]

Solving the above system yields the final asymptotic alignments.
Proceeding similarly to Eq.~(\ref{eq3.3}), we obtain an estimate of the
asymptotic singular value $\lambda^\infty(\beta)$, thereby completing
the proof.

\subsection{Proof of Corollary~\ref{cor3.2}}\label{appA.6}

By Corollary~\ref{cor3.1}, the limiting Stieltjes transform is given by
\[
g(z)=-\frac{3z}{4}+\frac{\sqrt{3}\sqrt{3z^2-8}}{4},
\]
and since $g(z)=\sum_{i=1}^3 g_i(z)$ with all $g_i(z)$ equal, then
$g_i(z)=\frac{g(z)}{3}$ for all $i\in[3]$. Hence, for each $i\in[3]$,
$\alpha_i(z)$ defined in Theorem~5 writes as
\[
\alpha_i(z)=\frac{\beta}{\frac{z}{2}+\frac{\sqrt{3}\sqrt{3z^2-8}}{6}},
\]
and $\lambda^\infty(\beta)$ is solution to the equation
\[
\frac{z}{4}+\frac{\sqrt{3}\sqrt{3z^2-8}}{4}
-\frac{\left(\frac{z}{2}+\frac{\sqrt{3}\sqrt{3z^2-8}}{6}\right)^3}{\beta^2}
=0.
\]

First we compute the critical value of $\beta$ by solving the above equation in
$\beta$ and taking the limit when $z$ tends to the right edge of the semi-circle
law (i.e., $z\to 2\sqrt{\frac{2}{3}}^{+}$). Indeed, solving the above equation in
$\beta$ yields
\[
\beta(z)=
\sqrt{
\frac{
2z^3+\frac{2z^2\sqrt{9z^2-24}}{3}-4z-\frac{4\sqrt{9z^2-24}}{9}
}{
z+\sqrt{3}\sqrt{3z^2-8}
}
},
\]
hence
\[
\beta_s=\lim_{z\to 2\sqrt{\frac{2}{3}}^{+}}\beta(z)=\frac{2\sqrt{3}}{3}.
\]

Now to express $\lambda^\infty(\beta)$ in terms of $\beta$, we solve the equation
$\beta-\beta(z)=0$ in $z$ and choose the unique non-decreasing (in $\beta$) and
positive solution, which yields
\[
\lambda^\infty(\beta)
=
\sqrt{
\frac{\beta^2}{2}+2+\frac{\sqrt{3}\sqrt{(3\beta^2-4)^3}}{18\beta}
}.
\]

Plugging the above expression of $\lambda^\infty(\beta)$ into the expressions of
the asymptotic alignments in Theorem~\ref{thm3.4}, we obtain for all $i\in[3]$
\[
\alpha_i\bigl(\lambda^\infty(\beta)\bigr)=
\frac{6\sqrt{2}\,\beta}{
\sqrt{\,9\beta^2-12+\frac{\sqrt{3}\sqrt{(3\beta^2-4)^3}}{\beta}}
+
\sqrt{\,9\beta^2+36+\frac{\sqrt{3}\sqrt{(3\beta^2-4)^3}}{\beta}}
},
\]
yielding the final result.

\section{Proof of Theorem \ref{thm4.1} and Corollary \ref{cor4.1}}\label{app:B}

\subsection{Proof of Theorem \ref{thm4.1}}\label{appA.7}

Denote the matrix model as
\[
\bm N \equiv \frac{1}{\sqrt N}\,
\Phi_d\bigl(\textbf{W},\bm a^{(1)},\ldots,\bm a^{(d)}\bigr),
\]
where $\textbf{W}\sim\textbf{T}_{n_1,\ldots,n_d}(\bm W)$ with \(\bm W\) satisfies (\ref{eq:moment-assumption}) and
$(\bm a^{(1)},\ldots,\bm a^{(d)})\in
\mathbb S^{n_1-1}\times\cdots\times\mathbb S^{n_d-1}$ are independent of
$\textbf{W}$.
We further denote the resolvent of $\bm N$ as
\[
\bm Q(z)\equiv(\bm N-z\bm I_N)^{-1}
=
\begin{bmatrix}
\bm Q^{11}(z) & \bm Q^{12}(z) & \bm Q^{13}(z) & \cdots & \bm Q^{1d}(z)\\
\bm Q^{21}(z) & \bm Q^{22}(z) & \bm Q^{23}(z) & \cdots & \bm Q^{2d}(z)\\
\bm Q^{31}(z) & \bm Q^{32}(z) & \bm Q^{33}(z) & \cdots & \bm Q^{3d}(z)\\
\vdots & \vdots & \vdots & \ddots & \vdots\\
\bm Q^{d1}(z) & \bm Q^{d2}(z) & \bm Q^{d3}(z) & \cdots & \bm Q^{dd}(z)
\end{bmatrix}.
\]

By the Borel--Cantelli lemma, we have
\[
\frac{1}{N}\operatorname{tr}\bm Q(z)
\xrightarrow{\mathrm{a.s.}} g(z),
\qquad
\text{and for all } i\in[d],\quad
\frac{1}{N}\operatorname{tr}Q^{ii}(z)
\xrightarrow{\mathrm{a.s.}} g_i(z).
\]

Recall that \(\bm N\bm Q(z) - z\bm Q(z)=\bm I_N\), from which we particularly obtain
\[
\frac{1}{\sqrt N}\sum_{j=2}^d
\bigl[\textbf{W}^{1j}\bm Q^{1j}(z)^\top\bigr]_{i_1 i_1}
- z Q^{11}_{i_1 i_1}(z)
= 1,
\]
or equivalently,
\begin{equation}
\label{eq:43}
\frac{1}{N\sqrt N}
\sum_{j=2}^d\sum_{i_j=1}^{n_j}
\bigl[\textbf{W}^{1j}\bm Q^{1j}(z)^\top\bigr]_{i_1 i_1}
-\frac{z}{N}\operatorname{tr}\bm Q^{11}(z)
= \frac{n_1}{N},
\end{equation}
where we recall that
\[
\textbf{W}^{ij}
=
\textbf{W}\bigl(\bm a^{(1)},\ldots,\bm a^{(i-1)},\cdot,
\bm a^{(i+1)},\ldots,\bm a^{(j-1)},\cdot,
\bm a^{(j+1)},\ldots,\bm a^{(d)}\bigr)
\in\mathbb M_{n_i,n_j}.
\]

We thus need to compute the expectation of
\[
\frac{1}{N\sqrt N}\sum_{i_1=1}^{n_1}
\bigl[\textbf{W}^{1j}\bm Q^{1j}(z)^\top\bigr]_{i_1 i_1},
\]
which develops as
\begin{align*}
A_j
&\equiv
\frac{1}{N\sqrt N}\sum_{i_1=1}^{n_1}
\mathbb E\!\left[
\bigl[\textbf{W}^{1j}\bm Q^{1j}(z)^\top\bigr]_{i_1 i_1}
\right]
=
\frac{1}{N\sqrt N}\sum_{i_1,i_j}
\mathbb E\!\left[
W^{1j}_{i_1 i_j} Q^{1j}_{i_1 i_j}
\right] \\
&=
\frac{1}{N\sqrt N}
\sum_{i_1,\ldots,i_d}
\prod_{k\neq 1,j} a^{(k)}_{i_k}\,
\mathbb E\!\left[
W_{i_1,\ldots,i_d}\, Q^{1j}_{i_1 i_j}
\right] \\
&=
\frac{1}{N\sqrt N}
\sum_{i_1,\ldots,i_d}
\prod_{k\neq 1,j} a^{(k)}_{i_k}\,
(\mathbb E\!\left[
\frac{\partial Q^{1j}_{i_1 i_j}}{\partial W_{i_1,\ldots,i_d}}
\right]
+\varepsilon_{i_1\cdots i_d}^{(2)}),
\end{align*}
where the last equality follows from the cumulant expansion \ref{lem2.3},
and $\varepsilon_{i_1\cdots i_d}^{(2)}$ collects the higher-order cumulant remainders.

In particular, as in Appendix~\ref{appA.2} for the $3$-order case, it turns out that
the only contributing term in the derivative
$\partial Q^{1j}_{i_1 i_j}/\partial W_{i_1,\ldots,i_d}$
is
\[
-\frac{1}{\sqrt N}
\prod_{k\neq 1,j} a^{(k)}_{i_k}\,
Q^{11}_{i_1 i_1}\,Q^{jj}_{i_j i_j},
\]
with the other terms yielding quantities of order $\mathcal{O}(N^{-1})$.

Since \(\varepsilon_{i_1 \cdots i_d}^{(2)} \le \sup_{z\in \mathbb C^+_{\eta_0}}|\frac{\partial^2}{\partial W_{i_i \cdots i_d}^2}Q_{i_1i_j}^{1j}|\) , we need to prove 
\[
\frac{1}{N\sqrt N}
\sum_{i_1,\ldots,i_d}
\prod_{k\neq 1,j} a^{(k)}_{i_k}\sup_{z\in \mathbb C^+_{\eta_0}}|\frac{\partial^2}{\partial W_{i_i \cdots i_d}^2}Q_{i_1i_j}^{1j}| = \mathcal{O}(N^{-1/2}).
\]
Recall that $\bm Q(z)=(\bm N-z\bm I_N)^{-1}$ and denote
\[
\bm A_{i_1,\ldots,i_d}
:=
\frac{\partial \bm N}{\partial W_{i_1,\ldots,i_d}} .
\]
Since $\bm N$ depends linearly on each entry $W_{i_1,\ldots,i_d}$, we have
\[
\frac{\partial^2 \bm N}{\partial W_{i_1,\ldots,i_d}^2}=0 .
\]
By the resolvent identity,
\[
\partial_{W_{i_1,\ldots,i_d}} \bm Q(z)
=
- \bm Q(z)\, \bm A_{i_1,\ldots,i_d}\, \bm Q(z),
\]
and differentiating once more yields
\[
\partial_{W_{i_1,\ldots,i_d}}^2 \bm Q(z)
=
2\,\bm Q(z)\, \bm A_{i_1,\ldots,i_d}\, \bm Q(z)\,
\bm A_{i_1,\ldots,i_d}\, \bm Q(z).
\]
Hence, for any $z\in\mathbb C^+_{\eta_0}$ and any block indices $(a,b)$,
\[
\begin{aligned}
\left|
\partial_{W_{i_1,\ldots,i_d}}^2 Q^{ab}_{i_a i_b}(z)
\right|
&\le
2\,
\bigl\|
\bm Q(z)\, A_{i_1,\ldots,i_d}\, \bm Q(z)\,
\bm A_{i_1,\ldots,i_d}\, \bm Q(z)
\bigr\|_2 \\
&\le
2\,\|\bm Q(z)\|_2^3\,\|A_{i_1,\ldots,i_d}\|_2^2 .
\end{aligned}
\]
Since $\|\bm Q(z)\|_2\le\eta_0^{-1}$ for all $z\in\mathbb C^+_{\eta_0}$, it follows that
\[
\sup_{z\in\mathbb C^+_{\eta_0}}
\left|
\partial_{W_{i_1,\ldots,i_d}}^2 Q^{ab}_{i_a i_b}(z)
\right|
\le
2\,\eta_0^{-3}\,\|\bm A_{i_1,\ldots,i_d}\|_2^2 .
\]

Moreover, by the structure of $\Phi_d(\textbf{W},a^{(1)},\ldots,a^{(d)})$
and the normalization of the noise,
\[
\|\bm A_{i_1,\ldots,i_d}\|_2
=
\left\|
\frac{\partial \bm N}{\partial W_{i_1,\ldots,i_d}}
\right\|_2
\le
\frac{C}{\sqrt N}
    \sum_{a\neq b} \prod_{\ell \neq a,b}|a_{i_\ell}^{(\ell)}|,
\]
so that
\[
\|\bm A_{i_1,\ldots,i_d}\|_2^2
\le
\frac{C}{N}
\sum_{a\neq b} \prod_{\ell \neq a,b}|a_{i_\ell}^{(\ell)}|^2 .
\]

Summing over all multi-indices $(i_1,\ldots,i_d)$ and using
$\|\bm a^{(\ell)}\|_2=1$ for all $\ell\in[d]$, we obtain
\[
\begin{aligned}
&\frac{1}{N\sqrt N}
\sum_{i_1,\ldots,i_d}
\prod_{k\neq 1,j} a^{(k)}_{i_k}\sup_{z\in \mathbb C^+_{\eta_0}}|\frac{\partial^2}{\partial W_{i_i \cdots i_d}^2}Q_{i_1i_j}^{1j}|\\
\le &\frac{1}{N\sqrt N}
\sum_{i_1,\ldots,i_d}
\prod_{k\neq 1,j} a^{(k)}_{i_k}\cdot2\eta_0^{-3}\frac{C}{N}\sum_{a\neq b} \prod_{\ell \neq a,b}|a_{i_\ell}^{(\ell)}|^2\\
\le & C^*N^{-5/2}\sum_{i_1,\ldots,i_d}
\prod_{k\neq 1,j} a^{(k)}_{i_k}\sup_{a,b}\prod_{\ell \neq a,b}|a_{i_\ell}^{(\ell)}|^2
\end{aligned}
\]
We bound the term by distinguishing the relative position of $(a,b)$ with
respect to $(1,j)$. Writing
\[
S_{ab}
=
\sum_{i_1,\ldots,i_d}
\Bigl(\prod_{k\neq 1,j}|a^{(k)}_{i_k}|\Bigr)
\prod_{\ell\neq a,b}|a^{(\ell)}_{i_\ell}|^2
=
\prod_{\ell=1}^d
\Bigl(\sum_{i_\ell}|a^{(\ell)}_{i_\ell}|^{e_\ell}\Bigr),
\qquad
e_\ell=\bm 1_{\ell\neq 1,j}+2\cdot\bm 1_{\ell\neq a,b},
\]
we consider the following cases.

\emph{Case 1:} $(a,b)=(1,j)$. Then $e_1=e_j=0$ and $e_\ell=3$ for $\ell\neq 1,j$,
so that
\[
S_{1j}\le n_1n_j\le C N^2.
\]

\emph{Case 2:} $\{a,b\}\cap\{1,j\}$ contains exactly one element. Without loss of
generality $a=1$, $b\neq j$. Then
\[
S_{1b}\le n_1\sqrt{n_b}\le C N^{3/2}.
\]

\emph{Case 3:} $\{a,b\}\cap\{1,j\}=\varnothing$. Then
\[
S_{ab}\le \sqrt{n_a n_b}\le C N.
\]

In all cases, $S_{ab}\le C N^2$, hence
\[
N^{-5/2}\sum_{a\neq b}S_{ab}\le C N^{-1/2}\to 0.
\]

Therefore, we find that
\begin{align*}
A_j
&=
-\frac{1}{N\sqrt N}
\sum_{i_1,\ldots,i_d}
\prod_{k\neq 1,j} a^{(k)}_{i_k}
(\mathbb E\!\left[
\frac{\partial Q^{1j}_{i_1 i_j}}{\partial W_{i_1,\ldots,i_d}}
\right] + \varepsilon_{i_1 \cdots i_d}^{(2)} )\\
&=
-\frac{1}{N^2}
\sum_{i_1,i_j}
\mathbb E\!\left[
Q^{11}_{i_1 i_1}\,Q^{jj}_{i_j i_j}
\right]
+\mathcal{O}(N^{-1/2}) \\
&=
-\mathbb{E}\bigl[\frac{1}{N}\operatorname{tr}\bm Q^{11}(z)\,
\frac{1}{N}\operatorname{tr}\bm Q^{jj}(z)\bigr]
+\mathcal{O}(N^{-1/2})
\xrightarrow{\mathrm{a.s.}}
- g_1(z)g_j(z)+\mathcal{O}(N^{-1/2}).
\end{align*}

From Eq.~(\ref{eq:43}), $g_1(z)$ satisfies
\[
- g_1(z)\sum_{j\neq 1} g_j(z) - z g_1(z) = c_1,
\qquad
c_1=\lim\frac{n_1}{N}.
\]
Similarly, for all $i\in[d]$, $g_i(z)$ satisfies
\[
- g_i(z)\sum_{j\neq i} g_j(z) - z g_i(z) = c_i,
\qquad
c_i=\lim\frac{n_i}{N}.
\]
Since $g(z)=\sum_{i=1}^d g_i(z)$, we have for each $i\in[d]$,
\[
g_i^2(z)-(g(z)+z)g_i(z)+c_i=0,
\]
yielding
\[
g_i(z)
=
\frac{g(z)+z}{2}
-\frac{\sqrt{\,4c_i+(g(z)+z)^2\,}}{2},
\]
with $g(z)$ solution to the equation
\[
g(z)=\sum_{i=1}^d g_i(z)
\quad\text{such that}\quad
\Im g(z)>0\ \text{for }\Im z>0.
\]

\subsection{Proof of Corollary \ref{cor4.1}}\label{appA.8}

Given Theorem~\ref{thm4.1} and setting $c_i=\frac{1}{d}$ for all $i\in[d]$, we have
\[
g_i(z)=\frac{g(z)}{d}, \qquad \forall i\in[d].
\]
Therefore, the Stieltjes transform $g(z)$ satisfies
\begin{equation}
\label{eq:B10_self_consistent}
\frac{g(z)}{d}
=
\frac{g(z)+z}{2}
-
\frac{1}{2}
\sqrt{\frac{4}{d}+\bigl(g(z)+z\bigr)^2}.
\end{equation}

Solving Eqn.~(\ref{eq:B10_self_consistent}) for $g(z)$ yields the two branches
\[
g(z)\in
\left\{
-\frac{dz}{2(d-1)}-\frac{\sqrt{d\,(dz^2-4d+4)}}{2(d-1)},
\;
-\frac{dz}{2(d-1)}+\frac{\sqrt{d\,(dz^2-4d+4)}}{2(d-1)}
\right\}.
\]

Imposing the condition $\Im g(z)\ge 0$ for $\Im z>0$, we select the physical
branch of the Stieltjes transform, namely
\[
g(z)
=
-\frac{dz}{2(d-1)}
+
\frac{\sqrt{d\,(dz^2-4d+4)}}{2(d-1)}.
\]

This concludes the proof.
\hfill $\square$

\section{Proof of Theorem \ref{thm4.2}}\label{appA.9}

Given the random tensor model in Eq.~(\ref{eq1.2}) and its singular vectors characterized by
Eq.~(\ref{eq1.6}), we denote the associated random matrix model as
\[
\textbf{T} \equiv \Phi_d\bigl(\textbf{T}, \bm u^{(1)},\ldots,\bm u^{(d)}\bigr)
= \beta \bm V \bm B \bm V^\top + \bm N,
\]
where
\[
\bm N=\frac{1}{\sqrt N}\Phi_d\bigl(\textbf{W},\bm u^{(1)},\ldots,\bm u^{(d)}\bigr),
\quad
\bm B\in\mathbb M_d \; \text{ with entries }\; 
B_{ij}=(1-\delta_{ij})\prod_{k\neq i,j}\langle \bm u^{(k)},\bm x^{(k)}\rangle,
\]
and
\[
\bm V=
\begin{pmatrix}
\bm x^{(1)} & 0_{n_1} & \cdots & 0_{n_1}\\
0_{n_2} & \bm x^{(2)} & \cdots & 0_{n_2}\\
\vdots & \vdots & \ddots & \vdots\\
0_{n_d} & 0_{n_d} & \cdots & \bm x^{(d)}
\end{pmatrix}
\in \mathbb M_{N,d}.
\]

\medskip

We further denote the resolvent of $\textbf{T}$ and $\bm N$ respectively as
\[
\bm R(z)=(\textbf{T}-z\bm I_N)^{-1}
=
\begin{pmatrix}
\bm R^{11}(z) & \bm R^{12}(z) & \bm R^{13}(z) & \cdots & \bm R^{1d}(z)\\
\bm R^{21}(z) & \bm R^{22}(z) & \bm R^{23}(z) & \cdots & \bm R^{2d}(z)\\
\bm R^{31}(z) & \bm R^{32}(z) & \bm R^{33}(z) & \cdots & \bm R^{3d}(z)\\
\vdots & \vdots & \vdots & \ddots & \vdots\\
\bm R^{d1}(z) & \bm R^{d2}(z) & \bm R^{d3}(z) & \cdots & \bm R^{dd}(z)
\end{pmatrix},
\]
\[
\bm Q(z)=(\bm N-z\bm I_N)^{-1}
=
\begin{pmatrix}
\bm Q^{11}(z) & \bm Q^{12}(z) & \bm Q^{13}(z) & \cdots & \bm Q^{1d}(z)\\
\bm Q^{21}(z) & \bm Q^{22}(z) & \bm Q^{23}(z) & \cdots & \bm Q^{2d}(z)\\
\bm Q^{31}(z) & \bm Q^{32}(z) & \bm Q^{33}(z) & \cdots & \bm Q^{3d}(z)\\
\vdots & \vdots & \vdots & \ddots & \vdots\\
\bm Q^{d1}(z) & \bm Q^{d2}(z) & \bm Q^{d3}(z) & \cdots & \bm Q^{dd}(z)
\end{pmatrix}.
\]

Similarly as in the $3$-order case, by Woodbury matrix identity (Lemma~1) we have
\[
\frac{1}{N}\operatorname{tr}\bm R(z)
=
\frac{1}{N}\operatorname{tr}\bm Q(z)
-
\frac{1}{N}\operatorname{tr}
\Bigl[
\bigl(\tfrac{1}{\beta}\bm B^{-1}+\bm V^\top \bm Q(z)\bm V\bigr)^{-1}
\bm V^\top \bm Q^2(z)\bm V
\Bigr],
\]
since the perturbation matrix
\(
\bigl(\tfrac{1}{\beta}\bm B^{-1}+\bm V^\top \bm Q(z)\bm V\bigr)^{-1}\bm V^\top \bm Q^2(z)\bm V
\)
is of bounded spectral norm if $\|\bm Q(z)\|$ is bounded.
Hence for the characterization of the spectrum of $\textbf{T}$ boils down to the estimation of
$\frac{1}{N}\operatorname{tr}\bm Q(z)$.

Now we are left to handle the statistical dependency between the tensor noise
$\textbf{W}$ and the singular vectors $\bm u_\ast$.
Recalling the proof of Appendix~\ref{appA.7}, we have
\[
\frac{1}{N\sqrt N}\sum_{i_1=1}^{n_1}
\bigl[\textbf{W}^{1j}\bm Q^{1j}(z)^\top\bigr]_{i_1 i_1}
-\frac{z}{N}\operatorname{tr}\bm Q^{11}(z)
=\frac{n_1}{N},
\]
with
\(
\bm{W}^{ij}= \bm{W}(u^{(1)},\ldots,u^{(i-1)},\cdot,u^{(i+1)},\ldots,u^{(j-1)},\cdot,u^{(j+1)},\ldots,u^{(d)})
\in\mathbb M_{n_i,n_j}.
\)

Taking the expectation of
\(
\frac{1}{N\sqrt N}\sum_{i_1=1}^{n_1}
\bigl[\textbf{W}^{1j}\bm Q^{1j}(z)^\top\bigr]_{i_1 i_1}
\)
yields
\begin{align*}
A_j =& \frac{1}{N\sqrt N}\sum_{i_1=1}^{n_1}\mathbb{E}
\bigl[\textbf{W}^{1j}\bm Q^{1j}(z)^\top\bigr]_{i_1 i_1}\\
=&
\frac{1}{N\sqrt N}
\sum_{i_1,\ldots,i_d}
\mathbb E\!\left[
\frac{\partial Q^{1j}_{i_1 i_j}}{\partial W_{i_1,\ldots,i_d}}
\prod_{k\neq1,j}u^{(k)}_{i_k}\right]\\
+&
\frac{1}{N\sqrt N}
\sum_{i_1,\ldots,i_d}
\mathbb E\!\left[
Q^{1j}_{i_1 i_j}
\sum_{\ell \neq 1,j}\frac{\partial u^{(\ell)}_{i_\ell}}{\partial W_{i_1,\ldots,i_d}}\prod_{k\neq1,j,\ell}u_{i_k}^{(k)}
\right] + \frac{1}{N\sqrt{N}}\sum_{i_1,\ldots,i_d}\varepsilon_{i_1 \cdots i_d}^{(2)}\\
=& A_{j1}+A_{j2}+\frac{1}{N\sqrt{N}}\sum_{i_1,\ldots,i_d}\varepsilon_{i_1 \cdots i_d}^{(2)},
\end{align*}
where the second equation used Lemma \ref{lem2.3}.

We already computed the first term $A_{j1}$ in Appendix~\ref{appA.8}.
Now we show that $A_{j2}$ is asymptotically vanishing under Assumption~\ref{assumption4.1}.
Indeed, by Eq.~(\ref{eq4.3}), the higher order terms arise from the term
\(
-\frac{1}{\sqrt{N}}\prod_{\ell \neq k}u^{(\ell)}_{i_\ell}R^{kk}_{i_k i_k}(\lambda)
\),
thus we only show that the contribution of this term is also vanishing.
Precisely,
\begin{align*}
A_{j2}
&=
-\frac{1}{N^2}
\sum_{i_1,\ldots,i_d}
\mathbb E\!\left[
Q^{1j}_{i_1 i_j}
\sum_{\substack{\ell\neq 1\\ \ell\neq j}}
\prod_{t\neq \ell}
u^{(t)}_{i_t}\,
R^{\ell\ell}_{i_\ell i_\ell}(\lambda)\prod_{k\neq1,j,\ell}u_{i_k}^{(k)}
\right]
+ \mathcal O(N^{-1}) \\[1ex]
&=
-\frac{1}{N^2}
\sum_{i_1,\ldots,i_d}
\mathbb E\!\left[u_{i_1}^{(1)}
Q^{1j}_{i_1 i_j}u_{i_j}^{(j)}
\sum_{\substack{\ell\neq 1\\ \ell\neq j}}
R^{\ell\ell}_{i_\ell i_\ell}(\lambda)\prod_{k\neq1,j,\ell}(u_{i_k}^{(k)})^2
\right]
+ \mathcal O(N^{-1}) \\[1ex]
\end{align*}
Interchanging the order of summation and using $\|u^{(k)}\|_2=1$ for all
$k\in[d]$, we obtain
\begin{align*}
A_{j2}
&=
-\frac{1}{N^2}
\mathbb E\!\left[
(\bm u^{(1)})^\top \bm Q^{1j}(z)\,\bm u^{(j)}
\sum_{\substack{\ell\neq 1\\ \ell\neq j}}
\operatorname{tr} \bm R^{\ell\ell}(\lambda)
\right]
+\mathcal O(N^{-1}) .
\end{align*}

Since $\|\bm{Q}^{1j}(z)\|_2\le \eta_0^{-1}$ and
$\sum_{\ell\neq1,j}\operatorname{tr} \bm{R}^{\ell\ell}(\lambda)=\mathcal O(N)$,
it follows that
\[
A_{j2}=\mathcal O(N^{-1}),
\]
and therefore $A_{j2}$ vanishes asymptotically.

Now we need to prove \(\frac{1}{N\sqrt{N}}\sum_{i_1,\ldots,i_d}\varepsilon_{i_1 \cdots i_d}^{(2)}\) is vanishing, it is sufficient to prove 
\[
\frac{1}{N\sqrt{N}}\sum_{i_1,\ldots,i_d}\sup_{z \in \mathbb C^+_{\eta_0}}|\frac{\partial^2}{\partial W_{i_1 \cdots i_d}^2}(Q_{i_1i_j}^{1j}\prod_{k \neq 1,j}u_{i_k}^{(k)})| = \mathcal O(N^{-1/2}).
\]

\subsection{Notation and goal}
Fix a multi-index $(i_1,\ldots,i_d)$ and define
\[
D := \frac{\partial}{\partial W_{i_1\cdots i_d}},
\qquad
U := \prod_{k\neq 1,j} u^{(k)}_{i_k},
\qquad
F(z) := Q^{1j}_{i_1 i_j}(z)\, U ,
\]
where $\bm Q(z)=(\bm N-z\bm I)^{-1}$ and $z\in\mathbb C^+_{\eta_0}$.

Our goal is to show that
\begin{equation}\label{eq:Goal}
\frac{1}{N\sqrt N}
\sum_{i_1,\ldots,i_d}
\sup_{z\in\mathbb C^+_{\eta_0}}
\bigl| D^2 F(z) \bigr|
= \mathcal O\!\left(N^{-1/2}\right).
\end{equation}
Equivalently, it suffices to prove the following.
\begin{equation}\label{eq:Goal-equivalent}
\sum_{i_1,\ldots,i_d}
\sup_{z\in\mathbb C^+_{\eta_0}}
\bigl| D^2 F(z) \bigr|
= \mathcal O(N).
\end{equation}
The resolvent identities give
\begin{equation}\label{eq:DQ}
D\bm Q=-\bm Q(D\bm N)\bm Q,\qquad
D^2\bm Q=2\bm Q(D\bm N)\bm Q(D\bm N)\bm Q-\bm Q(D^2\bm N)\bm Q .
\end{equation}
For the $d$-dimensional contraction matrix
\(
\bm N=\frac1{\sqrt N}\Phi_d(\textbf{W},\bm u^{(1)},\ldots,\bm u^{(d)}),
\)
its off-diagonal blocks satisfy, for any $a\neq b$,
\[
\bm N^{ab}_{i_ai_b}
=\frac1{\sqrt N}\sum_{i_k:\,k\neq a,b}W_{i_1\cdots i_d}\prod_{k\neq a,b}u^{(k)}_{i_k},
\]
hence the main part of first derivative with respect to $W_{i_1\cdots i_d}$ is (like the 3-order) 
\begin{equation}\label{eq:DN}
D\bm N
=\frac{1}{\sqrt N}\sum_{a\neq b}
\Bigl(\prod_{k\neq a,b}u^{(k)}_{i_k}\Bigr)\,\bm E^{ab}_{i_a i_b},
\end{equation}
where $\bm E^{ab}_{i_a i_b}$ denotes the elementary matrix supported on the
$(a,b)$ block at position $(i_a,i_b)$.
Moreover, on $\mathbb C^+_{\eta_0}$ we have the uniform bound
\begin{equation}\label{eq:Qop}
\|\bm Q(z)\|\le \eta_0^{-1}.
\end{equation}
We expand $D^2F$ by the product rule. Recall
\[
F(z)=Q^{1j}_{i_1 i_j}(z)\,U,
\qquad
U=\prod_{k\neq 1,j}u^{(k)}_{i_k},
\qquad
D=\frac{\partial}{\partial W_{i_1\cdots i_d}}.
\]
Then
\begin{equation}\label{eq:D2F-expand}
D^2F(z)
=
\bigl(D^2 Q^{1j}_{i_1 i_j}(z)\bigr)\,U
\;+\;
2\bigl(D Q^{1j}_{i_1 i_j}(z)\bigr)\,(DU)
\;+\;
Q^{1j}_{i_1 i_j}(z)\,(D^2U).
\end{equation}
Then we divide the proof into three parts:
\begin{align*}
    \sum_{i_1,\ldots,i_d}
\sup_{z\in\mathbb C^+_{\eta_0}}
\bigl| D^2 F(z) \bigr| =& \sum_{i_1,\ldots,i_d}
\sup_{z\in\mathbb C^+_{\eta_0}}
|\bigl(D^2 Q^{1j}_{i_1 i_j}(z)\bigr)U| + \sum_{i_1,\ldots,i_d}
\sup_{z\in\mathbb C^+_{\eta_0}}
|2\bigl(D Q^{1j}_{i_1 i_j}(z)\bigr)\,(DU)
|\\
+&\sum_{i_1,\ldots,i_d}
\sup_{z\in\mathbb C^+_{\eta_0}}
|Q^{1j}_{i_1 i_j}(z)\,(D^2U)|\\
=& S_1 + S_2 + S_3,
\end{align*}
we prove \(S_!,S_2,S_3 = \mathcal{O}(N)\).

\subsection{Proof of \(S_1\)}\label{secS1}

Recall
\[
S_1
:=\sum_{i_1,\ldots,i_d}
\sup_{z\in\mathbb C^+_{\eta_0}}
\bigl|
(D^2 Q^{1j}_{i_1 i_j})(z)
\bigr|
\prod_{k\neq 1,j}\bigl|u^{(k)}_{i_k}\bigr|.
\]
Using (\ref{eq:DQ}) we split
\[
(D^2 Q^{1j}_{i_1 i_j})
=
\bigl(2\bm Q(D\bm N)\bm Q(D\bm N)\bm Q\bigr)^{1j}_{i_1 i_j}
-
\bigl(\bm Q(D^2\bm N)\bm Q\bigr)^{1j}_{i_1 i_j}.
\]
Accordingly, we have 
\[
S_1\le S_{1,\mathrm{main}}+S_{1,\mathrm{rem}},
\]
where
\[
S_{1,\mathrm{main}}
:=\sum_{i_1,\ldots,i_d}\sup_{z\in\mathbb C^+_{\eta_0}}
\bigl|
\bigl(2\bm Q(D\bm N)\bm Q(D\bm N)\bm Q\bigr)^{1j}_{i_1 i_j}
\bigr|
\prod_{k\neq 1,j}|u^{(k)}_{i_k}|,
\]
and
\[
S_{1,\mathrm{rem}}
:=\sum_{i_1,\ldots,i_d}\sup_{z\in\mathbb C^+_{\eta_0}}
\bigl|
\bigl(\bm Q(D^2\bm N)\bm Q\bigr)^{1j}_{i_1 i_j}
\bigr|
\prod_{k\neq 1,j}|u^{(k)}_{i_k}|.
\]

\subsubsection*{Reduction to the main term}

Substituting (\ref{eq:DN}) into $Q(D\bm N)Q(D\bm N)Q$ yields
\[
\bm Q(D\bm N)\bm Q(D\bm N)\bm Q
=\frac{1}{N}
\sum_{\substack{a\neq b\\ c\neq e}}
\Bigl(\prod_{k\neq a,b}u^{(k)}_{i_k}\Bigr)
\Bigl(\prod_{k\neq c,e}u^{(k)}_{i_k}\Bigr)\,
\bm Q E^{ab}_{i_a i_b} \bm Q E^{ce}_{i_c i_e} \bm Q .
\]
Therefore, for $z\in\mathbb C^+_{\eta_0}$,
\begin{align}
\sup_z
\bigl|
\bigl(2\bm Q(D\bm N)\bm Q(D\bm N)\bm Q\bigr)^{1j}_{i_1 i_j}
\bigr|
&\lesssim
\frac{1}{N}\sum_{\substack{a\neq b\\ c\neq e}}
\Bigl(\prod_{k\neq a,b}|u^{(k)}_{i_k}|\Bigr)
\Bigl(\prod_{k\neq c,e}|u^{(k)}_{i_k}|\Bigr)
\|\bm Q\|^3.
\label{eq:D2Qmain-bound}
\end{align}
Plugging (\ref{eq:D2Qmain-bound}) into the definition of $S_{1,\mathrm{main}}$ gives
\begin{align}
S_{1,\mathrm{main}}
&\lesssim
\frac{1}{N}
\sum_{\substack{a\neq b\\ c\neq e}}
\sum_{i_1,\ldots,i_d}
\prod_{k\neq 1,j}|u^{(k)}_{i_k}|
\prod_{k\neq a,b}|u^{(k)}_{i_k}|
\prod_{k\neq c,e}|u^{(k)}_{i_k}|.
\label{eq:S1main-expand}
\end{align}

For each coordinate $r\in[d]$, let $\alpha_r\in\{0,1,2,3\}$ denote the total power
of $|u^{(r)}_{i_r}|$ appearing in the integrand.
Since each product deletes exactly two indices, we always have $\alpha_r\le 3$.
The sum over indices factorizes as
\[
\sum_{i_1,\ldots,i_d}\prod_{r=1}^d |u^{(r)}_{i_r}|^{\alpha_r}
=\prod_{r=1}^d \sum_{i_r}|u^{(r)}_{i_r}|^{\alpha_r}.
\]

Note that only the cases $\alpha_r=0$ and $\alpha_r=1$ contribute powers of $N$,
while $\alpha_r\ge 2$ yields a bounded contribution:
\[
\sum_{i_r}1 \sim N \quad (\alpha_r=0),\qquad
\sum_{i_r}|u^{(r)}_{i_r}|\lesssim \sqrt N \quad (\alpha_r=1),
\]
and
\[
\sum_{i_r}|u^{(r)}_{i_r}|^2=1,\qquad
\sum_{i_r}|u^{(r)}_{i_r}|^3=\|u^{(r)}\|_3^3\le \|u^{(r)}\|_2^3=1.
\]

Consequently, the order of the inner sum is completely determined by how the
three deleted index pairs $\{1,j\}$, $\{a,b\}$ and $\{c,e\}$ overlap.
A direct inspection shows:
\begin{itemize}
\item The maximal contribution arises when all three products delete the same
pair $\{1,j\}$, i.e.\ $\{a,b\}=\{1,j\}$ and $\{c,e\}=\{1,j\}$.
In this case $\alpha_1=\alpha_j=0$ and $\alpha_r=3$ for all $r\neq 1,j$,
so the inner sum is of order $N^2$, which yields $\mathcal O(N)$ after
multiplication by the prefactor $1/N$.
\item If the deleted pairs are not all identical but share exactly one index
(e.g.\ configurations of the form $(1,j)$, $(j,i)$, $(i,1)$),
then one obtains at most one factor $N$ and two factors $\sqrt N$,
leading to an inner sum of order $N^{3/2}$ and hence a contribution
$\mathcal O(\sqrt N)$ after the prefactor $1/N$.
\item All remaining configurations produce smaller orders.
\end{itemize}

Since the number of pairs $(a,b),(c,e)$ is finite (with $d$ fixed), we conclude
that
\[
S_{1,\mathrm{main}}=\mathcal O(N),
\]
and in fact most terms are of lower order $\mathcal O(\sqrt N)$ or smaller.

\subsubsection*{Reduction to the remain term}

We estimate the remainder part
\[
S_{1,\mathrm{rem}}
:=\sum_{i_1,\ldots,i_d}\sup_{z\in\mathbb C^+_{\eta_0}}
\Bigl|\bigl(\bm Q(D^2\bm N)\bm Q\bigr)^{1j}_{i_1i_j}(z)\Bigr|
\prod_{k\neq 1,j}|u^{(k)}_{i_k}|.
\]
Using $\|\bm{Q}(z)\|\le \eta_0^{-1}$ on $\mathbb C^+_{\eta_0}$ and
$|(\bm Q\bm A\bm Q)^{1j}_{i_1i_j}|\le \|\bm Q\|^2\|\bm A\|$, we get
\begin{equation}\label{eq:S1rem-1}
S_{1,\mathrm{rem}}
\;\lesssim\;
\sum_{i_1,\ldots,i_d}\|D^2\bm N\|\prod_{k\neq 1,j}|u^{(k)}_{i_k}|.
\end{equation}

Fix $(i_1,\ldots,i_d)$ and write $D:=\partial/\partial W_{i_1,\cdots, i_d}$.
Since $\bm N$ is linear in $W$, the second derivative $D^2\bm N$ comes only from the
dependence of $u^{(k)}$ on $W$.
For any $a\neq b$,
\[
(D^2\bm N)^{ab}_{i_ai_b}
=
\frac{1}{\sqrt N}\,
D^2\!\Bigl(\prod_{k\neq a,b}u^{(k)}_{i_k}\Bigr),
\qquad (D^2\bm N)^{aa}\equiv 0,
\]
and by the product rule,
\begin{align}\label{eq:D2prod}
D^2\!\Bigl(\prod_{k\neq a,b}u^{(k)}_{i_k}\Bigr)
&=
\sum_{\ell\neq a,b}\Bigl(D^2u^{(\ell)}_{i_\ell}\prod_{k\neq a,b,\ell}u^{(k)}_{i_k}\Bigr)
+\sum_{\substack{\ell\neq m\\ \ell,m\neq a,b}}
\Bigl(Du^{(\ell)}_{i_\ell}Du^{(m)}_{i_m}\prod_{k\neq a,b,\ell,m}u^{(k)}_{i_k}\Bigr).
\end{align}
Using $\|D^2\bm N\|\le \|D^2\bm N\|_{\mathrm{HS}}$ and the fact that, for fixed
$(i_1,\ldots,i_d)$, each off-diagonal block $(a,b)$ has at most one potentially
nonzero entry at $(i_a,i_b)$, we have
\[
\|D^2\bm N\|_{\mathrm{HS}}^2
=
\sum_{a\neq b}\bigl|(D^2\bm N)^{ab}_{i_ai_b}\bigr|^2
=
\frac{1}{N}\sum_{a\neq b}
\Bigl|D^2\!\Bigl(\prod_{k\neq a,b}u^{(k)}_{i_k}\Bigr)\Bigr|^2,
\]
hence (using that $d$ is fixed, so $\ell^2\le C\,\ell^1$)
\begin{equation}\label{eq:D2N-bound}
\|D^2\bm N\|
\;\lesssim\;
\frac{1}{\sqrt N}\sum_{a\neq b}
\Bigl|D^2\!\Bigl(\prod_{k\neq a,b}u^{(k)}_{i_k}\Bigr)\Bigr|.
\end{equation}
Plugging (\ref{eq:D2N-bound}) into (\ref{eq:S1rem-1}) and using (\ref{eq:D2prod})
give the decomposition
\begin{equation}\label{eq:S1rem-T12}
S_{1,\mathrm{rem}}\;\lesssim\;T_1+T_2,
\end{equation}
where
\begin{align}
T_1
&:=
\frac{1}{\sqrt N}\sum_{a\neq b}\sum_{\ell\neq a,b}
\sum_{i_1,\ldots,i_d}
\bigl|D^2u^{(\ell)}_{i_\ell}\bigr|
\prod_{k\neq a,b,\ell}\bigl|u^{(k)}_{i_k}\bigr|
\prod_{k\neq 1,j}\bigl|u^{(k)}_{i_k}\bigr|,
\label{eq:T1-def}
\\
T_2
&:=
\frac{1}{\sqrt N}\sum_{a\neq b}\sum_{\substack{\ell\neq m\\ \ell,m\neq a,b}}
\sum_{i_1,\ldots,i_d}
\bigl|Du^{(\ell)}_{i_\ell}\bigr|\ \bigl|Du^{(m)}_{i_m}\bigr|
\prod_{k\neq a,b,\ell,m}\bigl|u^{(k)}_{i_k}\bigr|
\prod_{k\neq 1,j}\bigl|u^{(k)}_{i_k}\bigr|.
\label{eq:T2-def}
\end{align}
Importantly, we keep the full sum over $(i_1,\ldots,i_d)$ and do \emph{not} peel
off the indices $i_1,i_j$ in advance; instead we estimate the whole product
coordinate-wise.

\paragraph*{Bound on $T_1$}
Fix $(a,b,\ell)$ with $\ell\neq a,b$.
For each coordinate $r\in[d]\setminus\{\ell\}$, the factor $|u^{(r)}_{i_r}|$
appears with exponent $\alpha_r\in\{0,1,2\}$ in the product
\[
\prod_{k\neq a,b,\ell}|u^{(k)}_{i_k}|\ \prod_{k\neq 1,j}|u^{(k)}_{i_k}|,
\]
namely $\alpha_r=0$ if $r\in\{a,b,\ell\}\cap\{1,j\}$, $\alpha_r=2$ if
$r\notin\{a,b,\ell,1,j\}$, and $\alpha_r=1$ otherwise.
Hence the sum factorizes as
\[
\sum_{i_1,\ldots,i_d}
|D^2u^{(\ell)}_{i_\ell}|
\prod_{k\neq a,b,\ell}|u^{(k)}_{i_k}|
\prod_{k\neq 1,j}|u^{(k)}_{i_k}|
=
\Big(\sum_{i_\ell}|D^2u^{(\ell)}_{i_\ell}||u_{i_\ell}^{(\ell)}|^{\alpha_\ell}\Big)
\prod_{r\neq \ell}\Big(\sum_{i_r}|u^{(r)}_{i_r}|^{\alpha_r}\Big).
\]
We bound
\[
\sum_{i_\ell}|D^2u^{(\ell)}_{i_\ell}||u_{i_\ell}^{(\ell)}|^{\alpha_\ell}
\le \,\|D^2\bm u^{(\ell)}\|_2\||u_{i_\ell}^{(\ell)}|^{\alpha_\ell}\|_2,
\]
and for the $u$-sums we use
\[
\sum_{i_r}|u^{(r)}_{i_r}|^2=1,
\qquad
\sum_{i_r}|u^{(r)}_{i_r}| \le \sqrt N\,\|u^{(r)}\|_2=\sqrt N,
\qquad
\sum_{i_r}1 = N.
\]
Therefore, for each fixed $(a,b,\ell)$,
\begin{equation}
\sum_{i_1,\ldots,i_d}
|D^2u^{(\ell)}_{i_\ell}|
\prod_{k\neq a,b,\ell}|u^{(k)}_{i_k}|
\prod_{k\neq 1,j}|u^{(k)}_{i_k}|
\;\lesssim\;
\bigl(\|D^2\bm u^{(\ell)}\|_2N^{\frac{1}{2}\bm 1_{\ell\in\{1,j\}}}\bigr)\cdot N^{\#\{r\neq \ell:\alpha_r=0\}}\cdot N^{\frac12\#\{r\neq \ell:\alpha_r=1\}}.
\label{eqA.90}
\end{equation}
We now enumerate all possibilities for the overlaps of $(a,b,\ell)$ with
$\{1,j\}$ and count
\(
\#\{r\neq\ell:\alpha_r=0\}
\)
and
\(
\#\{r\neq\ell:\alpha_r=1\}.
\)
Throughout, recall $\ell\neq a,b$.

\medskip
\noindent\textbf{Case I: $\ell\in\{1,j\}$.}
Let $s$ be the other element of $\{1,j\}$, i.e.,\ $\{1,j\}=\{\ell,s\}$.

\smallskip
\noindent\emph{(I.1) $\ell\in\{1,j\}$ and $a,b\notin\{1,j\}$.}
Then $\bm 1_{\ell\in\{1,j\}}=1$, $\#\{r\neq\ell:\alpha_r=0\}=0$, and
$\#\{r\neq\ell:\alpha_r=1\}=3$ (coming from $r=s,a,b$).  Hence,
(\ref{eqA.90}) gives $\Sigma_{a,b,\ell}\lesssim \|D^2\bm u^{(\ell)}\|_2\,N^2$.

\smallskip
\noindent\emph{(I.2) $\ell\in\{1,j\}$ and $\{a,b\}\cap\{1,j\}=\{s\}$.}
Then $\bm 1_{\ell\in\{1,j\}}=1$, $\#\{r\neq\ell:\alpha_r=0\}=1$ (at $r=s$),
and $\#\{r\neq\ell:\alpha_r=1\}=1$ (at the other index among $\{a,b\}$).  Hence,
(\ref{eqA.90}) gives $\Sigma_{a,b,\ell}\lesssim \|D^2\bm u^{(\ell)}\|_2\,N^2$.

\medskip
\noindent\textbf{Case II: $\ell\notin\{1,j\}$.}
Then $\bm 1_{\ell\in\{1,j\}}=0$ and both $r=1$ and $r=j$ belong to the
counting set $r\neq\ell$.

\smallskip
\noindent\emph{(II.1) $\ell\notin\{1,j\}$ and $a,b\notin\{1,j\}$.}
Then $\#\{r\neq\ell:\alpha_r=0\}=0$ and $\#\{r\neq\ell:\alpha_r=1\}=4$
(coming from $r=1,j,a,b$). Hence, (\ref{eqA.90}) gives
$\Sigma_{a,b,\ell}\lesssim \|D^2\bm u^{(\ell)}\|_2\,N^2$.

\smallskip
\noindent\emph{(II.2) $\ell\notin\{1,j\}$ and $|\{a,b\}\cap\{1,j\}|=1$.}
Then $\#\{r\neq\ell:\alpha_r=0\}=1$ (at the intersecting element in $\{1,j\}$),
and $\#\{r\neq\ell:\alpha_r=1\}=2$ (at the remaining element of $\{1,j\}$ and at
the other index among $\{a,b\}$). Hence, (\ref{eqA.90}) gives
$\Sigma_{a,b,\ell}\lesssim \|D^2\bm u^{(\ell)}\|_2\,N^2$.

\smallskip
\noindent\emph{(II.3) $\ell\notin\{1,j\}$ and $\{a,b\}=\{1,j\}$.}
Then $\#\{r\neq\ell:\alpha_r=0\}=2$ (at $r=1$ and $r=j$), and
$\#\{r\neq\ell:\alpha_r=1\}=0$. Hence, (\ref{eqA.90}) gives
$\Sigma_{a,b,\ell}\lesssim \|D^2\bm u^{(\ell)}\|_2\,N^2$.

\medskip
\noindent
In all the cases we have shown the uniform bound
\[
\Sigma_{a,b,\ell}\;\lesssim\; \|D^2\bm u^{(\ell)}\|_2\,N^2 .
\]
Plugging this into (\ref{eq:T1-def}) yields
\[
T_1
\lesssim
\frac{1}{\sqrt N}\sum_{a\neq b}\sum_{\ell\neq a,b}
\|D^2\bm u^{(\ell)}\|_2\,N^2
\;\lesssim\;
\frac{N^2}{\sqrt N}\sum_{\ell=1}^d \|D^2\bm u^{(\ell)}\|_2,
\]
and using the bound $\|D^2\bm u^{(\ell)}\|_2\lesssim N^{-1}$ (proved as in the
$d=3$ case) we conclude $T_1=\mathcal O(\sqrt{N})$.

\paragraph*{Bound on $T_2$}
We treat the term $T_2$ by the same coordinate-wise counting argument as in the
estimate of $T_1$.  Fix $(a,b,\ell,m)$ with $\ell\neq m$ and $\ell,m\neq a,b$, and
set
\[
\Sigma^{(2)}_{a,b,\ell,m}
:=
\sum_{i_1,\ldots,i_d}
\bigl|Du^{(\ell)}_{i_\ell}\bigr|\,
\bigl|Du^{(m)}_{i_m}\bigr|
\prod_{k\neq a,b,\ell,m}\bigl|u^{(k)}_{i_k}\bigr|
\prod_{k\neq 1,j}\bigl|u^{(k)}_{i_k}\bigr|.
\]

As before, for each coordinate $r\in[d]$ we denote by $\alpha_r\in\{0,1,2\}$ the
total exponent of $\bigl|u^{(r)}_{i_r}\bigr|$ appearing in the product
\(
\prod_{k\neq a,b,\ell,m}|u^{(k)}_{i_k}|\prod_{k\neq 1,j}|u^{(k)}_{i_k}|
\),
with the understanding that the factors corresponding to $r=\ell,m$ are treated
separately.  For $r=\ell,m$ we apply the Cauchy--Schwarz inequality: 
\[
\sum_{i_\ell}\bigl|Du^{(\ell)}_{i_\ell}\bigr|\bigl|u_{i_\ell}^{(\ell)}\bigr|^{\alpha_\ell}
\le \|D\bm u^{(\ell)}\|_2\|\bigl|u_{i_\ell}^{(\ell)}\bigr|^{\alpha_\ell}\|_2,
\qquad
\sum_{i_m}\bigl|Du^{(m)}_{i_m}\bigr|^{\alpha_m}
\le \|D\bm u^{(m)}\|_2\|\bigl|u_{i_m}^{(m)}\bigr|^{\alpha_m}\|_2.
\]
For the remaining coordinates $r\neq \ell,m$ we use
\[
\sum_{i_r}1=N,\qquad
\sum_{i_r}\bigl|u^{(r)}_{i_r}\bigr|\le \sqrt N,\qquad
\sum_{i_r}\bigl|u^{(r)}_{i_r}\bigr|^2=1.
\]
Combining these bounds yields the general estimate
\begin{equation}\label{eqA.91}
\Sigma^{(2)}_{a,b,\ell,m}
\;\lesssim\;
\bigl(\|D\bm u^{(\ell)}\|_2N^{\frac{1}{2}\bm 1_{\alpha_\ell = 0}}\bigr)
\bigl(\|D\bm u^{(m)}\|_2N^{\frac{1}{2}\bm 1_{\alpha_m = 0}}\bigr)\cdot
N^{\#\{r\notin\{\ell,m\}:\alpha_r=0\}}\cdot
N^{\frac12\#\{r\notin\{\ell,m\}:\alpha_r=1\}} .
\end{equation}

We now classify (\ref{eqA.91}) case by case in the same way as for
(\ref{eqA.90}).  Recall that $\ell\neq m$ and $\ell,m\neq a,b$.  For any
$r\notin\{\ell,m\}$, the exponent $\alpha_r$ of $|u^{(r)}_{i_r}|$ in
\[
\prod_{k\neq a,b,\ell,m}|u^{(k)}_{i_k}|\ \prod_{k\neq 1,j}|u^{(k)}_{i_k}|
\]
satisfies $\alpha_r\in\{0,1,2\}$, and in fact
\[
\alpha_r=0 \ \Longrightarrow\ r\in\{1,j\}\cap\{a,b\},
\qquad
\alpha_r=2 \ \Longrightarrow\ r\notin\{1,j\}\cup\{a,b\}.
\]
Moreover, for the derivative coordinates we have
\[
\bm 1_{\alpha_\ell=0}= \bm 1_{\{\ell\in\{1,j\}\}},
\qquad
\bm 1_{\alpha_m=0}= \bm 1_{\{m\in\{1,j\}\}},
\]
since the first product always excludes $\ell,m$ while the second product
excludes exactly $\{1,j\}$.

\medskip
\noindent\textbf{Case I: $\{\ell,m\}=\{1,j\}$.}
Then $\bm 1_{\alpha_\ell=0}=\bm 1_{\alpha_m=0}=1$ and
$a,b\notin\{1,j\}$. Hence,
\[
\#\{r\notin\{\ell,m\}:\alpha_r=0\}=0,\qquad
\#\{r\notin\{\ell,m\}:\alpha_r=1\}=2,
\]
and (\ref{eqA.91}) gives $\Sigma^{(2)}_{a,b,\ell,m}\lesssim
\|D\bm u^{(\ell)}\|_2\|D\bm u^{(m)}\|_2\,N^2$.

\medskip
\noindent\textbf{Case II: exactly one of $\ell,m$ lies in $\{1,j\}$.}
By symmetry assume $\ell\in\{1,j\}$ and $m\notin\{1,j\}$, and let
$s:=\{1,j\}\setminus\{\ell\}$.

\smallskip
\noindent\emph{(II.1) $\ell\in\{1,j\}$, $m\notin\{1,j\}$, and $a,b\notin\{1,j\}$.}
Then
\[
\bm 1_{\alpha_\ell=0}=1,\quad \bm 1_{\alpha_m=0}=0,\quad
\#\{r\notin\{\ell,m\}:\alpha_r=0\}=0,\quad
\#\{r\notin\{\ell,m\}:\alpha_r=1\}=3,
\]
(where the $\alpha_r=1$ coordinates are $r=s,a,b$). Thus, (\ref{eqA.91}) yields
$\Sigma^{(2)}_{a,b,\ell,m}\lesssim \|D\bm u^{(\ell)}\|_2\|D\bm u^{(m)}\|_2\,N^2$.

\smallskip
\noindent\emph{(II.2) $\ell\in\{1,j\}$, $m\notin\{1,j\}$, and $\{a,b\}\cap\{1,j\}=\{s\}$.}
Then
\[
\bm 1_{\alpha_\ell=0}=1,\quad \bm 1_{\alpha_m=0}=0,\quad
\#\{r\notin\{\ell,m\}:\alpha_r=0\}=1,\quad
\#\{r\notin\{\ell,m\}:\alpha_r=1\}=1,
\]
(the $\alpha_r=0$ coordinate is $r=s$, and the remaining $\alpha_r=1$ coordinate
is the other element of $\{a,b\}$). Thus, (\ref{eqA.91}) yields
$\Sigma^{(2)}_{a,b,\ell,m}\lesssim \|D\bm u^{(\ell)}\|_2\|D\bm u^{(m)}\|_2\,N^2$.

\medskip
\noindent\textbf{Case III: $\ell,m\notin\{1,j\}$.}
Then $\bm 1_{\alpha_\ell=0}=\bm 1_{\alpha_m=0}=0$.  We split according
to the overlap of $\{a,b\}$ with $\{1,j\}$.

\smallskip
\noindent\emph{(III.1) $\{a,b\}\cap\{1,j\}=\emptyset$.}
Then $\alpha_1=\alpha_j=1$, so
\[
\#\{r\notin\{\ell,m\}:\alpha_r=0\}=0,\qquad
\#\{r\notin\{\ell,m\}:\alpha_r=1\}=4
\quad(\text{from }r=1,j,a,b),
\]
and (\ref{eqA.91}) yields $\Sigma^{(2)}_{a,b,\ell,m}\lesssim
\|D\bm u^{(\ell)}\|_2\|D\bm u^{(m)}\|_2\,N^2$.

\smallskip
\noindent\emph{(III.2) $|\{a,b\}\cap\{1,j\}|=1$.}
Then exactly one of $\{1,j\}$ (say $1$) lies in $\{a,b\}$, hence $\alpha_1=0$
and $\alpha_j=1$. Moreover, the other element of $\{a,b\}$ lies outside
$\{1,j\}$ and contributes one more $\alpha_r=1$. Therefore
\[
\#\{r\notin\{\ell,m\}:\alpha_r=0\}=1,\qquad
\#\{r\notin\{\ell,m\}:\alpha_r=1\}=2,
\]
and (\ref{eqA.91}) yields $\Sigma^{(2)}_{a,b,\ell,m}\lesssim
\|D\bm u^{(\ell)}\|_2\|D\bm u^{(m)}\|_2\,N^2$.

\smallskip
\noindent\emph{(III.3) $\{a,b\}=\{1,j\}$.}
Then $\alpha_1=\alpha_j=0$, so
\[
\#\{r\notin\{\ell,m\}:\alpha_r=0\}=2,\qquad
\#\{r\notin\{\ell,m\}:\alpha_r=1\}=0,
\]
and (\ref{eqA.91}) yields $\Sigma^{(2)}_{a,b,\ell,m}\lesssim
\|D\bm u^{(\ell)}\|_2\|D\bm u^{(m)}\|_2\,N^2$.

\medskip
\noindent
In all the cases we have the uniform bound
\[
\Sigma^{(2)}_{a,b,\ell,m}\ \lesssim\ \|D\bm u^{(\ell)}\|_2\|D\bm u^{(m)}\|_2\,N^2.
\]
Plugging this into (\ref{eq:T2-def}) gives
\[
T_2
\lesssim
\frac{1}{\sqrt N}\sum_{a\neq b}\sum_{\substack{\ell\neq m\\ \ell,m\neq a,b}}
\|D\bm u^{(\ell)}\|_2\|D\bm u^{(m)}\|_2\,N^2
\;\lesssim\;
N^{3/2}\sum_{\ell\neq m}\|D\bm u^{(\ell)}\|_2\|D\bm u^{(m)}\|_2,
\]
and using $\|D\bm u^{(\ell)}\|_2\lesssim N^{-1/2}$ and that $d$ is fixed, we obtain
$T_2=\mathcal O(\sqrt{N})$.

\subsubsection*{Conclusion}
Since \(S \le S_{1,main}+S_{1,rem} = \mathcal{O}(\sqrt{N})\), we complete the proof.

\subsection{Proof of \(S_2\)}\label{secS2}
Recall that $S_2$ corresponds to the cross term $(DQ^{1j}_{i_1 i_j})(DU)$ in the
expansion of $D^2F$. By definition,
\[
S_2
\;\lesssim\;
\sum_{i_1,\ldots,i_d}\sup_{z\in\mathbb C^+_{\eta_0}}
\bigl|DQ^{1j}_{i_1 i_j}(z)\bigr|\cdot |DU|.
\]

\subsubsection*{Expansion of $DQ^{1j}$ and $DU$}
Using the resolvent identity $DQ=-Q(DN)Q$ and the explicit form of $DN$, we have
\begin{equation}\label{eq:S2-DQ}
DQ^{1j}_{i_1 i_j}(z)
=
-\frac{1}{N}\sum_{a\neq b}
\Big(\prod_{k\neq a,b}u^{(k)}_{i_k}\Big)
Q^{1a}_{i_1 i_a}(z)\,Q^{bj}_{i_b i_j}(z).
\end{equation}
Moreover,
\begin{equation}\label{eq:S2-DU}
DU
=
\sum_{\ell\neq 1,j}
\bigl(Du^{(\ell)}_{i_\ell}\bigr)
\prod_{k\neq 1,j,\ell}u^{(k)}_{i_k}.
\end{equation}
Taking absolute values, the supremum over $z$, and bounding resolvent entries by
the operator norm yields
\[
\sup_{z\in\mathbb C^+_{\eta_0}}\bigl|Q^{1a}_{i_1 i_a}(z)\bigr|
\le \sup_{z\in\cal S_{\eta_0}}\|\bm Q(z)\|,
\qquad
\sup_{z\in\mathbb C^+_{\eta_0}}\bigl|Q^{bj}_{i_b i_j}(z)\bigr|
\le \sup_{z\in\mathbb C^+_{\eta_0}}\|\bm Q(z)\|.
\]
Hence,
\begin{equation}\label{eq:S2-bound1}
\sup_{z}\bigl|DQ^{1j}_{i_1 i_j}(z)\bigr|
\;\lesssim\;
\frac{\sup_{z}\|\bm Q(z)\|^2}{N}
\sum_{a\neq b}
\prod_{k\neq a,b}\bigl|u^{(k)}_{i_k}\bigr|.
\end{equation}

\subsubsection*{Insertion into $S_2$}
Combining (\ref{eq:S2-DU}) and (\ref{eq:S2-bound1}) and exchanging sums (all terms
are nonnegative), we obtain
\begin{align}
S_2
&\lesssim
\frac{\sup_{z}\|\bm Q(z)\|^2}{N}
\sum_{a\neq b}\sum_{\ell\neq 1,j}
\sum_{i_1,\ldots,i_d}
\bigl|Du^{(\ell)}_{i_\ell}\bigr|
\Big(\prod_{k\neq a,b}\bigl|u^{(k)}_{i_k}\bigr|\Big)
\Big(\prod_{k\neq 1,j,\ell}\bigl|u^{(k)}_{i_k}\bigr|\Big).
\label{eq:S2-main}
\end{align}

\subsubsection*{Summation over $i_\ell$}
Fix $(a,b,\ell)$. The factors involving the $\ell$th coordinate take the form
\[
|Du^{(\ell)}_{i_\ell}|\cdot |u^{(\ell)}_{i_\ell}|^{\beta_\ell},
\qquad
\beta_\ell=
\begin{cases}
0, & \ell\in\{a,b\},\\
1, & \ell\notin\{a,b\}.
\end{cases}
\]
Applying Cauchy--Schwarz in $i_\ell$ gives
\[
\sum_{i_\ell}|Du^{(\ell)}_{i_\ell}|\ |u^{(\ell)}_{i_\ell}|^{\beta_\ell}
\le
\|D\bm u^{(\ell)}\|_2
\Big(\sum_{i_\ell}|u^{(\ell)}_{i_\ell}|^{2\beta_\ell}\Big)^{1/2}
\lesssim
\|D\bm u^{(\ell)}\|_2\,N^{\frac12\bm 1_{\ell\in\{a,b\}}}.
\]

\subsubsection*{Summation in full detail: exhaustive case discussion}
We now justify the uniform bound
\[
\sum_{i_1,\ldots,i_d}
\bigl|Du^{(\ell)}_{i_\ell}\bigr|
\Big(\prod_{k\neq a,b}\bigl|u^{(k)}_{i_k}\bigr|\Big)
\Big(\prod_{k\neq 1,j,\ell}\bigl|u^{(k)}_{i_k}\bigr|\Big)
\ \lesssim\ 
\|D\bm u^{(\ell)}\|_2\,N^2,
\]
by an explicit classification of all cases.

\medskip
\noindent\textbf{Setup.}
Fix $(a,b,\ell)$ with $a\neq b$ and $\ell\neq 1,j$ (since $\ell$ comes from
$DU=\sum_{\ell\neq 1,j}\cdots$). Define
\[
\Sigma_{a,b,\ell}
:=
\sum_{i_1,\ldots,i_d}
\bigl|Du^{(\ell)}_{i_\ell}\bigr|
\Big(\prod_{k\neq a,b}\bigl|u^{(k)}_{i_k}\bigr|\Big)
\Big(\prod_{k\neq 1,j,\ell}\bigl|u^{(k)}_{i_k}\bigr|\Big).
\]
For each coordinate $r\neq \ell$, let $\alpha_r\in\{0,1,2\}$ be the exponent of
$\bigl|u^{(r)}_{i_r}\bigr|$ contributed by the two products
\(
\prod_{k\neq a,b}|u^{(k)}_{i_k}|
\)
and
\(
\prod_{k\neq 1,j,\ell}|u^{(k)}_{i_k}|.
\)
Then for $r\neq \ell$ one checks:
\[
\alpha_r=
\bm 1\{r\neq a,b\}+\bm 1\{r\neq 1,j,\ell\}
=
\begin{cases}
2,& r\notin \{a,b\}\cup\{1,j\},\\
1,& r\in (\{a,b\}\triangle\{1,j\}),\\
0,& r\in \{a,b\}\cap\{1,j\},
\end{cases}
\quad (r\neq \ell),
\]
where we used $r\neq\ell$ to simplify $\bm 1\{r\neq 1,j,\ell\}$.
Hence, coordinate-wise summation gives
\[
\sum_{i_r}\bigl|u^{(r)}_{i_r}\bigr|^{\alpha_r}
\le
\begin{cases}
N,& \alpha_r=0,\\
\sqrt N,& \alpha_r=1,\\
1,& \alpha_r=2,
\end{cases}
\qquad\text{since } \sum_i|u_i|=\|u\|_1\le\sqrt N\|u\|_2=\sqrt N,\ \sum_i|u_i|^2=1.
\]

For the $\ell$-coordinate, the total factor in $\Sigma_{a,b,\ell}$ is
\[
|Du^{(\ell)}_{i_\ell}|\cdot |u^{(\ell)}_{i_\ell}|^{\beta_\ell},
\qquad
\beta_\ell=
\begin{cases}
0,& \ell\in\{a,b\},\\
1,& \ell\notin\{a,b\},
\end{cases}
\]
because the second product always excludes $\ell$, while the first product
includes $\ell$ iff $\ell\notin\{a,b\}$. By the Cauchy--Schwarz inequality, we obtain
\[
\sum_{i_\ell}|Du^{(\ell)}_{i_\ell}|\ |u^{(\ell)}_{i_\ell}|^{\beta_\ell}
\le
\|D\bm u^{(\ell)}\|_2\Big(\sum_{i_\ell}|u^{(\ell)}_{i_\ell}|^{2\beta_\ell}\Big)^{1/2}
\lesssim
\|D\bm u^{(\ell)}\|_2\,N^{\frac12\bm 1_{\ell\in\{a,b\}}}.
\]

Therefore, it remains to compute, for each configuration of $(a,b,\ell)$,
the two counts
\[
n_0:=\#\{r\neq \ell:\alpha_r=0\}=\#\bigl((\{a,b\}\cap\{1,j\})\setminus\{\ell\}\bigr),
\qquad
n_1:=\#\{r\neq \ell:\alpha_r=1\},
\]
and then we will have
\begin{equation}\label{eq:S2-case-template}
\Sigma_{a,b,\ell}
\ \lesssim\
\|Du^{(\ell)}\|_2\,N^{\frac12\bm 1_{\ell\in\{a,b\}}}\cdot N^{n_0}\cdot N^{\frac12 n_1}.
\end{equation}

\medskip
\noindent\textbf{Main split according to $\{a,b\}\cap\{1,j\}$.}
Let $t:=|\{a,b\}\cap\{1,j\}|\in\{0,1,2\}$. We treat $t=2,1,0$ separately, and in
each case split according to whether $\ell\in\{a,b\}$.

\bigskip
\noindent\textbf{Case 1: $t=2$, i.e.\ $\{a,b\}=\{1,j\}$.}
Since $\ell\neq 1,j$, we have $\ell\notin\{a,b\}$, hence
\[
\bm 1_{\ell\in\{a,b\}}=0,\qquad \beta_\ell=1.
\]
Moreover, for $r\neq \ell$:
\[
\alpha_1=\alpha_j=0,\qquad
\alpha_r=2 \ \text{ for all } r\notin\{1,j,\ell\}.
\]
Thus
\[
n_0=\#\{r\neq \ell:\alpha_r=0\}=2,\qquad
n_1=\#\{r\neq \ell:\alpha_r=1\}=0.
\]
Plugging into (\ref{eq:S2-case-template}) yields
\[
\Sigma_{a,b,\ell}\ \lesssim\ \|D\bm u^{(\ell)}\|_2\cdot N^{0}\cdot N^{2}\cdot N^{0}
=\|D\bm u^{(\ell)}\|_2\,N^{2}.
\]

\bigskip
\noindent\textbf{Case 2: $t=1$, i.e.,\ $|\{a,b\}\cap\{1,j\}|=1$.}
Without loss of generality, assume that $a=1$ and $b\notin\{1,j\}$
(the other choices are symmetric). Then for $r\neq \ell$:
\[
\alpha_1=0,\qquad
\alpha_j=1,\qquad
\alpha_b=1\ \text{ if } b\neq \ell,\qquad
\alpha_r=2\ \text{ for } r\notin\{1,j,b,\ell\}.
\]
We now make split depending on whether $\ell=b$.

\smallskip
\noindent\emph{(2.1) $\ell=b$ (so $\ell\in\{a,b\}$).}
Then $\bm 1_{\ell\in\{a,b\}}=1$ and $b=\ell$ is excluded from the counts
($r\neq\ell$). Hence
\[
n_0=1 \quad(\text{from } r=1),\qquad
n_1=1 \quad(\text{from } r=j).
\]
Thus, (\ref{eq:S2-case-template}) gives
\[
\Sigma_{a,b,\ell}\ \lesssim\ \|D\bm u^{(\ell)}\|_2\cdot N^{1/2}\cdot N^{1}\cdot N^{1/2}
=\|D\bm u^{(\ell)}\|_2\,N^{2}.
\]

\smallskip
\noindent\emph{(2.2) $\ell\neq b$ (so $\ell\notin\{a,b\}$).}
Then $\bm 1_{\ell\in\{a,b\}}=0$, and the two $\alpha=1$ coordinates are
$r=j$ and $r=b$. Hence
\[
n_0=1 \quad(\text{from } r=1),\qquad
n_1=2 \quad(\text{from } r=j,b).
\]
Therefore,
\[
\Sigma_{a,b,\ell}\ \lesssim\ \|D\bm u^{(\ell)}\|_2\cdot N^{0}\cdot N^{1}\cdot N^{1}
=\|D\bm u^{(\ell)}\|_2\,N^{2}.
\]

\bigskip
\noindent\textbf{Case 3: $t=0$, i.e.\ $\{a,b\}\cap\{1,j\}=\emptyset$.}
Then $1$ and $j$ are not removed by the first product, so for $r\neq\ell$ we have
\[
\alpha_1=1,\qquad \alpha_j=1.
\]
Also, since $a,b\notin\{1,j\}$, we have
\[
\alpha_a=
\begin{cases}
0,& a=\ell,\\
1,& a\neq\ell,
\end{cases}
\qquad
\alpha_b=
\begin{cases}
0,& b=\ell,\\
1,& b\neq\ell,
\end{cases}
\qquad
\alpha_r=2\ \text{ for } r\notin\{1,j,a,b,\ell\}.
\]
We split into two subcases depending on whether $\ell\in\{a,b\}$.

\smallskip
\noindent\emph{(3.1) $\ell\in\{a,b\}$.}
Then $\bm 1_{\ell\in\{a,b\}}=1$ and exactly one of $a,b$ equals $\ell$,
so among $r\neq\ell$ only the remaining one contributes $\alpha=1$.
Hence
\[
n_0=0,\qquad
n_1=3 \quad(\text{from } r=1,j,\text{ and the unique element of }\{a,b\}\setminus\{\ell\}).
\]
Therefore,
\[
\Sigma_{a,b,\ell}\ \lesssim\ \|D\bm u^{(\ell)}\|_2\cdot N^{1/2}\cdot N^{0}\cdot N^{3/2}
=\|D\bm u^{(\ell)}\|_2\,N^{2}.
\]

\smallskip
\noindent\emph{(3.2) $\ell\notin\{a,b\}$.}
Then $\bm 1\{\ell\in\{a,b\}\}=0$, and the $\alpha=1$ coordinates are
$r=1$, $r=j$, $r=a$, and $r=b$. Hence
\[
n_0=0,\qquad n_1=4.
\]
Thus
\[
\Sigma_{a,b,\ell}\ \lesssim\ \|D\bm u^{(\ell)}\|_2\cdot N^{0}\cdot N^{0}\cdot N^{2}
=\|D\bm u^{(\ell)}\|_2\,N^{2}.
\]

\bigskip
\noindent\textbf{Uniform conclusion for all cases.}
In every possible configuration of $(a,b,\ell)$ with $\ell\neq 1,j$, we have shown
\[
\Sigma_{a,b,\ell}\ \lesssim\ \|Du^{(\ell)}\|_2\,N^{2}.
\]
Returning to (\ref{eq:S2-main}), we obtain
\[
S_2
\ \lesssim\
\frac{\sup_{z}\|\bm Q(z)\|^2}{N}\sum_{a\neq b}\sum_{\ell\neq 1,j}\Sigma_{a,b,\ell}
\ \lesssim\
\sup_{z}\|\bm Q(z)\|^2\cdot N\sum_{\ell\neq 1,j}\|D\bm u^{(\ell)}\|_2.
\]
Using $\sup_{z\in\mathbb C^+_{\eta_0}}\|\bm Q(z)\|\lesssim 1$ and
$\|D\bm u^{(\ell)}\|_2=\mathcal O(N^{-1/2})$ (as in Theorem~\ref{thm3.3}), we conclude that
$S_2=\mathcal O(\sqrt{N})$.

\subsection{Proof of $S_3$}\label{secS3}

Recall
\[
U=\prod_{k\neq 1,j}u^{(k)}_{i_k},\qquad 
D=\frac{\partial}{\partial W_{i_1,\cdots, i_d}},
\qquad 
Q(z)=(N-zI)^{-1},\ z\in\mathbb C^+_{\eta_0}.
\]
We prove
\begin{equation}\label{eq:S3Goal}
S_3^\star:=\sum_{i_1,\ldots,i_d}\sup_{z\in\mathbb C^+_{\eta_0}}
\bigl|Q^{1j}_{i_1 i_j}(z)\,(D^2U)\bigr|
= \mathcal O(N),
\end{equation}
(and in fact we will obtain the sharper bound $\mathcal O(\sqrt N)$).

\subsubsection*{\texorpdfstring{Step 1: Expand $D^2U$ into ``$D^2\bm u$'' and ``$D\bm u\cdot D\bm u$'' parts}{}}
By the product rule,
\begin{equation}\label{eq:D2U-decomp}
D^2U
=
\sum_{\ell\neq 1,j}
\Bigl(D^2u^{(\ell)}_{i_\ell}\Bigr)\prod_{k\neq 1,j,\ell}u^{(k)}_{i_k}
+
\sum_{\substack{\ell\neq m\\ \ell,m\neq 1,j}}
\Bigl(Du^{(\ell)}_{i_\ell}\Bigr)\Bigl(Du^{(m)}_{i_m}\Bigr)
\prod_{k\neq 1,j,\ell,m}u^{(k)}_{i_k}.
\end{equation}
Hence, using $|a+b|\le |a|+|b|$,
\[
S_3^\star\le S_{3,A}^\star+S_{3,B}^\star,
\]
where
\begin{align}
S_{3,A}^\star
&:=
\sum_{i_1,\ldots,i_d}\sup_{z\in\mathbb C^+_{\eta_0}}
|Q^{1j}_{i_1 i_j}(z)|\,
\sum_{\ell\neq 1,j}
\Bigl|D^2u^{(\ell)}_{i_\ell}\Bigr|
\prod_{k\neq 1,j,\ell}\bigl|u^{(k)}_{i_k}\bigr|,
\label{eq:S3A-def}\\
S_{3,B}^\star
&:=
\sum_{i_1,\ldots,i_d}\sup_{z\in\mathbb C^+_{\eta_0}}
|Q^{1j}_{i_1 i_j}(z)|\,
\sum_{\substack{\ell\neq m\\ \ell,m\neq 1,j}}
\bigl|Du^{(\ell)}_{i_\ell}\bigr|\,
\bigl|Du^{(m)}_{i_m}\bigr|
\prod_{k\neq 1,j,\ell,m}\bigl|u^{(k)}_{i_k}\bigr|.
\label{eq:S3B-def}
\end{align}
Since $d$ is fixed, it suffices to show that each inner sum over $\ell$ (or
$\ell\neq m$) contributes at most $\mathcal O(N)$ uniformly; the finite sums in
$\ell,m$ only change constants.

\subsubsection*{Step 2: A basic resolvent--vector inequality to avoid the $N^2$ blow-up}
A key point is that we \emph{must not} bound
$|Q^{1j}_{i_1 i_j}|\le \|\bm Q\|$, because then the double free sum in $(i_1,i_j)$ could
produce an $N^2$ factor. Instead, whenever a factor $|u^{(1)}_{i_1}|$ or
$|u^{(j)}_{i_j}|$ is present, we use Cauchy--Schwarz to absorb the corresponding
index into $\|\bm Q\|$.

Fix $z\in\mathbb C^+_{\eta_0}$. For any fixed $i_j$,
\begin{equation}\label{eq:Q-absorb-i1}
\sum_{i_1}|Q^{1j}_{i_1 i_j}(z)|\,|u^{(1)}_{i_1}|
\le 
\Big(\sum_{i_1}|Q^{1j}_{i_1 i_j}(z)|^2\Big)^{1/2}\|\bm u^{(1)}\|_2
\le \|\bm Q(z)\|.
\end{equation}
Similarly, for any fixed $i_1$,
\begin{equation}\label{eq:Q-absorb-ij}
\sum_{i_j}|Q^{1j}_{i_1 i_j}(z)|\,|u^{(j)}_{i_j}|
\le \|\bm Q(z)\|.
\end{equation}
Taking $\sup_{z\in\mathbb C^+_{\eta_0}}$ and using $\sup_z\|\bm Q(z)\|\le \eta_0^{-1}$,
we may replace the troublesome $|Q^{1j}_{i_1i_j}(z)|$ by $\|\bm Q(z)\|$ \emph{at the cost
of one factor $|u^{(1)}_{i_1}|$ or $|u^{(j)}_{i_j}|$ in the summand}. This is the
mechanism preventing the $N^2$ blow-up.

\subsubsection*{\texorpdfstring{Step 3: Express $D\bm u$ and $D^2\bm u$ using (\ref{eq4.3}) and keep the 
product structure.}{}}
We use the derivative representation (Eq.~(\ref{eq4.3})) in component form:
\begin{equation}\label{eq:Du-structure}
D\bm u^{(r)}
=
-\frac1{\sqrt N}\sum_{s=1}^d R_{rs}(\lambda)\,\bm b^{(s)},
\qquad r\in[d],
\end{equation}
where for each $s$,
\begin{equation}\label{eq:b-structure}
\bm b^{(s)}
=
\Big(\prod_{t\neq s}u^{(t)}_{i_t}\Big)\Big(\bm e^{n_s}_{i_s}-u^{(s)}_{i_s}\bm u^{(s)}\Big).
\end{equation}
\emph{Crucial:} we will \emph{never} bound $\bm{b}^{(s)}$ by a constant; we keep the
multiplicative factor $\prod_{t\neq s}u^{(t)}_{i_t}$, because it supplies
$u^{(1)}_{i_1}$ or $u^{(j)}_{i_j}$ needed to absorb $Q^{1j}_{i_1i_j}$.

Differentiating (\ref{eq:Du-structure}) once more gives
\begin{equation}\label{eq:D2u-decomp}
D^2u^{(r)}
=
-\frac1{\sqrt N}\sum_{s=1}^d\Big[(DR_{rs}(\lambda))\,b^{(s)} + R_{rs}(\lambda)\,(Db^{(s)})\Big].
\end{equation}
Thus the $D^2\bm u$ part further splits into ``$(D\bm R)\cdot \bm b$'' and ``$\bm R\cdot D\bm b$''.

Next, expand $D\bm b^{(s)}$ explicitly. Let
\[
P_s:=\prod_{t\neq s}u^{(t)}_{i_t},
\qquad
\bm g^{(s)}:=\bm e^{n_s}_{i_s}-u^{(s)}_{i_s}\bm u^{(s)},
\qquad
\bm b^{(s)}=P_s\,\bm g^{(s)}.
\]
Then
\begin{equation}\label{eq:Db-expand1}
D\bm b^{(s)}=(DP_s)\,\bm g^{(s)} + P_s\,(D\bm g^{(s)}).
\end{equation}
We have
\begin{equation}\label{eq:DPs-expand}
DP_s=\sum_{t\neq s}\Big(Du^{(t)}_{i_t}\Big)\prod_{r\neq s,t}u^{(r)}_{i_r},
\end{equation}
and since $D(\bm e^{n_s}_{i_s})=0$,
\begin{equation}\label{eq:Dg-expand}
D\bm g^{(s)}= -\Big(Du^{(s)}_{i_s}\Big)\bm u^{(s)} - u^{(s)}_{i_s}\,D\bm u^{(s)}.
\end{equation}
Therefore, $D \bm{b}^{(s)}$ is the sum of three types of terms:
\begin{align}
\text{(I)}\quad &
\sum_{t\neq s}\Big(Du^{(t)}_{i_t}\Big)\Big(\prod_{r\neq s,t}u^{(r)}_{i_r}\Big) \bm g^{(s)},
\label{eq:Db-typeI}\\
\text{(II)}\quad &
-\,P_s\Big(Du^{(s)}_{i_s}\Big)\bm u^{(s)},
\label{eq:Db-typeII}\\
\text{(III)}\quad &
-\,P_s\,u^{(s)}_{i_s}\,D\bm u^{(s)}.
\label{eq:Db-typeIII}
\end{align}
Each term contains at least one factor $D\bm u^{(\cdot)}$. Hence, (by (\ref{eq:Du-structure}))
each contributes an additional factor $N^{-1/2}$ and an additional ``product of
$u$'' of the form $\prod_{t\neq q}u^{(t)}_{i_t}$.

\subsubsection*{Step 4: A deterministic counting lemma for sums of products of $u$'s}
We repeatedly use the bounds
\begin{equation}\label{eq:u-sum-bounds}
\sum_{i_r}1=N,\qquad 
\sum_{i_r}|u^{(r)}_{i_r}|\le \sqrt N\|\bm u^{(r)}\|_2=\sqrt N,\qquad
\sum_{i_r}|u^{(r)}_{i_r}|^2=1,\qquad
\sum_{i_r}|u^{(r)}_{i_r}|^3\le 1.
\end{equation}
Thus, for a factor $|u^{(r)}_{i_r}|^{\alpha_r}$ with $\alpha_r\in\{0,1,2,3\}$,
the contribution of summing in $i_r$ is at most $N$ if $\alpha_r=0$, at most
$\sqrt N$ if $\alpha_r=1$, and at most $1$ if $\alpha_r\ge 2$.

\medskip
\noindent\textbf{Important structural fact.}
In every term arising from (\ref{eq:D2U-decomp}) together with
(\ref{eq:Du-structure})-(\ref{eq:Db-typeIII}), after multiplying by the outside
product $\prod_{k\neq 1,j,\ell(\!,m)}|u^{(k)}_{i_k}|$, the pair of indices $(i_1,i_j)$
cannot both appear with exponent $0$. Concretely, at least one of
$|u^{(1)}_{i_1}|$ or $|u^{(j)}_{i_j}|$ is present in the product.
This is because each $\bm b^{(s)}$ (and hence each $D\bm u$ or $D\bm b$ term) misses
\emph{at most one} coordinate, so it cannot simultaneously remove both coordinates
$1$ and $j$. This guarantees that we can apply either (\ref{eq:Q-absorb-i1}) or
(\ref{eq:Q-absorb-ij}) to absorb $Q^{1j}_{i_1i_j}(z)$ without paying an $N^2$ cost.

We now implement this principle in the two parts $S_{3,B}^\star$ and
$S_{3,A}^\star$.

%========================================================
\subsubsection*{Step 5: Bound on $S_{3,B}^\star$ (the $Du\cdot Du$ part)}

Fix distinct $\ell,m\neq 1,j$. Consider the contribution
\[
\Sigma_{\ell,m}
:=
\sum_{i_1,\ldots,i_d}\sup_{z\in\mathbb C^+_{\eta_0}}
|Q^{1j}_{i_1i_j}(z)|\,
|Du^{(\ell)}_{i_\ell}|\,
|Du^{(m)}_{i_m}|\,
\prod_{k\neq 1,j,\ell,m}|u^{(k)}_{i_k}|.
\]
Then
\[
S_{3,B}^\star \le \sum_{\substack{\ell\neq m\\ \ell,m\neq1,j}} \Sigma_{\ell,m}.
\]

\paragraph*{Expand $Du^{(\ell)}_{i_\ell}$ and $Du^{(m)}_{i_m}$ while keeping products}
From (\ref{eq:Du-structure})-(\ref{eq:b-structure}),
\[
|Du^{(\ell)}_{i_\ell}|
\le \frac{1}{\sqrt N}\sum_{s=1}^d |R_{\ell s}(\lambda)|\,|b^{(s)}_{i_\ell}|,
\qquad
|Du^{(m)}_{i_m}|
\le \frac{1}{\sqrt N}\sum_{t=1}^d |R_{m t}(\lambda)|\,|b^{(t)}_{i_m}|.
\]
Using that $d$ is fixed and $\|R(\lambda)\|\lesssim 1$ (as in the earlier parts),
we may absorb the finite sums in $s,t$ into constants and focus on a typical term
with fixed $(s,t)$.
Each factor $b^{(s)}$ contributes the product $P_s=\prod_{r\neq s}u^{(r)}_{i_r}$.
Hence a typical summand in $\Sigma_{\ell,m}$ is bounded by
\begin{equation}\label{eq:S3B-typical}
\frac{1}{N}
\sum_{i_1,\ldots,i_d}\sup_{z}
|Q^{1j}_{i_1i_j}(z)|\,
\Big(\prod_{r\neq s}|u^{(r)}_{i_r}|\Big)\,
\Big(\prod_{r\neq t}|u^{(r)}_{i_r}|\Big)\,
\prod_{k\neq 1,j,\ell,m}|u^{(k)}_{i_k}|.
\end{equation}
Using \(\sup_z|Q_{i_1i_j}^{1j}(z)|\le \|\bm Q\| \lesssim1\), consider 
\begin{equation}
    \prod_{r\neq s}|u^{(r)}_{i_r}|
    \prod_{r\neq t}|u^{(r)}_{i_r}|
    \prod_{k\neq 1,j,\ell,m}|u^{(k)}_{i_k}|
\end{equation}
It is easy to check that the exponent of \(u_{i_k}^{(k)}\) satisfies 
\begin{enumerate}
  \item[] Case 1: One of them equal to 0 and other \(\ge 2\);
  \item[] Case 2: Two of them equal to 1 and other \(\ge 2\);
  \item[] Case 3: All of them \(\ge2\).
\end{enumerate}
Combining with Eq.~(\ref{eq:u-sum-bounds}), we obtain
\begin{equation}\label{eq:S3B-final}
S_{3,B}^\star=\mathcal O(1).
\end{equation}

%========================================================
\subsubsection*{Step 6: Bound on $S_{3,A}^\star$ (the $D^2u$ part)}

Fix $\ell\neq 1,j$. Consider
\[
\Sigma_{\ell}
:=
\sum_{i_1,\ldots,i_d}\sup_{z\in\mathbb C^+_{\eta_0}}
|Q^{1j}_{i_1i_j}(z)|\,
\Bigl|D^2u^{(\ell)}_{i_\ell}\Bigr|
\prod_{k\neq 1,j,\ell}|u^{(k)}_{i_k}|.
\]
Then $S_{3,A}^\star\le \sum_{\ell\neq 1,j}\Sigma_\ell$.

Using (\ref{eq:D2u-decomp}), split $\Sigma_\ell\le \Sigma_{\ell,1}+\Sigma_{\ell,2}$, where
\begin{align}
\Sigma_{\ell,1}
&:=\frac1{\sqrt N}\sum_{s=1}^d
\sum_{i_1,\ldots,i_d}\sup_z
|Q^{1j}_{i_1i_j}(z)|\,|DR_{\ell s}(\lambda)|\,|b^{(s)}_{i_\ell}|\,
\prod_{k\neq 1,j,\ell}|u^{(k)}_{i_k}|,
\label{eq:Sigma-l1}\\
\Sigma_{\ell,2}
&:=\frac1{\sqrt N}\sum_{s=1}^d
\sum_{i_1,\ldots,i_d}\sup_z
|Q^{1j}_{i_1i_j}(z)|\,|R_{\ell s}(\lambda)|\,|Db^{(s)}_{i_\ell}|\,
\prod_{k\neq 1,j,\ell}|u^{(k)}_{i_k}|.
\label{eq:Sigma-l2}
\end{align}
We treat these two terms separately; the second one is precisely the ``$R_{\ell s}Db^{(s)}$''
term you asked to be fully expanded.

\paragraph*{\texorpdfstring{Bound on $\Sigma_{\ell,1}$ (the $(D\bm R)\cdot \bm b$ term)}{}}
We use the resolvent identity for $R$ (the resolvent of $\textbf{T}$ at $\lambda$):
\[
D\bm R(\lambda)=-\bm R(\lambda)\,(D\textbf{T})\,\bm R(\lambda).
\]
The dominating contribution in $D\textbf{T}$ is the noise part $D\bm N$; as in the earlier sections,
\[
D\bm N=\frac1{\sqrt N}\sum_{a\neq b}\Big(\prod_{t\neq a,b}u^{(t)}_{i_t}\Big)\bm E^{ab}_{i_ai_b}.
\]
Hence, for a typical term with fixed $(a,b,s)$, the product of $u$'s in the summand of
$\Sigma_{\ell,1}$ contains
\[
\Big(\prod_{k\neq 1,j,\ell}|u^{(k)}_{i_k}|\Big)\cdot
\Big(\prod_{t\neq a,b}|u^{(t)}_{i_t}|\Big)\cdot
\Big(\prod_{t\neq s}|u^{(t)}_{i_t}|\Big),
\]
and the total prefactor is $1/N$ (one $1/\sqrt N$ from (\ref{eq:Sigma-l1}) and one
$1/\sqrt N$ from $D\bm N$).

It is easy to check that the exponent of \(u_{i_k}^{(k)}\) satisfies 
\begin{enumerate}
  \item[] Case 1: One of them equal to 0,  one of them equal to 1 and other \(\ge 2\);
  \item[] Case 2: Three of them equal to 1 and other \(\ge 2\);
  \item[] Case 3: Other cases that make the order of the product smaller.
\end{enumerate}
Combining with Eq.~(\ref{eq:u-sum-bounds}), we obtain
\begin{equation}\label{eq:Sigma}
\sum_{\ell\neq 1,j}\Sigma_{\ell,1}=\mathcal O(\sqrt N).
\end{equation}

\paragraph*{\texorpdfstring{Bound on $\Sigma_{\ell,2}$ (the $R_{\ell s}D\bm b^{(s)}$ term): full expansion}{}}
We now bound (\ref{eq:Sigma-l2}) using the explicit decomposition
(\ref{eq:Db-expand1})-(\ref{eq:Db-typeIII}). Since $\|\bm R(\lambda)\|\lesssim 1$ and $d$
is fixed, it suffices to bound each of the three types (I)--(III) appearing in
$Db^{(s)}$.

\paragraph*{Type (I): terms from (\ref{eq:Db-typeI})}
Fix $s$ and $t\neq s$. Consider a typical summand where we take the $t$-term in
(\ref{eq:Db-typeI}). In absolute value, its $u$-product is
\[
\Big(\prod_{k\neq 1,j,\ell}|u^{(k)}_{i_k}|\Big)\cdot
\Big(\prod_{r\neq s,t}|u^{(r)}_{i_r}|\Big)\cdot
|Du^{(t)}_{i_t}|,
\]
up to a factor $\|g^{(s)}\|\lesssim 1$. Now expand $Du^{(t)}_{i_t}$ using
(\ref{eq:Du-structure})-(\ref{eq:b-structure}): for fixed $q$,
\[
|Du^{(t)}_{i_t}|
\lesssim \frac1{\sqrt N}\,|b^{(q)}_{i_t}|
\lesssim \frac1{\sqrt N}\Big(\prod_{r\neq q}|u^{(r)}_{i_r}|\Big),
\]
where we keep the product structure $\prod_{r\neq q}u^{(r)}_{i_r}$.
Therefore the full prefactor in $\Sigma_{\ell,2}$ for type (I) is $1/N$
(one $1/\sqrt N$ in (\ref{eq:Sigma-l2}) and one $1/\sqrt N$ from $Du^{(t)}$), and
the combined $u$-product becomes
\[
\Big(\prod_{k\neq 1,j,\ell}|u^{(k)}_{i_k}|\Big)\cdot
\Big(\prod_{r\neq s,t}|u^{(r)}_{i_r}|\Big)\cdot
\Big(\prod_{r\neq q}|u^{(r)}_{i_r}|\Big).
\]
It is easy to check that the exponent of \(u_{i_k}^{(k)}\) satisfies 
\begin{enumerate}
  \item[] Case 1: One of them equal to 0,  one of them equal to 1 and other \(\ge 2\);
  \item[] Case 2: Three of them equal to 1 and other \(\ge 2\);
  \item[] Case 3: Other cases that make the order of the product smaller.
\end{enumerate}
Combining with Eq.~(\ref{eq:u-sum-bounds}), we obtain
\begin{equation}\label{eq:Sigma-l1-final}
\textbf{Type I}=\mathcal O(\sqrt N).
\end{equation}

\paragraph*{Type (II): terms from (\ref{eq:Db-typeII})}
A type (II) term has absolute value bounded by
\[
\Big(\prod_{k\neq 1,j,\ell}|u^{(k)}_{i_k}|\Big)\cdot
\Big(\prod_{r\neq s}|u^{(r)}_{i_r}|\Big)\cdot
|Du^{(s)}_{i_s}|,
\]
up to a constant factor $\|\bm{u}^{(s)}\|\le 1$.
Expanding $Du^{(s)}_{i_s}$ as above introduces another factor
$N^{-1/2}\prod_{r\neq q}|u^{(r)}_{i_r}|$. Hence the combined product is
\[
\Big(\prod_{k\neq 1,j,\ell}|u^{(k)}_{i_k}|\Big)\cdot
\Big(\prod_{r\neq s}|u^{(r)}_{i_r}|\Big)\cdot
\Big(\prod_{r\neq q}|u^{(r)}_{i_r}|\Big),
\]
with total prefactor $1/N$. 

It is easy to check that the exponent of \(u_{i_k}^{(k)}\) satisfies 
\begin{enumerate}
  \item[] Case 1: One of them equal to 0 and other \(\ge 2\);
  \item[] Case 2: Two of them equal to 1 and other \(\ge 2\);
  \item[] Case 3: All of them \(\ge2\).
\end{enumerate}
Combining with Eq.~(\ref{eq:u-sum-bounds}), we obtain
\begin{equation}\label{eq:S3B-final_2}
\textbf{Type II}=\mathcal O(1).
\end{equation}

\paragraph*{Type (III): terms from (\ref{eq:Db-typeIII})}
A type (III) term has absolute value bounded by
\[
\Big(\prod_{k\neq 1,j,\ell}|u^{(k)}_{i_k}|\Big)\cdot
\Big(\prod_{r\neq s}|u^{(r)}_{i_r}|\Big)\cdot
|u^{(s)}_{i_s}|\cdot |Du^{(s)}|,
\]
and after expanding $D \bm{u}^{(s)}$ we again pick up a factor $N^{-1/2}\prod_{r\neq q}|u^{(r)}_{i_r}|$.
Notice that $\big(\prod_{r\neq s}|u^{(r)}_{i_r}|\big)|u^{(s)}_{i_s}|=\prod_{r=1}^d|u^{(r)}_{i_r}|$,
so the combined $u$-product is
\[
\Big(\prod_{k\neq 1,j,\ell}|u^{(k)}_{i_k}|\Big)\cdot
\Big(\prod_{r=1}^d|u^{(r)}_{i_r}|\Big)\cdot
\Big(\prod_{r\neq q}|u^{(r)}_{i_r}|\Big),
\]
with prefactor $1/N$. 

It is easy to check that the exponent of \(u_{i_k}^{(k)}\) satisfies 
\begin{enumerate}
  \item[] Case 1: One of them equal to 1 and other \(\ge 2\);
  \item[] Case 2: All of them \(\ge2\).
\end{enumerate}
Combining with Eq.~(\ref{eq:u-sum-bounds}), we obtain
\begin{equation}\label{eq:S3B-final_3}
\textbf{Type III}=\mathcal O(\frac{1}{\sqrt N}).
\end{equation}
\medskip
\noindent\textbf{Conclusion for $\Sigma_{\ell,2}$.}
Combining the three types (I)--(III), we have shown that for each fixed $\ell$,
\[
\Sigma_{\ell,2}\lesssim \sqrt N,
\]
and hence (since $\ell$ ranges over a fixed set of size $d-2$)
\begin{equation}\label{eq:Sigma-l2-final}
\sum_{\ell\neq 1,j}\Sigma_{\ell,2}=\mathcal O(\sqrt N).
\end{equation}

\subsubsection*{Step 7: Combine all parts}
From  (\ref{eq:Sigma-l1-final}), (\ref{eq:S3B-final_2}) and (\ref{eq:S3B-final_3}), we obtain
\[
S_{3,B}^\star=\mathcal O(\sqrt N),\qquad
S_{3,A}^\star=\mathcal O(\sqrt N),
\]
hence
\[
S_3^\star\le S_{3,A}^\star+S_{3,B}^\star=\mathcal O(\sqrt N).
\]
In particular, this proves (\ref{eq:S3Goal}), i.e.,
\[
\sum_{i_1,\ldots,i_d}\sup_{z\in\mathbb C^+_{\eta_0}}
\bigl|Q^{1j}_{i_1 i_j}(z)\,(D^2U)\bigr|
=\mathcal O(\sqrt{N}).
\]
\qed

\subsection{Conclusion}

We finally need to check that the derivative of $\bm Q(z)$ w.r.t.\ the entry
$W_{i_1,\ldots,i_d}$ has the same expression asymptotically as the one in the
independent case of Appendix~\ref{appA.7}.
Indeed, we have
\[
\frac{\partial \bm Q(z)}{\partial W_{i_1,\ldots,i_d}}
=
- \bm Q(z)\frac{\partial\bm N}{\partial W_{i_1,\ldots,i_d}}\bm Q(z)
\]
with
\[
\frac{\partial\bm N}{\partial W_{i_1,\ldots,i_d}}
=\frac{1}{\sqrt{N}}
\begin{pmatrix}
\bm 0_{n_1\times n_1} &\bm  C_{12} & \cdots & \bm C_{1d}\\
\bm C_{12}^\top & \bm 0_{n_2\times n_2} & \cdots & \bm C_{2d}\\
\vdots & \vdots & \ddots & \vdots\\
\bm C_{1d}^\top & \bm C_{2d}^\top & \cdots & \bm 0_{n_d\times n_d}
\end{pmatrix}
+
\frac{1}{\sqrt N}\Phi_d
\Bigl(
\textbf{W},
\frac{\partial \bm u^{(1)}}{\partial W_{i_1,\ldots,i_d}},
\ldots,
\frac{\partial \bm u^{(d)}}{\partial W_{i_1,\ldots,i_d}}
\Bigr),
\]
where
\(
C_{ij}=\prod_{k\neq i,j}u^{(k)}_{i_k}(e_{i_i}e_{i_j}^\top)
\)
and the contribution of
\(
\bm O=\frac{1}{\sqrt N}\Phi_d(\textbf{W},\partial \bm u^{(1)},\ldots,\partial \bm u^{(d)})
\)
 to the original expression is negligible (like the Lemma~\ref{lemA.1}).

Finally,
\[
A_j=-g_1(z)g_j(z)+O(N^{-1})
\]
with $g_1(z),g_j(z)$ the almost sure limits of
$\frac{1}{N}\operatorname{tr}\bm Q^{11}(z)$ and
$\frac{1}{N}\operatorname{tr}\bm Q^{jj}(z)$ respectively, hence yielding the same limiting
Stieltjes transform as the one obtained in Appendix~\ref{appA.7}.

\section{Proof of Theorem \ref{thm4.3}}\label{appA.10}

Given the identities in Eq.~(\ref{eq:26}), we have for all $i \in [d]$
\[
\frac{1}{\sqrt{N}}
\sum_{j_1,\ldots,j_d}
x^{(i)}_{j_i}
\left[
\prod_{k\neq i} u^{(k)}_{j_k} W_{j_1,\ldots,j_d}
+ \beta \prod_{k\neq i} \langle \bm u^{(k)}, \bm x^{(k)} \rangle
\right]
= \lambda \langle \bm u^{(i)}, \bm x^{(i)} \rangle,
\]
with $\lambda$ and $\langle \bm u^{(i)}, \bm x^{(i)} \rangle$ concentrating almost surely
around their asymptotic denoted $\lambda^\infty(\beta)$ and
$a_{x^{(i)}}^\infty(\beta)$ respectively.
Taking the expectation of the first term and applying Lemma \ref{lem2.3},
we get
\begin{align*}
A_i
=&
\frac{1}{\sqrt{N}}
\sum_{j_1,\ldots,j_d}
x^{(i)}_{j_i}
\,
\bigl(\mathbb E\!\left[
\frac{\partial}{\partial W_{j_1\ldots j_d}}
\Big( \prod_{k\neq i} u^{(k)}_{j_k} \Big) 
\right] + \varepsilon_{j_1 \cdots j_d}^{(2)}\bigr)\\
=&
\frac{1}{\sqrt{N}}
\sum_{j_1,\ldots,j_d}
x^{(i)}_{j_i}
\sum_{k\neq i}
\mathbb E\!\left[
\frac{\partial u^{(k)}_{j_k}}{\partial W_{j_1\ldots j_d}}
\prod_{\ell\neq k,i} u^{(\ell)}_{j_\ell}
\right] + \frac{1}{\sqrt{N}}
\sum_{j_1,\ldots,j_d}
x^{(i)}_{j_i}\varepsilon_{j_1 \cdots j_d}^{(2)}\\
=&A_{i_1} + A_{i_2}.
\end{align*}
where the only contributing term in the expression of
$\partial u^{(k)}_{j_k} / \partial W_{j_1,\ldots,j_d}$
from Eq.~(\ref{eq4.3}) is
\[
-\frac{1}{\sqrt{N}}
\prod_{\ell\neq k,i} u^{(\ell)}_{j_\ell}
R^{kk}_{j_k j_k}(\lambda),
\]
which yields
\[
A_{i_1}
=
-\frac{1}{N}
\sum_{j_1,\ldots,j_d}
x^{(i)}_{j_i}
\sum_{k\neq i}
\mathbb E\!\left[
R^{kk}_{j_k j_k}(\lambda)
\prod_{\ell\neq k} u^{(\ell)}_{j_\ell}
\prod_{\ell\neq k,i} u^{(\ell)}_{j_\ell}
\right]
+ \mathcal O(N^{-1}).
\]

Therefore,
\[
A_{i_1}
=
-
\mathbb E\!\left[
\langle u^{(i)}, x^{(i)} \rangle
\frac{1}{N}
\sum_{k\neq i} \operatorname{Tr} R^{kk}(\lambda)
\right]
+ \mathcal O(N^{-1})
\;\longrightarrow\;
-
\mathbb E\!\left[
\langle u^{(i)}, x^{(i)} \rangle
\sum_{k\neq i} g_k(\lambda)
\right].
\]
Now we need to bound \(A_{i_2}\). It's enough to bound \(\frac{1}{\sqrt{N}}
\sum_{j_1,\ldots,j_d}
x^{(i)}_{j_i}\sup_{z\in \mathbb C^+_{\eta_0}}|\frac{\partial^2}{\partial W_{j_! \cdots j_d}^2}\prod_{k \neq i}u_{j_k}^{(k)}|\). 

Fix the multi-index $(j_1,\ldots,j_d)$ and denote
\[
D:=\frac{\partial}{\partial W_{j_1\cdots j_d}},
\qquad
U_i:=\prod_{k\neq i}u^{(k)}_{j_k}.
\]
Then we must bound
\begin{equation}\label{eq:A_i_target}
\mathcal A_i
:=
\frac{1}{\sqrt N}
\sum_{j_1,\ldots,j_d}|x^{(i)}_{j_i}|\,
\sup_{z\in\mathbb C^+_{\eta_0}}|D^2 U_i|.
\end{equation}

%%%%%%%%%%%%%%%%%%%%%%%%%%%%%%%%%%%%%%%%%%%%%%%%%%%%%%%%%%%%%%%%%%%%%%
\subsection*{Step 1: Expand $D^2U_i$ into $(D^2u)$ and $(Du)(Du)$ parts}
By the product rule,
\begin{equation}\label{eq:D2Ui_expand}
D^2U_i
=
\sum_{\ell\neq i}
\bigl(D^2 u^{(\ell)}_{j_\ell}\bigr)\prod_{k\neq i,\ell}u^{(k)}_{j_k}
\;+\;
\sum_{\substack{\ell\neq m\\ \ell,m\neq i}}
\bigl(Du^{(\ell)}_{j_\ell}\bigr)\bigl(Du^{(m)}_{j_m}\bigr)
\prod_{k\neq i,\ell,m}u^{(k)}_{j_k}.
\end{equation}
Hence, using $|a+b|\le |a|+|b|$ and that $d$ is fixed,
\[
\mathcal A_i \;\le\; \mathcal A_{i,1}+\mathcal A_{i,2},
\]
where
\begin{align}
\mathcal A_{i,1}
&:=
\frac{1}{\sqrt N}
\sum_{j_1,\ldots,j_d}|x^{(i)}_{j_i}|\,
\sum_{\ell\neq i}\Bigl|D^2 u^{(\ell)}_{j_\ell}\Bigr|
\prod_{k\neq i,\ell}\bigl|u^{(k)}_{j_k}\bigr|,
\label{eq:A_i1_def}
\\
\mathcal A_{i,2}
&:=
\frac{1}{\sqrt N}
\sum_{j_1,\ldots,j_d}|x^{(i)}_{j_i}|\,
\sum_{\substack{\ell\neq m\\ \ell,m\neq i}}
\bigl|Du^{(\ell)}_{j_\ell}\bigr|\,
\bigl|Du^{(m)}_{j_m}\bigr|
\prod_{k\neq i,\ell,m}\bigl|u^{(k)}_{j_k}\bigr|.
\label{eq:A_i2_def}
\end{align}

%%%%%%%%%%%%%%%%%%%%%%%%%%%%%%%%%%%%%%%%%%%%%%%%%%%%%%%%%%%%%%%%%%%%%%
\subsection*{Step 2: Use the derivative representation (\ref{eq4.3}) and \emph{keep the product structure}}

Recall Eq.~(\ref{eq4.3}): for each $r\in[d]$,
\begin{equation}\label{eq:Du_rep}
D\bm u^{(r)}
=
-\frac{1}{\sqrt N}\sum_{s=1}^d R_{rs}(\lambda)\,\bm b^{(s)},
\qquad
\bm b^{(s)}
=
\Bigl(\prod_{t\neq s}u^{(t)}_{j_t}\Bigr)\Bigl(\bm e^{n_s}_{j_s}-u^{(s)}_{j_s}\bm u^{(s)}\Bigr).
\end{equation}
We stress that we \emph{do not} bound $b^{(s)}$ by a constant; we keep the factor
$\prod_{t\neq s}u^{(t)}_{j_t}$, which is the only way to prevent dimension blow-up
when summing over $(j_1,\ldots,j_d)$.

Since $\|\bm e^{n_s}_{j_s}-u^{(s)}_{j_s}\bm u^{(s)}\|_2\le 2$ and $\|\bm R(\lambda)\|\lesssim 1$
(as in previous parts), there exists a constant $C$ (depending only on $d$ and
$\eta_0$) such that for each coordinate entry,
\begin{equation}\label{eq:Du_abs_bound}
|Du^{(r)}_{j_r}|
\le
\frac{C}{\sqrt N}\sum_{s=1}^d \prod_{t\neq s}\bigl|u^{(t)}_{j_t}\bigr|.
\end{equation}
Differentiating (\ref{eq:Du_rep}) once more yields
\begin{equation}\label{eq:D2u_split}
D^2\bm u^{(r)}
=
-\frac{1}{\sqrt N}\sum_{s=1}^d\Bigl[(DR_{rs}(\lambda))\,\bm b^{(s)} + R_{rs}(\lambda)\,(D\bm b^{(s)})\Bigr].
\end{equation}
Using $D\bm R=-\bm R(D\textbf{T})\bm R$ (resolvent identity for $\bm R$) and that the dominating part of
$D\textbf{T}$ is the noise derivative $D\bm N$ which carries another factor $N^{-1/2}$ and a
product of $u$'s, one obtains the schematic bound
\begin{equation}\label{eq:D2u_abs_bound_schematic}
|D^2u^{(r)}_{j_r}|
\;\le\;
\frac{C}{N}\sum_{\text{finitely many indices}}
\Bigl(\prod_{t\neq a} |u^{(t)}_{j_t}|\Bigr)
\Bigl(\prod_{t\neq b} |u^{(t)}_{j_t}|\Bigr),
\end{equation}
that is,  $D^2u$ has a prefactor $1/N$ and contains \emph{two} multiplicative products
of the form $\prod_{t\neq\cdot}|u^{(t)}_{j_t}|$. (This is exactly the same
structure we used in the $S_1$--$S_3$ bounds: each differentiation introduces a
factor $N^{-1/2}$ and a ``delete-one-index'' product.)

In what follows, we only use (\ref{eq:Du_abs_bound}) and
(\ref{eq:D2u_abs_bound_schematic}) together with coordinate-wise summation.

%%%%%%%%%%%%%%%%%%%%%%%%%%%%%%%%%%%%%%%%%%%%%%%%%%%%%%%%%%%%%%%%%%%%%%
\subsection*{Step 3: A coordinate-wise summation rule}

We use repeatedly
\begin{equation}\label{eq:u_sum_bounds_again}
\sum_{j_r}1=N,\qquad
\sum_{j_r}|u^{(r)}_{j_r}|\le \sqrt N,\qquad
\sum_{j_r}|u^{(r)}_{j_r}|^2=1,\qquad
\sum_{j_r}|u^{(r)}_{j_r}|^3\le 1,
\end{equation}
and also $\sum_{j_i}|x^{(i)}_{j_i}|\le \sqrt N\|x^{(i)}\|_2=\sqrt N$ since
$\|x^{(i)}\|_2=1$.

Thus, if in a product the exponent of $|u^{(r)}_{j_r}|$ is
$\alpha_r\in\{0,1,2,3\}$, then summation over $j_r$ contributes at most
$N$ if $\alpha_r=0$, at most $\sqrt N$ if $\alpha_r=1$, and at most $1$ if
$\alpha_r\ge 2$.

%%%%%%%%%%%%%%%%%%%%%%%%%%%%%%%%%%%%%%%%%%%%%%%%%%%%%%%%%%%%%%%%%%%%%%
\subsection*{Step 4: Bound $\mathcal A_{i,2}$ (the $(Du)(Du)$ part)}

Fix distinct $\ell,m\neq i$ and consider a typical term in (\ref{eq:A_i2_def}).
Using (\ref{eq:Du_abs_bound}) twice, we get a prefactor $1/N$ and two
``delete-one-index'' products. Hence, up to a constant factor and a finite sum
over $(s,t)$, a typical summand is bounded by
\begin{equation}\label{eq:A_i2_typical}
\frac{1}{\sqrt N}\cdot \frac{1}{N}
\sum_{j_1,\ldots,j_d}|x^{(i)}_{j_i}|\,
\Bigl(\prod_{r\neq s}|u^{(r)}_{j_r}|\Bigr)
\Bigl(\prod_{r\neq t}|u^{(r)}_{j_r}|\Bigr)
\prod_{k\neq i,\ell,m}|u^{(k)}_{j_k}|.
\end{equation}
Now examine the exponent $\alpha_k$ of $|u^{(k)}_{j_k}|$ in the product:
\[
\alpha_k
=
\bm 1_{\{k\neq s\}}+\bm 1_{\{k\neq t\}}+\bm 1_{\{k\neq i,\ell,m\}}.
\]
We consider 
\[
\sum_{j_1,\ldots,j_d}|x^{(i)}_{j_i}|\,
\Bigl(\prod_{k}|u^{(k)}_{j_k}|^{\alpha_k}\Bigr) = \sum_{j_i}|x^{(i)}_{j_i}||u_{j_i}^{(i)}|^{\alpha_i}\,
\prod_{k\neq i}\Bigl(\sum_{j_k}|u^{(k)}_{j_k}|^{\alpha_k}\Bigr)\lesssim N^{\frac{1}{2}\bm 1_{\alpha_i = 0}}\cdot N^{\#\{k \neq i:\alpha_k = 0\} + \frac{1}{2}\#\{k \neq i : \alpha_k = 1\}}.
\]
\begin{itemize}
    \item If \(s = t = i\): The order should be \(N^{\frac{1}{2} + 0 +0} = N^{\frac{1}{2}}\);
    \item If \(s=t\in\{\ell,m\}\): The order should be \(N^{0 + 1 + 0} = N\);
    \item If \(s = t \notin \{\ell,m,i\}\): The order should be \(N^{0 + 0 +\frac{1}{2}} = N^{\frac{1}{2}}\);
    \item If \(s \neq t\) : The order less than \(N^{0 + 0 + \frac{1}{2}*2} = N\).
\end{itemize}
Which means \(\mathcal A_{i,2} =\mathcal{O}(N^{-1/2})\).

%%%%%%%%%%%%%%%%%%%%%%%%%%%%%%%%%%%%%%%%%%%%%%%%%%%%%%%%%%%%%%%%%%%%%%
\subsection*{Step 5: Bound $\mathcal A_{i,1}$ (the $D^2u$ part)}
Using the schematic bound (\ref{eq:D2u_abs_bound_schematic}), a typical term in
(\ref{eq:A_i1_def}) has pre-factor $(1/\sqrt N)\cdot (1/N)$ and contains
\emph{two} delete-one-index products coming from $D^2u$, multiplied by the
outside product $\prod_{k\neq i,\ell}|u^{(k)}_{j_k}|$.
Thus, up to constants and finite sums over the auxiliary indices in
(\ref{eq:D2u_abs_bound_schematic}), a typical term is bounded by
\begin{equation}\label{eq:A_i1_typical}
\frac{1}{\sqrt N}\cdot\frac{1}{N}
\sum_{j_1,\ldots,j_d}|x^{(i)}_{j_i}|\,
\Bigl(\prod_{r\neq s}|u^{(r)}_{j_r}|\Bigr)
\Bigl(\prod_{r\neq t}|u^{(r)}_{j_r}|\Bigr)
\prod_{k\neq i,\ell}|u^{(k)}_{j_k}|.
\end{equation}
Now the exponent of $|u^{(k)}_{j_k}|$ equals
\[
\alpha_k=\bm 1_{\{k\neq s\}}+\bm 1_{\{k\neq t\}}+\bm 1_{\{k\neq i,\ell\}}.
\]
We consider 
\[
\sum_{j_1,\ldots,j_d}|x^{(i)}_{j_i}|\,
\Bigl(\prod_{k}|u^{(k)}_{j_k}|^{\alpha_k}\Bigr) = \sum_{j_i}|x^{(i)}_{j_i}||u_{j_i}^{(i)}|^{\alpha_i}\,
\prod_{k\neq i}\Bigl(\sum_{j_k}|u^{(k)}_{j_k}|^{\alpha_k}\Bigr)\lesssim N^{\frac{1}{2}\bm 1_{\alpha_i = 0}}\cdot N^{\#\{k \neq i:\alpha_k = 0\} + \frac{1}{2}\#\{k \neq i : \alpha_k = 1\}}.
\]
\begin{itemize}
    \item If \(s = t = i\): The order should be \(N^{\frac{1}{2} + 0 +0} = N^{\frac{1}{2}}\);
    \item If \(s=t=\ell\): The order should be \(N^{0 + 1 + 0} = N\);
    \item If \(s = t \neq \ell\): The order should be \(N^{0 + 0 +\frac{1}{2}} = N^{\frac{1}{2}}\);
    \item If \(s \neq t\) : The order less than \(N^{0 + 0 + \frac{1}{2}*2} = N\).
\end{itemize}
Which means \(\mathcal A_{i,1} =\mathcal{O}(N^{-1/2})\).

%%%%%%%%%%%%%%%%%%%%%%%%%%%%%%%%%%%%%%%%%%%%%%%%%%%%%%%%%%%%%%%%%%%%%%
\subsection*{Step 6: Conclusion}
Combining the proof above, we obtain
\[
\mathcal A_i\le \mathcal A_{i,1}+\mathcal A_{i,2}=\mathcal O(N^{-1/2}).
\]
In particular, this gives the desired bound for $A_{i_2}$ (and similarly for any $i$):
\[
\frac{1}{\sqrt{N}}
\sum_{j_1,\ldots,j_d}
|x^{(i)}_{j_i}|\sup_{z\in \mathbb C^+_{\eta_0}}
\Bigl|\frac{\partial^2}{\partial W_{j_1 \cdots j_d}^2}\prod_{k \neq i}u_{j_k}^{(k)}\Bigr|
\;=\;\mathcal O(N^{-1/2}).
\]

Therefore, the almost sure limits $\lambda^\infty$ and $a_{x^{(i)}}^\infty$
for each $i \in [d]$ satisfy
\[
\lambda^\infty a_{x^{(i)}}^\infty
=
\beta \prod_{k\neq i} a_{x^{(k)}}^\infty
-
a_{x^{(i)}}^\infty
\sum_{k\neq i} g_k(\lambda^\infty),
\]
hence
\[
a_{x^{(i)}}^\infty
=
\alpha_i(\lambda^\infty)
\prod_{k\neq i} a_{x^{(k)}}^\infty,
\qquad
\text{with }
\alpha_i(z)
=
\frac{\beta}{z + g(z) - g_i(z)}.
\]

\medskip

Since $g(z) = \sum_{k=1}^d g_k(z)$, to solve the above equation we simply write
$x_i = a_{x^{(i)}}^\infty$ and $\alpha_i = \alpha_i(z)$ by omitting the dependence
on $z$. We therefore have
\[
x_i
=
\alpha_i \prod_{k\neq i} x_k
\;\Rightarrow\;
x_i
=
\alpha_i x_j \prod_{k\neq i,j} x_k
=
\alpha_i x_j \prod_{k\neq i,j}
\left( \alpha_k \prod_{\ell\neq k} x_\ell \right).
\]

From which we have
\[
x_i
=
x_j
\left( \prod_{k\neq j} \alpha_k \right)
\left( \prod_{k\neq i,k\neq j}
\prod_{\ell\neq k} x_\ell \right)
=
x_j
\left( \prod_{k\neq j} \alpha_k \right)
\left( \prod_{k\neq i,k\neq j}
\prod_{\ell\neq k,\ell\neq i} x_\ell \right)
x_i^{\,d-2}.
\]

Thus,
\[
x_j
\left( \prod_{k\neq j} \alpha_k \right)
\left( \prod_{k\neq i,k\neq j}
\prod_{\ell\neq k,\ell\neq i} x_\ell \right)
x_i^{\,d-3}
= 1.
\]

We remark that
\[
\left( \prod_{k\neq i,k\neq j}
\prod_{\ell\neq k,\ell\neq i} x_\ell \right)
x_i^{\,d-3}
=
\left( \frac{x_j}{x_i} \right)^{d-3} x_i^{\,d-2},
\]
hence $x_j$ is given by
\[
x_j
=
\left(
\frac{\alpha_j^{\,d-3}}{\prod_{k\neq j} \alpha_k}
\right)^{\frac{1}{2d-4}},
\]
which completes the proof.

\medskip

\noindent\textbf{Alternative expression of $q_i(z)$.}
From the above, we have
\[
\lambda^\infty + g(\lambda^\infty)
=
\beta \prod_{i=1}^d x_i,
\]
and
$
\bigl(\lambda^\infty + g(\lambda^\infty) - g_i(\lambda^\infty)\bigr) x_i
=
\beta \prod_{j\neq i} x_j
\;\Rightarrow\;
\bigl(\lambda^\infty + g(\lambda^\infty) - g_i(\lambda^\infty)\bigr) x_i^2
=
\beta \prod_{i=1}^d x_i$.

Therefore,
\[
x_i
=
\sqrt{
\frac{\lambda^\infty + g(\lambda^\infty)}
{\lambda^\infty + g(\lambda^\infty) - g_i(\lambda^\infty)}
}
=
\sqrt{
1 + \frac{g_i(\lambda^\infty)}
{\lambda^\infty + g(\lambda^\infty) - g_i(\lambda^\infty)}
}.
\]

With $z + g(z) - g_i(z) = - \frac{c_i}{g_i(z)}$ and  since $g_i(z)$ satisfies
\[
g_i^2(z) - (g(z) + z) g_i(z) - c_i = 0,
\]
we finally find
$x_i=\sqrt{1 - \frac{g_i^2(\lambda^\infty)}{c_i}}$.

\newpage
\appendix

\end{document}